\documentclass[10pt]{amsart}

\usepackage{amsmath}
\usepackage{amsxtra}
\usepackage{amscd}
\usepackage{amsthm}
\usepackage{amsfonts}
\usepackage{amssymb}
\usepackage{mathtools}
\usepackage{eucal}
\usepackage[all,color]{xy}
\usepackage{graphicx}
\usepackage{tikz-cd}
\usepackage{mathrsfs}
\usepackage{color}
\usepackage{longtable}
\usepackage{caption}
\usepackage{enumerate}
\usepackage{faktor}
\usepackage{bbm}
\usepackage{stmaryrd}
\usepackage{rotating}
\usepackage[T1]{fontenc} 
\usepackage[utf8]{inputenc} 
\usepackage{tikz}
\usepackage{tikz-cd}
\usetikzlibrary{matrix,arrows,decorations.pathmorphing}
\usepackage{keytheorems} 
\usepackage{hyperref}
\hypersetup{
    colorlinks=true,
    linkcolor=blue,
    urlcolor=blue, 
    citecolor= blue} 
\usepackage{float}
\usepackage[noabbrev, capitalize, nameinlink]{cleveref}
\usepackage{url}
\usepackage{comment}

\newkeytheorem{theorem}[parent=section]
\newkeytheorem{corollary,lemma,proposition}[sibling=theorem]
\newkeytheorem{definition}[style=definition,sibling=theorem]
\newkeytheorem{remark,remarks}[style=remark,sibling=theorem]

\usepackage[color=blue!10]{todonotes}

\usepackage{xcolor}

\newcommand{\sm}[4]{\ensuremath{\big(\begin{smallmatrix}#1 & #2 \\ #3 & #4\end{smallmatrix}\big)}}
\newcommand{\ttmat}[4]{\left( \begin{array}{cc}
#1 & #2 \\
#3 & #4
\end{array}
\right)}

\newcommand{\B}{\mathcal{B}}

\newcommand{\C}{\mathbb{C}}

\newcommand{\F}{\mathbb{F}}
\newcommand{\cF}{\mathcal{F}}

\newcommand{\m}{\mathfrak{m}}
\newcommand{\n}{\mathfrak{n}}

\newcommand{\cO}{\mathcal{O}}

\newcommand{\Q}{\mathbb{Q}}
\newcommand{\T}{\mathbb{T}}

\newcommand{\Z}{\mathbb{Z}}

\newcommand{\Irr}{\mathrm{Irr}}

\newcommand{\St}{\mathrm{St}}

\newcommand{\onto}{\twoheadrightarrow}

\newcommand{\isoto}{\xrightarrow{\sim}}

\newcommand{\rhobar}{\overline{\rho}}

\DeclareMathOperator{\Aut}{Aut}
\DeclareMathOperator{\Br}{Br}
\DeclareMathOperator{\coker}{coker}
\DeclareMathOperator{\End}{End}

\DeclareMathOperator{\Frob}{Frob}

\DeclareMathOperator{\Gal}{Gal}

\DeclareMathOperator{\GL}{GL}
\DeclareMathOperator{\Hom}{Hom}

\DeclareMathOperator{\Ind}{Ind}

\DeclareMathOperator{\PGL}{PGL}
\DeclareMathOperator{\PSL}{PSL}

\DeclareMathOperator{\rank}{rk}

\DeclareMathOperator{\SL}{SL}
\DeclareMathOperator{\Spec}{Spec}
\DeclareMathOperator{\tr}{tr}
\DeclareMathOperator{\new}{new}

\newcommand{\dnew}{\text{-}\new}

\theoremstyle{definition}

\theoremstyle{remark}

\title[Level-raising via modular representation theory]{
Counting level-raising congruences using modular representation theory}

\author{Jaclyn Lang}
\address[Jaclyn Lang]{Department of Mathematics, Temple University}
\email{jaclyn.lang@temple.edu}

\author{Robert Pollack}
\address[Robert Pollack]{Department of Mathematics, University of Arizona}
\email{rpollack@arizona.edu}

\author{Preston Wake}
\address[Preston Wake]{Department of Mathematics, Michigan State University}
\email{wakepres@msu.edu}

\date{\today}

\begin{document}

\begin{abstract}
    We introduce a new method for studying mod-$\ell$ congruences between eigenforms through the modular representation theory of $\PGL_2(\F_p)$.  When $p\equiv \pm 1 \pmod{\ell}$, we use this theory to construct and describe extra structures on spaces of modular forms with $\Gamma_0(p^2)$-level at $p$ and a fixed mod-$\ell$ Galois representation.  The structural results we obtain can be viewed as a refinement of classical level-raising theorems since they not only allow us to prove the existence of congruences, but also to count the number of such congruences.  Our methods work equally well in the residually irreducible and residually reducible cases, allowing us to prove several new instances of congruences between Eisenstein series and cuspforms (as well as independently rederiving classical results of Mazur and more recent results of Lang--Wake).  Notably, our approach proves these results without computing constant terms of Eisenstein series, without Galois deformation theory and $R=\T$ theorems, and without the Jacquet--Langlands correspondence.
\end{abstract}
\maketitle

\section{Introduction}

Let $\ell$ be an odd prime number. In the 1970s, Serre made what is now known as the weak form of Serre's conjecture:\ every odd, irreducible, 2-dimensional mod-$\ell$ Galois representation $\rhobar$ arises from a cuspidal eigenform $f$ \cite[Section 3]{serre75}. A decade later, Serre formulated the strong form of the conjecture, specifying the optimal level \cite{serre-conj}. This raises a question:\ given such a $\rhobar$, what are all possible levels of $f$ that give rise to it?  Answering this question amounts to finding congruences modulo $\ell$ between eigenforms of different levels.  Such congruences can lower the level \cite{ribetepsilon}, a fact used to prove that Serre's weak conjecture implies his strong one.  They can also raise the level \cite{ribetICM,ribet}. The most general such results were by Diamond and Taylor \cite{DT1,DT2}, where they showed that there are essentially no global obstructions to finding a lift with prescribed local behavior at a given finite set of primes.

The situation for reducible $\rhobar$ is more complicated and less well understood. In the simplest case $\rhobar = \omega \oplus 1$ with $\omega$ the mod-$\ell$ cyclotomic character, blindly following Serre's formula suggests that the optimal $f$ would come from the zero-dimensional space of cusp forms of weight $2$ and level $1$. However, Mazur showed that $\rhobar$ does arise from a weight-2 cuspform of prime level $p \geq 5$ for all $p \equiv 1 \mod{\ell}$, and this is related to the fact that the factor $p-1$ appears in the constant term of the Eisenstein series of weight 2 and level $p$ \cite{mazur}. There are results, for instance of Billerey--Menares \cite{BM1,BM2} and Yoo \cite{yoo}, about the possible levels of a given reducible $\rhobar$, though these results tend to focus on squarefree levels. The proofs of such results typically rely on computing constant terms of Eisenstein series.

In this paper, we study level-raising congruences at a prime $p$ with $p \equiv \pm 1 \pmod{\ell}$, so that $\ell$ divides the order of the group $G=\PGL_2(\F_p)$. We use the mod-$\ell$ representation theory of $G$ to construct extra structures on spaces of modular forms of level $p^2$ that are congruent to a given eigenform $f$. These structures imply not only the existence of congruences, but formulae for the number of congruent forms of a given type (principal series, Steinberg, or supercuspidal). Remarkably, these formulae even apply in the residually reducible case and can be used to deduce the existence of Eisenstein congruences without computing constant terms. Even in the non-Eisenstein case, our results are novel, and our techniques do not rely on Galois-deformation theory or on the Jacquet--Langlands correspondence.

\subsection{The vexing case}
\label{subsec:intro vexing}
We begin by illustrating our results in a case when they are particularly clean to state, the vexing case of Diamond \cite{Diamond97}, 
which is completely disjoint from the Eisenstein case. Let $F$ be a large-enough finite extension of $\Q_\ell$, with valuation ring $\cO$ and residue field $\F$. A representation $\rhobar:G_\Q \to \GL_2(\F)$ is said to be \emph{vexing} at $p$ if $\rhobar|_{G_{\Q_p}}$ is irreducible but $\rhobar|_{I_p}$ is reducible, where $G_{\Q_p} \supset I_p$ are, respectively, a decomposition group and an inertia group at $p$. As Diamond shows, if $\rhobar$ is vexing, then $\rhobar|_{G_{\Q_p}} \cong \Ind_{G_{\Q_{p^2}}}^{G_{\Q_p}}(\overline{\psi})$, where $\Q_{p^2}$ is the unramified quadratic extension of $\Q_p$, $\psi:G_{\Q_{p^2}} \to \cO^\times$ is a character of conductor $p$ with finite, prime-to-$\ell$ order, and $\bar\psi=\psi \otimes_\cO \F$. The minimal lift of $\rhobar|_{G_{\Q_p}}$ is $\Ind_{G_{\Q_{p^2}}}^{G_{\Q_p}}({\psi})$, but, if $p\equiv -1 \pmod{\ell}$, other lifts may be obtained by multiplying $\psi$ by characters of $\ell$-power order.

Our result can be thought of as counting modular lifts of $\rhobar$. Let $\Delta$ be the $\ell$-Sylow subgroup of $\F_{p^2}^\times$, and identify it with a quotient of $G_{\Q_{p^2}}$ via the Artin map. Let $M_k(\Gamma,\cO)$ denote the space of modular forms of weight $k \ge 2$ and prime-to-$p\ell$ level $\Gamma$ with coefficients in $\cO$. Assume that $\rhobar$ is modular, in the sense that there is an eigenform $f \in M_k(\Gamma \cap \Gamma_0(p^2),\cO)$ such that $\rhobar_f\cong \rhobar$, and let $\m$ be the corresponding maximal ideal in the Hecke algebra. If $\xi: \Delta \to \cO^\times$ is a character, we say that an eigenform $f \in M_k(\Gamma \cap \Gamma_0(p^2),\cO)_\m$ is a \emph{lift of $\rhobar$ of cuspidal type $\xi$} if $\rho_f|_{G_{\Q_p}} \cong\Ind_{G_{\Q_{p^2}}}^{G_{\Q_p}}(\psi\xi).$ The lifting results of Diamond--Taylor \cite{DT2} imply that a lift of $\rhobar$ of cuspidal type $\xi$ exists for every possible $\xi$, but they say nothing about the number of lifts of a given cuspidal type. We show that the number of lifts is independent of $\xi$.

Our counting result follows from a module structure. We define an action of $\cO[\Delta]$ on $M_k(\Gamma\cap \Gamma_0(p^2),\cO)_\m$ determined by the following property:\ for $f \in M_k(\Gamma\cap \Gamma_0(p^2),\cO)_\m$ an eigenform of cuspidal type $\xi$, an element $\delta\in \Delta$ acts by $\delta f = \xi(\delta)f$. In particular, the $F$-dimension of the $\xi$-isotypic component of $M_k(\Gamma \cap \Gamma_0(p^2),\cO)_\m \otimes_\cO F$ is the number of normalized eigenforms of cuspidal type $\xi$.

\begin{theorem}
\label{thm:intro vexing}
Suppose that $\ell > 3$ and $p \equiv -1 \pmod{\ell}$. Let $f \in M_k(\Gamma \cap \Gamma_0(p^2),\cO)$ be an eigenform with prime-to-$p\ell$ level $\Gamma$, and let $\m$ be the associated maximal ideal in the $\ell$-adic Hecke algebra. Assume that $\rhobar_f$ is vexing at $p$. Then $M_k(\Gamma \cap \Gamma_0(p^2),\cO)_\m$ is a free $\cO[\Delta]$-module with the action defined above.
\end{theorem}

In particular, the number of eigenforms in $M_k(\Gamma \cap \Gamma_0(p^2),\cO)_\m$ of cuspidal type $\xi$ equals the $\cO[\Delta]$-rank of $M_k(\Gamma \cap \Gamma_0(p^2),\cO)_\m$ and hence is independent of $\xi$.

\subsection{Projectivity of modular forms of full level \texorpdfstring{$p$}{}} Our main theorem generalizes  \cref{thm:intro vexing} to also include $p \equiv 1 \pmod{\ell}$ and other possible residual local representations $\rhobar|_{G_{\Q_p}}$. The global input that underlies our proof is a generalization of a result of Serre \cite{serrecoursenotes} about the projectivity of the $\cO$-module of modular forms with full level at $p$. 

Let $k\ge 2$, and let $\Gamma$ be a prime-to-$p\ell$ level subgroup. The space $M_k(\Gamma \cap \Gamma(p),\cO)$ of modular forms of weight $k$ and level $\Gamma \cap \Gamma(p)$ has an action of $\GL_2(\F_p)$ that commutes with the action of the anemic Hecke algebra\footnote{Here, and throughout the paper, \emph{anemic} means generated by Hecke operators $T_n$ with $p \nmid n$.}. If $\ell \ge 5$ or if $\ell=3$ and $\Gamma$ is rigid, then we show that $M_k(\Gamma \cap \Gamma(p),\cO)$ is a projective $\cO[\GL_2(\F_p)]$-module\footnote{Actually, one has to replace $\GL_2(\F_p)$ by $\GL_2(\F_p)/\{\pm 1\}$ if $-1 \in \Gamma$, but we ignore this in favor of notational simplicity here.}. Serre proved this in the case $\Gamma=\SL_2(\Z)$. He also showed that, if $\ell < 5$, then $M_k(\Gamma(p),\cO)$ may not be projective, and this failure of projectivity is related to the failure of the character part of Serre's conjecture for $\ell <5$. 

Since the anemic Hecke action and the $\cO[\GL_2(\F_p)]$-action on $M_k(\Gamma \cap \Gamma(p),\cO)$ commute, if $\m'$ is a maximal ideal in the anemic Hecke algebra\footnote{Our notational convention is that $\m'$ refers to a maximal ideal in the anemic Hecke algebra, whereas $\m$ refers to a corresponding ideal in the Hecke algebra that includes $U_p$.}, then the localization $M_k(\Gamma \cap \Gamma(p),\cO)_{\m'}$ is also a projective $\cO[\GL_2(\F_p)]$-module. The more usual spaces of modular forms appear as fixed vectors under subgroups of $\GL_2(\F_p)$. First, let $Z \subset \GL_2(\F_p)$ be the center, and let $G=\PGL_2(\F_p)$, so that $M:=M_k(\Gamma \cap \Gamma(p),\cO)_{\m'}^Z$ is a projective $\cO[G]$-module.
If $T \subset B \subset G$ are, respectively, the diagonal torus and the upper-triangular Borel in $G$, then there are isomorphisms
\begin{align*}
    & M^{{T}} \cong M_k(\Gamma \cap \Gamma_0(p^2),\cO)_{\m'} \\
    & M^{{B}} = M_k(\Gamma \cap \Gamma_0(p),\cO)_{\m'} \\
    & M^{G} = M_k(\Gamma,\cO)_{\m'}.
\end{align*}
For level-raising questions, we are interested in understanding how these three spaces relate to one another.
By Frobenius reciprocity, we can think of each of the fixed spaces on the left-hand side as a hom-space
\[
\Hom_{\cO[G]}(\Ind_H^{G} \cO,M)
\]
for $H \in \{T,B,G\}$. We are thus led to the following purely algebraic question: if $M$ is a projective $\cO[G]$-module, how are the three spaces
\[
\Hom_{\cO[G]}(\Ind_H^{G} \cO,M)
\]
for $H \in \{T,B,G\}$ related?

\subsection{Modular representation theory of \texorpdfstring{$G=\PGL_2(\F_p)$}{}} The study of projective $\cO[G]$-modules belongs to modular representation theory. Since we assume that $\ell$ divides $|G|$, the algebra $\F[G]$ is not semisimple, but it still decomposes into a product of local algebras, called \emph{block algebras}, given by the primitive central idempotents. The fact that these block algebras need not be simple is related to the fact that the reduction of a lattice in an irreducible $F[G]$-module may not be simple. Different irreducible $F[G]$-modules whose reductions share simple factors are said to be in the same block, and the simple $\F[G]$-modules that appear in this way are also said to belong to the block. The fact that these {\it a priori} different notions of a `block' all amount to the same thing is a basic theorem of modular representation theory. The structure of a block is strongly influenced by a (conjugacy class of an) $\ell$-subgroup of $G$ called the \emph{defect group} of the block.  In our case, the defect group is always an $\ell$-Sylow subgroup $\Delta$ of $G$, which is isomorphic to the $\ell$-part of $\F_{p^2}^\times$. 

Block theory is a powerful tool for determining the structure of projective $\cO[G]$-modules. The projective indecomposable modules of a block algebra are in bijection with the simple $\F[G]$-modules in the block. A remarkable fact is that the projective $\cO[G]$-module $M=M_k(\Gamma \cap \Gamma(p),\cO)_{{\m'}}^Z$ of interest in level raising is already a module for a unique block algebra of $\cO[G]$. (That is, there is a unique block idempotent that acts as the identity on $M$; see \cref{prop:block of m}.) This block is determined by the $p$-component of the automorphic representation associated with any eigenform in $M$. In particular, $M$ is a direct sum of copies of the projective indecomposable modules of the block, of which there are at most two. There are two different types of blocks of interest.

\subsubsection{Nilpotent blocks} The simplest blocks we consider are the \emph{nilpotent blocks} of Brou\'e--Puig \cite{BP1980}. These blocks contain a single simple $\F[G]$-module and hence a single projective indecomposable module $P$.  Moreover, $\End_{\cO[G]}(P) = \cO[\Delta]$; that is, the block algebra is Morita equivalent to $\cO[\Delta]$. Nilpotent blocks appear in two situations:
\begin{itemize}
    \item the vexing case of \cref{subsec:intro vexing}, where $p\equiv -1 \pmod{\ell}$ and $\rhobar|_{G_{\Q_p}}$ is an irreducible induction from the unramified quadratic extension of $\Q_p$. This occurs when the $p$-component of an automorphic representation attached to ${\m'}$ is depth-zero supercuspidal for a character that is not quadratic modulo~$\ell$.
    \item when $p \equiv 1 \pmod{\ell}$ and $\rhobar|_{G_{\Q_p}}$ is a direct sum of two characters of conductor $p$. This occurs when the $p$-component of an automorphic representation attached to ${\m'}$ is a principal series for a pair of characters $\chi$ and $\chi^{-1}$, where $\chi$ has conductor $p$ and $\chi$ and $\chi^{-1}$ are not congruent modulo $\ell$.
\end{itemize}
The fact that these are nilpotent blocks is related to the fact that, in both cases, all lifts of $\rhobar|_{G_{\Q_p}}$ are of the same type (supercuspidal in the vexing case, and ramified principal series in the second case).

In both of these cases, if $e \in \cO[G]$ denotes the relevant block idempotent, we show that $e\Ind_T^G\cO$ is a projective indecomposable $e\cO[G]$-module and hence must be isomorphic to the unique projective indecomposable $P$.  Moreover, $e\Ind_B^G\cO=0$. Since every projective $e\cO[G]$-module $M$ is of the form $M \cong P^a$ for some integer $a$, we see from Frobenius reciprocity that
\[
M^T \cong \Hom_{\cO[T]}(\cO,M) \cong \Hom_{\cO[G]}(e\Ind_T^G\cO,M)\cong \End_{\cO[G]}(P)^a \cong \cO[\Delta]^a,
\]
and
\[
M^B \cong \Hom_{\cO[G]}(\Ind_B^G\cO,M) \cong \Hom_{\cO[G]}(e\Ind_B^G\cO,M)=0.
\]

\subsubsection{Principal blocks}
The other block of interest is the principal block, which is the block containing the trivial representation. This block contains two simple $\F[G]$-modules, the trivial representation and another one denoted $\pi$; hence it has two projective indecomposable modules, denoted $P_1$ and $P_\pi$. The principal block appears when $\rhobar$ is unramified at $p$, but the analysis again breaks into two cases, depending on the sign of $p$ modulo $\ell$.

We illustrate the structure of the principal block in the case $p \equiv -1 \pmod{\ell}$. Let $e \in \cO[G]$ be the idempotent of the principal block. For $p \equiv -1 \pmod{\ell}$, $\pi$ is a cuspidal representation, and we show that
\[
P_1=e\Ind_B^G\cO, \ P_\pi= \ker(e\Ind_T^G\cO \to e\Ind_B^G\cO).
\]
We compute that $\End_{\cO[G]}(P_\pi) \cong \cO[\Delta]^+$, where the $+$ denotes the subring that is fixed under the involution that inverts group-like elements.  Moreover, $\End_{\cO[G]}(P_1)$ is a rank-2 $\cO$-algebra, and $\Hom_{\cO[G]}(P_1,P_\pi) \cong \Hom_{\cO[G]}(P_\pi,P_1) \cong \cO$. We also get that $\Hom_{\cO[G]}(\cO,P_\pi)=0$ and $\Hom_{\cO[G]}(\cO,P_1) \cong \cO$. Since $P_1$ and $P_\pi$ are the only projective indecomposable modules in this block, $M=P_1^a \oplus P_\pi^b$ for some integers $a$ and $b$. Using Frobenius reciprocity and the above calculations of hom-spaces, these integers $a$ and $b$ can be recovered as
\[
a=\rank_\cO M^G, \ 2a+b = \rank_\cO M^B.
\]
The space $\Hom_{\cO[G]}(P_\pi,M)$ equals the kernel of trace from $M^T$ to $M^B$. 

\subsection{Counting level-raising congruences}
The module $M=M_k(\Gamma \cap \Gamma(p),\cO)_{{\m'}}^Z$ is a projective module over a block determined by the maximal ideal ${\m'}$, and we can use modular representation theory to describe its structure. When the block is nilpotent, this allows us to deduce \cref{thm:intro vexing} and a similar theorem in the ramified-principal-series case. The latter case may be more familiar, and is related, via twisting, to a freeness result over a ring of diamond operators for forms of level $\Gamma_1(p)$.

When ${\m'}$ is associated with a mod-$\ell$ Galois representation that is unramified at $p$, the associated block is principal. In this case, $M^T$ can be identified with $M_k(\Gamma \cap \Gamma_0(p^2),\cO)_{\m'}$, and we show that the kernel of the trace $M^T \to M^B$ can be identified with $M_k(\Gamma \cap \Gamma_0(p^2),\cO)_{\m}$, where $\m$ is the ideal containing $\m'$ and $U_p$. The result is the following (see \cref{thm:modular forms count}).

\begin{theorem}
    \label{thm:principal intro}
    Let $\ell > 3$ and $f \in M_k(\Gamma \cap \Gamma_0(p),\cO)$ be an eigenform with $\Gamma$ a prime-to-$p\ell$ level, and let $\m'$ be the associated maximal ideal in the anemic Hecke algebra.  Write $\m$ for the ideal generated by $\m'$ and $U_p$, and let
    \begin{align*}
        r_0&= \rank_\cO M_k(\Gamma,\cO)_{\m'}, \\
        r_1&=\rank_\cO M_k(\Gamma \cap \Gamma_0(p),\cO)_{\m'}^{p\dnew}.
    \end{align*}
    Then $M_k(\Gamma \cap \Gamma_0(p^2),\cO)_\m$ is an $\cO[\Delta]^+$-module in a natural way, and there is an isomorphism of $\cO[\Delta]^+$-modules
    \[
    M_k(\Gamma \cap \Gamma_0(p^2),\cO)_\m \cong \begin{cases}
        \cO[\Delta]^{r_0} \oplus (\cO[\Delta]^+)^{r_1} & \text{if }p \equiv 1 \pmod{\ell} \\
        \cO^{r_0} \oplus (\cO[\Delta]^+)^{r_1} & \text{if }p \equiv -1 \pmod{\ell}.
    \end{cases}
    \]
\end{theorem}

\begin{remarks} \hfill
\begin{enumerate}
    \item The localization $M_k(\Gamma \cap \Gamma_0(p^2),\cO)_\m$ can be identified with the kernel of $U_p$ on $M_k(\Gamma \cap \Gamma_0(p^2),\cO)_{\m'}$. The entire space $M_k(\Gamma \cap \Gamma_0(p^2),\cO)_{\m'}$ contains $p$-oldforms with greater multiplicity than the $p$-newforms, but taking this kernel ensures that $p$-oldforms appear with multiplicity one.
    \item Our definition of $M_k(\Gamma \cap \Gamma_0(p),\cO)_{\m'}^{p\dnew}$ is such that $r_1+2r_0=\rank M_k(\Gamma\cap\Gamma_0(p),\cO)_{\m'}$. When $\m'$ is Eisenstein, this may not agree with the number of forms that are Steinberg at $p$. For example, if $\Gamma=\SL_2(\Z)$, $p=11$, $\ell=5$, and $\m'$ is the Eisenstein maximal ideal, then $r_0=0$ and $M_2( \Gamma_0(11),\cO)_{\m'}$ has rank $2$, so $r_1=2$, but there is only one eigenform of $M_2( \Gamma_0(11),\cO)_{\m'}$ that is Steinberg at $11$.
    \item The action of $\cO[\Delta]^+$ on an eigenform $f$ can be described explicitly in terms of the $p$-component of the automorphic representation of $f$. It acts through the augmentation map on forms whose automorphic representation at $p$ is unramified or Steinberg. For forms whose automorphic representation at $p$ is ramified principal series or supercuspidal, it acts in a way determined by the character used to define that representation. This also explains why only $\cO[\Delta]^+$ acts and not the entire ring $\cO[\Delta]$:\ these representations are defined using not one character, but an (unordered) pair of a character and its inverse, and such a pair determines a homomorphism $\cO[\Delta]^+ \to \cO$.
    \item If $p \equiv 1 \pmod{\ell}$ and $r_1=0$ (in which case $p$ is called a \emph{Taylor--Wiles prime} for $\m$), then the theorem states that $M_k(\Gamma\cap \Gamma_0(p^2),\cO)_\m$ is a free $\cO[\Delta]$-module. This is exactly analogous to the $\cO[\Delta]$-freeness result of \cite[Theorem 2]{taylorwiles}, and the two theorems can be related to one another by twisting.
    \item If $p \equiv -1 \pmod{\ell}$, consider the saturation of the eigenforms in $M_k(\Gamma \cap \Gamma_0(p^2), \cO)_\m$ whose conductor (in the sense of \cite{Casselman}) is divisible by $p$.  The theorem implies that, after inverting $\ell$, this space is a free $F[\Delta]^+$-module, and one may ask whether the space is a free $\cO[\Delta]^+$-module.  In \cite{LPWSteinbergTorsion}, we give examples for which such an integral freeness statement fails.
    \item Although we state our results in terms of classical (coherent cohomology) modular forms, our methods can be used for other avatars of modular forms, such as modular symbols or \'etale cohomology of classical or quaternionic modular curves. The key input is the $\cO[G]$-projectivity of the space with full level at $p$, and techniques similar to the ones we use to prove projectivity for coherent cohomology are available for these other avatars (for instance, for \'etale cohomology, see \cite[Theorem 2.7]{rickard}). 
\end{enumerate}   
\end{remarks}

\subsection{Eisenstein congruences}
\label{sec:intro eis}
The proof of \cref{thm:principal intro} relies on the fact that $M=M_k(\Gamma \cap \Gamma(p),\cO)^Z$ is a projective $\cO[G]$-module in order to deduce that $M_{\m'}$ is a projective $\cO[G]$-module. In particular, \emph{no assumption} is made about $\m'$ in order to draw this conclusion, so it applies equally well whether $\m'$ is associated with an Eisenstein series or a cuspform.

Using \cref{thm:principal intro} in this way, we prove results about the existence and number of congruences between cuspforms and Eisenstein series in different settings, with general weights $k\ge 2$ and levels $N$ (so long as there is a prime $p$ dividing $N$ with $p \equiv \pm 1 \pmod{\ell}$). Many of these existence results are known, and more could likely be deduced by known methods, but we emphasize that our proofs of these results are quite different from previous ones. They do not require computations of constant terms of Eisenstein series, congruence numbers, or Galois cohomology, and they do not involve Galois deformations or $R=\T$ theorems.  Moreover, many of our existence theorems allow the square of a prime in the level, which is largely untreated in the literature with the exception of \cite{LangWake2} (which we recover by new methods) and \cite{Deo} (to which our results are complementary).  We also prove quantitative results about the number of such congruences at level $p^2$, which seem less accessible by other methods. We first illustrate this kind of result in two settings, related to \cite{mazur} and \cite{LangWake1,LangWake2}.

In the setting of Mazur's Eisenstein ideal \cite{mazur}, we prove the following (see \cref{cor:newform count} and \cref{subsec:aut triv Eis congs}). Let $S_k(\Gamma,\cO) \subseteq M_k(\Gamma,\cO)$ denote the submodule of cuspforms.

\begin{theorem}\label{thm:Eisp=1modell}
Let $p$ and $\ell$ be primes with $p \equiv 1 \pmod{\ell}$ and $\ell>3$.  Let $\m'$ denote the maximal ideal of the $\ell$-adic Hecke algebra of $S_2(\Gamma_0(p^2),\cO)$ that contains $T_q-q-1$ for all $q \nmid p$, and let $\m$ be the maximal ideal generated by $\m'$ and $U_p$.  Let $\Delta$ be the $\ell$-Sylow subgroup of $\F_{p^2}^\times$. Then
\begin{equation}
\label{eqn:eisen_p^2}    
\rank_\cO S_2(\Gamma_0(p^2),\cO)^{\new}_\m
=
\frac{|\Delta|-1}{2} \cdot \left(\rank_\cO S_2(\Gamma_0(p),\cO)_{\m'}-1 \right).
\end{equation}
\end{theorem}

\cref{thm:Eisp=1modell} implies Mazur's famous result that, with these hypotheses, there is a cuspform of level $\Gamma_0(p)$ congruent to an Eisenstein series.  Indeed, the left-hand side of \eqref{eqn:eisen_p^2} must be nonnegative, and thus $\rank_\cO S_2(\Gamma_0(p),\cO)_{\m'}$ must be positive.  But our theorem goes further.  When $\rank_\cO S_2(\Gamma_0(p),\cO)_{\m'}=1$, both sides of \eqref{eqn:eisen_p^2} are zero, so there can be no cuspidal newforms of level $\Gamma_0(p^2)$ that are congruent to the Eisenstein series.  When $\rank_\cO S_2(\Gamma_0(p),\cO)_{\m'}>1$, \cref{thm:Eisp=1modell} implies that there is a wealth of such newforms. The question of which pairs $(p,\ell)$ have $\rank_\cO S_2(\Gamma_0(p),\cO)_{\m'}>1$ has been the subject of much study \cite{Merel,CE2005,WWE,Lecouturier}, and the recent paper \cite{LMP} makes use of \cref{thm:Eisp=1modell} and the associated $\cO[\Delta]^+$-freeness of $M_2(\Gamma_0(p^2), \cO)_\m$ to give a new perspective on this question. 

With the same setup as \cref{thm:Eisp=1modell} except assuming that $p \equiv -1 \pmod{\ell}$, Mazur's results imply that $S_2(\Gamma_0(p),\cO)_{\m'}$ is zero. But, in earlier work \cite{LangWake1,LangWake2}, we have shown that $S_2(\Gamma_0(p^2),\cO)_\m$ is nonzero and have described its structure precisely. As a consequence of \cref{thm:principal intro}, we give a new proof of the following theorem, which summarizes the main results of \cite{LangWake1,LangWake2}.

\begin{theorem}
\label{thm:LangWake}
Let $p$ and $\ell$ be primes with $\ell > 3$ and $p \equiv -1 \pmod{\ell}$. Let $\T$ denote the Hecke algebra acting on $M_2(\Gamma_0(p^2),\cO)$, and let $\m \subseteq \T$ denote the maximal ideal which contains $T_q-q-1$ for all $q \nmid p$ and $U_p$. Let $\Delta$ be the $\ell$-Sylow subgroup of $\F_{p^2}^\times$. Then there is a canonical isomorphism of $\cO$-algebras $\T_\m \cong \cO[\Delta]^+$.
In particular,
$$
\rank_\cO S_2(\Gamma_0(p^2),\cO)_\m
=
\frac{|\Delta|-1}{2}.
$$
\end{theorem}

The proof in \cite{LangWake1} that $S_2(\Gamma_0(p^2),\cO)_\m \ne 0$ uses a computation of the residues at the cusps of an Eisenstein series of level $\Gamma_0(p^2)$.  This is somewhat complicated due to the abundance of cusps at that level compared to level $\Gamma_0(p)$. The proof of \cref{thm:LangWake} in \cite{LangWake2} uses this residue computation together with a calculation that a certain inertia-at-$p$ pseudodeformation ring is isomorphic to $\cO[\Delta]^+$ to deduce a modularity theorem that gives the desired structure on $\T_\m$.  The proof of \cref{thm:LangWake} given in this paper avoids such computations and is essentially by `pure thought'. 

As an example of the results we prove for more general weights and levels, we have the following (see \cref{cor:eis_gen}). In the statement, $\mu(x)$ denotes the M\"obius function and $\delta_{i,j}$ is the Kronecker delta.

\begin{theorem}
Let $p$ and $\ell$ be primes with $\ell > 3$ and $p^2 \equiv 1 \pmod{\ell}$.  Let $N$ be an integer with $\gcd(N,p\ell)=1$, and let $\Delta$ be the $\ell$-Sylow subgroup of $\F_{p^2}^\times$. Let $\T'$ be the Hecke algebra acting on $M_k(\Gamma_0(Np^2),\cO)$ generated by $T_q$ for all primes $q \nmid Np$, and let $\n' \subset \T'$ be the maximal ideal containing $T_q-q^{k-1}-1$ for all $q \nmid Np$.
Then
\begin{equation*}
\frac{2}{|\Delta|-1} \rank S_k(\Gamma_0(Np^2))^{\new}_{\n'}
= \begin{cases}
\rank S_k(\Gamma_0(Np))^{N\dnew}_{\n'}  -\mu(N) \delta_{k,2}  & \text{if }   p \equiv 1 \pmod{\ell}\\
\rank S_k(\Gamma_0(Np))^{\new}_{\n'}  + \mu(N) \delta_{k,2} & \text{if }  p \equiv -1 \pmod{\ell}.
\end{cases}
\end{equation*}
\end{theorem}
Since the rank is always nonnegative, one can deduce that $S_2(\Gamma_0(Np^2))^{\new}_{\n'}$, $S_2(\Gamma_0(Np))^{\new}_{\n'}$, or $S_2(\Gamma_0(Np))^{N\dnew}_{\n'}$ is nontrivial, depending on the signs of $\mu(N)$ and $p \pmod{\ell}$ (see \cref{cor:nonzero eisen spaces}).

\subsection{Galois representations and the local Langlands correspondence}
We began this introduction by discussing the level part of Serre's conjecture and vexing primes, which have to do with global Galois representations and their behavior at a prime $p$. These concepts are related, via the Langlands correspondence, to automorphic representations associated with modular forms and their local behavior at $p$. In fact, our methods work entirely on the automorphic side. We rely on the projectivity of a space of modular forms over the ring $\cO[G]$, and this ring can be interpreted as a subring of the local Hecke algebra at $p$. The rings $\cO[\Delta]$ and $\cO[\Delta]^+$ that appear in the statements of our main theorems can be interpreted as parameterizing possible inertial types at $p$.

To relate our results back to Galois representations, we give a Galois-theoretic interpretation of the rings $\cO[\Delta]$ and $\cO[\Delta]^+$. First, we identify $\Delta$ with a (tame) quotient of $I_p$ via class field theory. We define a Galois pseudodeformation ring $R_p$ that can be thought of as parameterizing the deformations of $\rhobar|_{I_p}$ that extend to $G_{\Q_p}$. Just as in \cite{LangWake2}, we show that such deformations are determined by an unordered pair $\{\chi,\chi^{-1}\}$ of deformations of the trivial character. When $\rhobar$ corresponds to a nilpotent block, there is a way to distinguish these characters and there is an isomorphism $R_p \cong \cO[\Delta]$. When $\rhobar$ corresponds to the principal block, there is not, and $R_p \cong \cO[\Delta]^+$. The Galois deformations associated with newforms define a map $R_p \to \T_\m$ to the Hecke algebra in the usual way.

This defines two {\it a priori} different actions of the ring $R_p$ on modular forms:\ one defined automorphically using the identification $R_p \cong \cO[\Delta]$ or $R_p \cong \cO[\Delta]^+$, and one defined Galois-theoretically using the map $R_p \to \T_\m$. One can think of this as giving two different labelings of eigenforms by $\Spec(R_p)$:\ one in terms of the local automorphic representation at $p$, and one in terms of the local Galois representation at $p$. The local Langlands correspondence implies that these two labelings, and hence these two actions, are the same. The fact that the correspondence of labelings of individual eigenforms glues to give an isomorphism of rings is a special case of the local Langlands correspondence in families of Emerton--Helm \cite{EmertonHelm}, which is now known for $\GL_n$ \cite{Helm2020, HelmMoss}. 

One can think of our results as a global application of the local Langlands correspondence in families. We have global input, coming from the geometry of modular curves, that implies a local property, namely the $\cO[G]$-projectivity of the space of modular forms of full level $p$. This local property allows us to make deductions about the structure of the space of modular forms as a module for the automorphically defined action of $R_p$. The local Langlands correspondence in families then implies that this action is the same as the Galois-theoretic action of $R_p$.

\subsection{Structure of paper}
In \cref{sec:rep thy review} we review the representation theory of $G = \PGL_2(\F_p)$ in both characteristic zero and characteristic $\ell \neq p$ for the purpose of establishing notation.  The main aim of \cref{sec:block algebras} is to prove an abstract version of our main structure theorem for projective $\cO[G]$-modules, namely \cref{thm:X^T}.  To do so, we require a description of the block algebras $e\cO[G]$, where $e \in Z(\cO[G])$ is a block idempotent.  In \cref{sec:projective} we verify, using results of Nakajima and Serre, that the space of modular forms with appropriate principal level structure at $p$ is a projective $\cO[G]$-module and hence can be described using \cref{thm:X^T}.  For the purpose of studying congruences between modular forms, one wishes to localize the space of modular forms at a maximal ideal $\m'$ of the anemic Hecke algebra.  We introduce Hecke operators and prove some relevant facts about them in \cref{sec:Hecke ops}.  Then in \cref{sec:hecke} we show that localizing at such an $\m'$ produces a space of modular forms that belongs to a single $\cO[G]$-block, and we describe how to find that block from an eigenform that gives rise to $\m'$.  This leads to the main structure theorems for spaces of modular forms in \cref{sec:main thms}, namely \cref{thm:modular forms count,thm:cuspforms count eis case}.  In \cref{sec:counting newforms} we use these theorems to give counts of newforms that satisfy various congruences.  
Finally, \cref{sec:Galois} explains the connection between our automorphic methods and local Langlands in families.

\subsection{Acknowledgments}
The authors thank the following people for helpful conversations: Patrick Allen, Lassina Dembélé, Shaunak Deo, Fred Diamond, Matthew Emerton, Benedict Gross, David Helm, Robert Kottwitz, Martin Lorenz, and Vytautas Paškūnas.  Their thoughts have improved this paper.  We used a suggestion of Claude to simplify some arguments in Section 7.4 by using the residue sequence.

J.L.~gratefully acknowledges support from the National Science Foundation (grant DMS-2301738) and the Simons Foundation (MP-TSM-00002260).

R.P.’s research has been partially supported by the National Science Foundation (grant DMS-2302285) and by the Simons Foundation (MPS-TSM-00002405).

P.W.~was supported by the National Science Foundation (grant DMS-2337830) and by the Simons Foundation International [SFI-MPS-SFM-00021789, PW].

\section{Representation theory of \texorpdfstring{$\PGL_2(\F_p)$}{}}\label{sec:rep thy review}
In this section we review the representation theory of $\PGL_2(\F_p)$ in characteristic zero and work out the blocks over an $\ell$-adic ring for $\ell \neq p$.

\subsection{Notation} We use the following notation throughout.
For an odd prime number $p$, let $G=\PGL_2(\F_p)$.  Let $\ell \ne p$ be an odd prime number that divides the order of $G$. We abuse notation and write elements of $G$ as matrices. We frequently make use of the following subgroups of $G$.
\begin{itemize}
    \item Let $B \coloneqq U \rtimes T$ be the Borel subgroup, with $T \coloneqq \{\sm{*}{0}{0}{*} \in G\}$ the split torus and $U \coloneqq \{\sm{1}{*}{0}{1} \in G\}$ the unipotent radical.
    \item Let $N \coloneqq T \rtimes W$ be the normalizer of $T$, where $W$ is the Weyl group of $T$, which is generated by $\sm{0}{1}{1}{0}$.
    \item Let $T'$ be a nonsplit torus, which is isomorphic to $\F_{p^2}^\times/\F_p^\times$ and can be given explicitly as $T'=\{\sm{a}{b\alpha}{b}{a} \in G\}$ for $\alpha \in \F_p$ a nonsquare.
    \item Let $N' \coloneqq T' \rtimes W'$ be the normalizer of $T'$, where $W'$ is the Weyl group of $T'$, which is generated by $\sm{-1}{0}{0}{1}$.
    \item Write $T=T_\ell \times T^{(\ell)}$ and $T'=T'_\ell \times {T'}^{(\ell)}$, where $T_\ell$ and $T'_\ell$ are the $\ell$-Sylow subgroups of $T$ and $T'$, respectively, and $T^{(\ell)}$ and ${T'}^{(\ell)}$ are the complementary subgroups. Let $N^{(\ell)} \coloneqq T^{(\ell)} \rtimes W$ and ${N'}^{(\ell)} \coloneqq {T'}^{(\ell)} \rtimes W'$.
    \item Let $\Delta$ be the $\ell$-Sylow subgroup of $G$ chosen as follows: if $p \equiv 1 \pmod{\ell}$, then $\Delta=T_\ell$, and if $p \equiv -1 \pmod{\ell}$, then $\Delta=T'_\ell$.  In either case $\Delta$ is isomorphic to the $\ell$-Sylow subgroup of $\F_{p^2}^\times$.  We use $\Delta^\vee$ to denote the group of characters of $\Delta$.
\end{itemize}

For a commutative ring $A$ and a character $\chi: \F_p^\times \to A^\times$, we consider $\chi$ as a character of $T$ by $\chi\sm{a}{0}{0}{b} \coloneqq\chi(ab^{-1})$. Let $\eta: \F_p^\times \to \{ \pm 1\}$ be the unique character of order $2$. We also consider $\eta$ as a character of $G$ via the determinant. 

Let $F/\Q_\ell$ be a finite extension over which the irreducible representations of $G$ are defined.  Let $\cO$ be the valuation ring of $F$ with uniformizer $\varpi \in \cO$ and residue field $\F$. For a representation $\tilde\rho:G \to \GL_n(\cO)$, let $\rho=\tilde\rho \otimes_\cO F$ and $\rhobar=\tilde \rho \otimes_\cO \F$. Conversely, for a representation $\rho: G \to \GL_n(F)$, we let $\tilde\rho$ denote a choice of $\cO$-lattice in $\rho$, with the caveat that this lattice is not unique, but the resulting $\rhobar$ is well defined up to semisimplification.

For a ring $R$, let $Z(R)$ denote its center. For an irreducible representation $\rho:G \to \GL_n(F)$, let $e_\rho \in Z(F[G])$ denote the associated central idempotent.  Explicitly,
\[
e_\rho = \frac{\dim(\rho)}{|G|} \sum_{g \in G} \tr\rho(g) [g^{-1}].
\]
These are mutually orthogonal as $\rho$ varies over the irreducible $F[G]$-representations.
The \emph{central character of $\rho$}, denoted $\omega_\rho: Z(\cO[G]) \to \cO$, is defined to be the $\cO$-algebra homomorphism given by $f \cdot \rho = \omega_\rho(f) \rho$ for $f \in Z(\cO[G])$. For a conjugacy class $c \subset G$, we have $\omega_\rho(c)=\frac{|c|}{\dim(\rho)} \tr(\rho)(x)$, for any choice of $x \in c$.

For a group $G$ and a subgroup $H \le G$, let $N_G(H)$ and $C_G(H)$ denote its normalizer and centralizer in $G$, respectively.

\subsection{Representation theory of \texorpdfstring{$\PGL_2(\F_p)$}{} in characteristic zero}\label{subsec:char0 rep thy}
We briefly review the representation theory of $G$ over $F$, mostly for the purpose of establishing notation for its irreducible representations.  A more detailed summary for $\GL_2(\F_p)$, along with additional references, can be found in \cite[Chapter 2]{BushnellHenniart}.

For a character $\chi: T \to F^\times$, we can consider $\chi$ as a character of $B$ by inflation and let $R_T^G(\chi) \coloneqq \Ind_B^G(\chi)$. For a character $\theta: T' \to F^\times$ with $\theta^2 \ne 1$, let $R_{T'}^G(\theta)$ be the cuspidal representation of $G$ associated with $\theta$. It can be described using the character formula $R_{T'}^G(\theta) = \Ind_U^G(\psi) - \Ind_{T'}^G(\theta)$ for a nontrivial character $\psi$ of $U$. Both $R_T^G$ and $R_{T'}^G$ can be described using Deligne--Lusztig induction, which justifies using the similar notation.

The irreducible representations of $G$ over $F$ come in four types.
\begin{itemize}
    \item The $1$-dimensional representations are $1$ and $\eta$.
    \item The $p$-dimensional Steinberg representations are given by $\St = \St_1 = R_T^G(1)-1$ and $\St_\eta=R_T^G(\eta)-\eta$.
    \item The $(p+1)$-dimensional representations are the principal series representations $R_T^G(\chi)$ for $\chi:T \to F^\times$ with $\chi^2 \ne 1$. Note that $R_T^G(\chi) \cong R_T^G(\chi')$ if and only if $\chi' = \chi^{\pm 1}$.
    \item The $(p-1)$-dimensional representations are the cuspidal representations $R_{T'}^G(\theta)$ for $\theta$ with $\theta^2 \ne 1$. Note that $R_{T'}^G(\theta) \cong R_{T'}^G(\theta')$ if and only if $\theta' \in \{\theta, \theta^p\}$.
\end{itemize}
Let $\Irr(G)$ denote the set of isomorphism classes of irreducible representations of $G$ over $F$.
\subsection{Modular representation theory of \texorpdfstring{$\PGL_2(\F_p)$}{}}
Since $\ell$ divides the order of $G$, the ring $\F[G]$ is not semisimple. This comes from the fact that the idempotents $e_\rho \in Z(F[G])$ that are used to split exact sequences over $F[G]$ are not in $\cO[G]$ in general due to the $|G|$ in the denominator. Nonetheless, irreducible representations of $F[G]$ can be broken into blocks according to their reductions over $\F$, and these blocks have associated idempotents in $Z(\F[G])$.

The set $\Irr(G)$ is decomposed as a disjoint union of subsets called blocks. These blocks can be defined in any of the following equivalent ways.
\begin{itemize}
    \item If $\rho, \rho' \in \Irr(G)$ have the property that $\rhobar$ and $\rhobar'$ have a Jordan--H\"older factor in common, then $\rho$ and $\rho'$ are in the same block, and `in the same block' is the equivalence relation generated by this relation.
    \item Two $\rho, \rho' \in \Irr(G)$ are in the same block if and only if $\omega_\rho \equiv \omega_{\rho'} \pmod{\varpi}$.
    \item If $\B \subset \Irr(G)$, then $\B$ is a union of blocks if and only if $\sum_{\rho \in \B} e_\rho$ is in $\cO[G]$.
\end{itemize}
If $\B$ is a block, then the idempotent $e_\B \coloneqq \sum_{\rho \in \B} e_\rho \in Z(\cO[G])$ is called the \textit{block idempotent}. It also defines an idempotent in $\F[G]$ by reduction. We say that a simple $\F[G]$-module $\pi$ is in the block $\B$ if $e_\B \pi =\pi$. The rings $e_\B\cO[G]$ and $e_\B \F[G]$ are called the \textit{block algebras}.

Associated with a block is a conjugacy class of $\ell$-subgroups of $G$ called the \textit{defect groups} of the block; the $\ell$-adic valuation of the order of a defect group is called the \emph{defect} of the block. If $\B$ is a block and $D$ is a defect group of $\B$, then the $\ell$-part of the index of $D$ in $G$ is the highest power of $\ell$ that divides the dimension of every element of $\Irr(G)$ in $\B$. In particular, the defect groups of $\B$ are the $\ell$-Sylow subgroups of $G$ if and only if $\B$ contains an element of $\Irr(G)$ whose dimension is prime to $\ell$.

For each $\ell$-subgroup $L \le G$, the restriction map
\[
\Br_L:Z(\F[G]) \to \F[C_G(L)]
\]
is a ring homomorphism, called the \textit{Brauer homomorphism}. For a block $\B$ of $G$, a \textit{$\B$-Brauer pair} $(L,b)$ is a pair of an $\ell$-subgroup $L$ and a block $b$ of $C_G(L)$ such that $e_b \Br_L(e_\B) =e_b$. The group $G$ acts on the set of $\B$-Brauer pairs by conjugation, and we let $N_G(L,b)$ be the stabilizer of $(L,b)$ for this action.  If $L=D$ is a defect group of $\B$, then the number $[N_G(D,b):DC_G(D)]$ is called the \textit{inertial index} of the block. A block $\B$ with positive defect is called \textit{nilpotent}\footnote{Defect-zero blocks are also usually considered nilpotent, but in this paper we reserve the term \emph{nilpotent} for blocks with positive defect.} if $N_G(L,b)/C_G(L)$ is an $\ell$-group for every $\B$-Brauer pair $(L,b)$ \cite{BP1980}. By a theorem of Alperin and Brou\'e \cite[Proposition 4.21]{AB1979}, a block $\B$ that has abelian defect group and inertial index $1$ is nilpotent.

\subsection{Blocks of interest}
\label{sec: blocks}
The blocks with positive defect fall into three classes: the principal block (that is, the block containing the trivial representation), the block containing the quadratic character $\eta$, and the nilpotent blocks.  Since the block containing $\eta$ is essentially the twist of the principal block by $\eta$, we ignore this case in our analysis. It can be treated similarly to the principal-block case. In every case the $\ell$-Sylow subgroup $\Delta$ of $G$ is a defect group of the block.

We now give more explicit descriptions of the simple $F[G]$-modules and $\F[G]$-modules that appear in these blocks of interest.  Since $\ell$ divides $|G|$ and $\ell \ne p$, it follows that $p \equiv \pm 1 \pmod \ell$. The descriptions of the blocks depend on this sign.

\subsubsection{Blocks for $p \equiv 1 \pmod{\ell}$}
Suppose that $p \equiv 1 \pmod{\ell}$.
\begin{itemize}
    \item The principal block
    \begin{itemize}
        \item consists of the simple $F[G]$-representations $1$, $\St$, and $R_T^G(\xi)$, where $\xi$ runs over the nontrivial characters of $\Delta$.
        \item The simple $\F[G]$-modules in this block are $1$ and $\overline{\St}$.
    \end{itemize}
    \item The nilpotent blocks are in bijection with pairs $\{\chi, \chi^{-1}\}$, where $\chi \colon T \to F^\times$ is a character such that $\bar{\chi}^2 \neq 1$. 
        \begin{itemize}
        \item The block associated with $\{\chi, \chi^{-1}\}$ consists of the simple $F[G]$-representations  $R_T^G(\chi\xi)$, where $\xi$ runs over the characters of $\Delta$. 
        \item The only simple $\F[G]$-module in this block is $\overline{R_T^G(\chi)}$.
    \end{itemize}
\end{itemize}

\subsubsection{Blocks for $p \equiv -1 \pmod{\ell}$}
Suppose that $p \equiv -1 \pmod{\ell}$.
\begin{itemize}
    \item The principal block
    \begin{itemize}
        \item consists of the simple $F[G]$-representations $1$, $\St$, and $R_{T'}^G(\xi)$, where $\xi$ runs over the nontrivial characters of $\Delta$.
        \item The simple $\F[G]$-modules in this block are $1$ and $\overline{R_{T'}^G(\xi)}$ for any $\xi$.
    \end{itemize}
    \item The nilpotent blocks are in bijection with pairs $\{\theta, \theta^p\}$, where $\theta: T' \to F^\times$ is a character such that $\bar\theta^2\ne 1$.
        \begin{itemize}
        \item The block associated with $\{\theta, \theta^p\}$ consists of the simple $F[G]$-representations   $R_{T'}^G(\theta \xi)$, where $\xi$ runs over the characters of $\Delta$.
        \item The only simple $\F[G]$-module in this block is $\overline{R_{T'}^G(\theta)}$.
    \end{itemize}
\end{itemize}

\section{Computation of block algebras}
\label{sec:block algebras}
For each block $\B$ considered in \cref{sec: blocks}, we compute the block algebras $e_\B \F[G]$ and $e_\B\cO[G]$. The main purpose of this is to prove the following theorem, which describes the $T$-fixed parts of projective $e_\B\cO[G]$-modules.  As in the introduction, $\cO[\Delta]^+$ denotes the subring of $\cO[\Delta]$ that is fixed by the involution that inverts group-like elements, and $\cO[\Delta]^-$ is the submodule on which that involution acts by $-1$.

\begin{theorem}\label{thm:X^T}
Let $\B$ be a block of $G$, and let $M$
be a finitely generated projective $e_\B\cO[G]$-module. Let $\tr: M^T \to M^B$ denote the trace map.
\begin{enumerate}
    \item Assume $p \equiv 1 \pmod{\ell}$.
    \begin{enumerate}
        \item \label{part:p=1 principal}
        If $\B$ is the principal block, then $\ker(M^T\xrightarrow{\tr} M^B)$ has a natural action of $\cO[\Delta]^+$, and there is an isomorphism of $\cO[\Delta]^+$-modules
        \[
         \ker(M^T\xrightarrow{\tr} M^B) \cong (\cO[\Delta]^-)^a \oplus (\cO[\Delta]^+)^b,
        \]
        where $a=\rank_{\cO}(M^G)$ and $a+b=\rank_{\cO}(M^B)$.  Moreover, if $a=b$ then the $\cO[\Delta]^+$-action can be upgraded to an $\cO[\Delta]$-action, and $\ker \tr$ is a free $\cO[\Delta]$-module of rank $a$.
        \item \label{part:p=1 nilp}
        If $\B$ is the block containing $R_T^G(\chi)$ for a character $\chi$ with $\bar \chi^2 \ne 1$, then $M^B=0$, and $M^T$ has a natural action of $\cO[\Delta]$ and is a free $\cO[\Delta]$-module.
    \end{enumerate}
    \item Assume $p \equiv -1 \pmod{\ell}$.
    \begin{enumerate}
        \item \label{part:p=-1 principal}
        If $\B$ is the principal block, then $\ker(M^T \xrightarrow{\tr} M^B)$ has a natural action of $\cO[\Delta]^+$ and there is an isomorphism of $\cO[\Delta]^+$-modules
        \[
        \ker(M^T\xrightarrow{\tr} M^B) \cong \cO^a \oplus (\cO[\Delta]^+)^b,
        \]
        where $a=\rank_{\cO}(M^G)$ and $2a+b=\rank_{\cO}(M^B)$.
        \item \label{part:vexing}
        If $\B$ is the block containing $R_{T'}^G(\theta)$ for a character $\theta$ with $\bar \theta^2 \ne 1$, then $M^B=0$, and $M^T$ has a natural action of $\cO[\Delta]$ and is a free $\cO[\Delta]$-module.
    \end{enumerate}
\end{enumerate}
\end{theorem}
The adjective `natural' in the statement of the theorem is in the sense of category theory.  We show that the functor $M \mapsto \ker(M^T \xrightarrow{\tr} M^B)$ is representable, and in each case the indicated ring is the endomorphism ring of the representing object.

We now outline the proof of \cref{thm:X^T} and explain why it is related to computing the block algebras. In general, to compute a finite-dimensional $\F$-algebra $A$, one can use the fact that $A^\mathrm{opp}=\End_A(A)$.  Since $A$ is a projective $A$-module, it can be written as a direct sum of projective indecomposable modules $A=\oplus_i P_i$. One is thus led to compute the $P_i$'s along with the homomorphism modules $\Hom_A(P_i,P_j)$. Over $\F$, the projective indecomposable modules are exactly the projective covers of the simple modules. Since our blocks have at most two simple modules over $\F$, there are at most two projective indecomposables to consider. 

Once we know the projective indecomposable $e_\B \F[G]$-modules, we can lift them to projective modules for $e_\B\cO[G]$. We compute the homomorphisms between them on a case by case basis. We can rely, in every case but one, on known Morita equivalences to simpler rings. For the principal block when $p \equiv -1 \pmod{\ell}$, only a derived equivalence exists, and we are forced to proceed by hand. 

Finally, $M$ must be a direct sum of copies of the projective indecomposable $e_\B\cO[G]$-modules. By Frobenius reciprocity, $M^T=\Hom_{\cO[G]}(e\cO[G/T],M)$ and $M^B=\Hom_{\cO[G]}(e\cO[G/B],M)$.  To complete the proof of the theorem, we express $e\cO[G/B]$ and $e\cO[G/T]$ in terms of known modules.

\subsection{Block algebras for \texorpdfstring{$p \equiv 1 \pmod{\ell}$}{}}
Assume that $p \equiv 1 \pmod{\ell}$. We prove \cref{thm:X^T}(1) by considering the two blocks separately.

\subsubsection{The principal block}
For this section, let $\mathcal{B}$ be the principal block of $\cO[G]$, and let $e=e_\mathcal{B}$. The main tool for analyzing $e\mathcal{O}[G]$ is the following theorem, attributed to Puig in \cite[Theorem 12.2.4, pg.~156]{bonnafe} (see also \cite[Theorem 23.12, pg.~366]{CE2004}).

\begin{theorem}[Puig]
\label{thm:mortia broue}
There is a Morita equivalence between the principal block $e\cO[G]$ and the principal block of $\cO[N]$.
\end{theorem}

Since $N=T \rtimes W$ is a semidirect product, one easily calculates its irreducible $F$-representations \cite[Proposition 25, pg.~62]{serre_repthry}. The $F[N]$-representations in the principal block are:
\begin{itemize}
    \item the 1-dimensional representations $1$ and $\sigma$, where $\sigma$ is the inflation of the nontrivial representation of $W$,
    \item the 2-dimensional representations $\Ind_T^N\chi$, where $\chi: T \to F^\times$ has $\ell$-power order.
\end{itemize}
The simple $\F[N]$-modules in the block are $1$ and $\bar\sigma$. The projective covers of $1$ and $\sigma$ are $\Ind_{N^{(\ell)}}^N 1$ and $\Ind_{N^{(\ell)}}^N \bar\sigma$, respectively. Let $Q_1=\Ind_{N^{(\ell)}}^N \tilde{1}$ and $Q_\sigma = \Ind_{N^{(\ell)}}^N \tilde\sigma$ be corresponding projective $\cO[N]$-modules.

\begin{lemma}
\label{lem:homs over N}
    For $x,y \in \{1,\sigma\}$, there are isomorphisms
    \[
    \Hom_{\cO[N]}(Q_x,Q_y) \cong \begin{cases}
        \cO[\Delta]^+ & \text{if }x=y,\\
        \cO[\Delta]^- & \text{if }x \ne y.
    \end{cases}
    \]
    These are isomorphisms of rings if $x=y$ and of $\cO[\Delta]^+$-modules if $x \ne y$.
\end{lemma}

\begin{proof}
For $(x,y)=(1,1)$, Frobenius reciprocity implies that
\[
\End_{\cO[N]}(\Ind_{N^{(\ell)}}^N1) =\cO[N^{(\ell)} \backslash N /N^{(\ell)}].
\]
The bijection $N/N^{(\ell)} \cong \Delta$ induces a bijection between $N^{(\ell)} \backslash N /N^{(\ell)}$ and the orbits of $\Delta$ for the action of inversion. Sending such an orbit $\{\delta,\delta^{-1}\}$ to the element $[\delta]+[\delta]^{-1}$ defines an $\cO$-module bijection between $\End_{\cO[N]}(Q_1)$ and $\cO[\Delta]^+$, and it is a simple exercise to check that it respects multiplication. The remaining cases are similar but require cumbersome notation, so we leave them as an exercise. 
\end{proof}

Now let $P_1$ and $P_{\St}$ be the projective $e\cO[G]$-modules that correspond to $Q_1$ and $Q_\sigma$, respectively, under the equivalence of \cref{thm:mortia broue}. They correspond to projective covers of $1$ and $\St$, respectively. It is not difficult to check the following descriptions over $F$:
\begin{equation}
\label{eqn:constituents}
P_1 \otimes_\cO F = 1 \oplus \bigoplus_{\chi} R_T^G(\chi), \ \quad P_\St \otimes_\cO F = \St \oplus \bigoplus_{\chi} R_T^G(\chi),
\end{equation}
where $\chi$ ranges over equivalence classes of nontrivial characters of $\Delta$ up to inversion.
\cref{lem:homs over N} implies that, for $x,y \in \{1,\St\}$, 
\begin{equation}
    \label{eq:Homs P_1 and P_st}
        \Hom_{\cO[G]}(P_x,P_y) \cong \begin{cases}
        \cO[\Delta]^+ & \text{if }x=y,\\
        \cO[\Delta]^- & \text{if }x \ne y.
    \end{cases}
\end{equation}
In particular, $e\cO[G]$ is Morita equivalent to the ring
\[
\End_{\cO[G]}(P_1 \oplus P_\St) \cong \ttmat{\cO[\Delta]^+}{\cO[\Delta]^-}{\cO[\Delta]^-}{\cO[\Delta]^+} \subseteq M_2(\cO[\Delta]).
\]
The diagonal inclusion of $\cO[\Delta]^+$ into the center of this ring defines a map $\cO[\Delta]^+ \to Z(e\cO[G])$ that can be written explicitly as sending an element $[\alpha]+[\alpha]^{-1} \in \cO[\Delta]^+$ for $\alpha \in \Delta$ to
\begin{equation}
\label{eq:f_alpha principal p=1}
 2e_1+2e_\St + \sum_\chi (\chi(\alpha)+\chi^{-1}(\alpha))e_{R_T^G(\chi)},   
\end{equation}
where $\chi$ ranges over inversion-classes of nontrivial characters of $\Delta$.

Note that there is an injective ring homomorphism
\[
\cO[\Delta] \to \End_{\cO[G]}(P_1 \oplus P_\St), \ f \mapsto \ttmat{f^+}{f^-}{f^-}{f^+},
\]
where $f^\pm = \frac{f \pm f^\iota}{2} \in \cO[\Delta]^\pm$ for $f \in \cO[\Delta]$ are the usual projections. 
For $x \in \{1,\St\}$, $\Hom_{\cO[G]}(P_x,P_1 \oplus P_\St)$ is an $\cO[\Delta]$-module via this homomorphism, and \eqref{eq:Homs P_1 and P_st} implies
\begin{equation}
\label{eq:ODelta upgrade}
\Hom_{\cO[G]}(P_x,P_1 \oplus P_\St) \cong \cO[\Delta].
\end{equation}
It is clear that this is actually an isomorphism of $\cO[\Delta]$-modules.

We later require the following concrete descriptions of $P_1$ and $P_\St$.

\begin{lemma}
\label[lemma]{lem:O[G/T]}
    We have $e\cO[G/U] \cong P_1 \oplus P_\St$ and $e\cO[G/T] \cong e\cO[G/B] \oplus P_\St$.
\end{lemma}
\begin{proof}
Because $\ell \nmid |U|$, $e\cO[G/U]$ is a projective $e\cO[G]$-module, so it must be a direct sum of copies of $P_1$ and $P_\St$.  That each appears once can be checked easily over $F$. 

Regarding $e\cO[G/T]$, first note that, since $\ell$ does not divide either of $[B:T]$ or $[G:B]$, the inclusions $\cO[G/B] \to \cO[G/T]$ and $\cO \to \cO[G/B]$ split.  Thus $\cO[G/T] = \cO[G/B] \oplus M$ for some module $M$, and $\cO[G/B] = \cO \oplus \St$. It remains to check that $eM\cong P_\St$. For this, observe that $T \backslash G/B$ has three elements, so Mackey's formula implies that $\Hom_{\F[G]}(\F[G/T],\F[G/B])$ is three-dimensional. From this and what we have already proved, it follows that $\Hom_{\F[G]}(e\overline{M},1)=0$ and $\dim_{\F}\Hom_{\F[G]}(e\overline{M},\overline{\St})=1$, so that the cosocle of $e\bar M$ is $\overline{\St}$. It follows that there is a surjection $P_\St \onto eM$, and it can be checked easily that $P_\St \otimes_\cO F$ is isomorphic to $eM \otimes_\cO F$, so this surjection is an isomorphism.
\end{proof}

\begin{proof}[Proof of \Cref{thm:X^T}\eqref{part:p=1 principal}]
    Since $P_1$ and $P_\St$ are the only projective indecomposable $e\cO[G]$-modules up to isomorphism, there is an isomorphism $M \cong P_1^a \oplus P_\St^b$ for some integers $a$ and $b$. 
    The identifications 
     $$
     M^T \cong \Hom_{e\cO[G]}(e\cO[G/T],M) \quad \text{and} \quad M^B \cong \Hom_{e\cO[G]}(e\cO[G/B],M),
     $$ 
     and the isomorphism $e\cO[G/T]\cong \cO[G/B] \oplus P_\St$ of \cref{lem:O[G/T]}, combine to give an isomorphism
     \[
     \ker(M^T \to M^B) \cong \Hom_{e\cO[G]}(P_\St, M).
     \]
     Then $\cO[\Delta]^+\cong\End_{e\cO[G]}(P_\St)$ acts on $\ker(M^T \to M^B)$ through its action on $P_\St$, and 
     \[
     \ker(M^T \to M^B) \cong \Hom_{e\cO[G]}(P_\St, P_1)^a \oplus \End_{e\cO[G]}(P_\St)^b,
     \]
     so the isomorphism 
     \[
     \ker(M^T \to M^B) \cong (\cO[\Delta]^-)^a \oplus (\cO[\Delta]^+)^b
     \]
     follows from \eqref{eq:Homs P_1 and P_st}.
    
     To identify the constants $a$ and $b$ in terms of invariants, we have by \eqref{eqn:constituents} that   $P_1^G \cong \cO$ and $P_{\St}^G = 0$, and thus $\rank_{\cO}(M^G) = a$.  Further, since $\cO[G/B] \cong \cO \oplus \St$ and 
    $M^B \cong \Hom_{\cO[G]}(\cO[G/B],M)$, \eqref{eqn:constituents} gives $\rank_{\cO}(M^B) = a+b$.
    
     The fact that when $a = b$ the $\cO[\Delta]^+$-structure can be upgraded to an $\cO[\Delta]$-structure making $\ker(M^T \to M^B)$ into a free $\cO[\Delta]$-module of rank $a$ follows from  the fact that $\ker(M^T \to M^B) \cong \Hom_{\cO[G]}(P_{\St}, P_1 \oplus P_{\St})^a$ and \eqref{eq:ODelta upgrade}.
\end{proof}
\subsubsection{Nilpotent blocks}
\label{sec:Nil blocks p=1}
Now let $\chi: T \to F^\times$ be a character of prime-to-$\ell$ order such that $\bar\chi^2 \ne 1$.  Let $\B$ be the block containing $R_T^G(\chi)$, and let $e=e_\B$. Note that the character $\chi$ is not determined by the block because $R_T^G(\chi) \cong R_T^G(\chi^{-1})$; we call the choice of a character $\chi$ a \emph{pinning character} of the block.

To study this block, we use Brou\'{e} and Puig's theory of nilpotent blocks \cite{BP1980}. The defect group of $\B$ is $D=\Delta$ and $C_G(D)=T$. Let $b_{\chi}$ be the block of $T$ containing $\chi$; it can be checked using Brauer's second main theorem that $(\Delta,b_\chi)$ is a $\B$-Brauer pair. We claim that the inertial index of $\B$ is $1$; that is, $N_G(\Delta,b_\chi)=T$. To see this, note that $N_G(\Delta,b_\chi)$ is the subgroup of elements $n \in N$ such that $nb_\chi n^{-1}=b_\chi$. It is easy to see that $nb_\chi n^{-1} = b_{\chi^n} = \begin{cases}
    b_\chi & n \in T \\
    b_{\chi^{-1}} & n \not\in T.
\end{cases}$. Since $\bar\chi \ne \bar\chi^{-1}$, it follows that $N_G(\Delta,b_\chi)=T$.

Then the structure results of Brou\'e and Puig \cite[pg.~120]{BP1980} imply that $e\cO[G] \cong \mathrm{M}_{p+1}(\cO[\Delta])$ is a matrix algebra over $\cO[\Delta]$. In particular, there is only one indecomposable projective module $P$ for $e\cO[G]$ up to isomorphism, and $\End_{\cO[G]}(P) \cong \cO[\Delta]$. This $P$ satisfies $P \otimes_\cO F \cong \bigoplus_{\xi \in \Delta^\vee} R_T^G(\chi \xi)$. The isomorphism $\cO[\Delta] \isoto \End_{\cO[G]}(P)$ can be written explicitly as the composition
\[
\cO[\Delta] \xrightarrow{f_\chi} Z(\cO[G]) \to \End_{\cO[G]}(P),
\] 
where the map $f_\chi$ sends $\alpha \in \Delta$ to the element
\begin{equation}
\label{eq:p=1 falpha}
    f_\chi(\alpha)=\sum_{\xi \in \Delta^\vee} \xi(\alpha) e_{R_T^G(\chi \xi)} = \sum_{\xi \in \Delta^\vee} (\chi\xi)(\alpha)e_{R_T^G(\chi\xi)} \in Z(\cO[G]).
\end{equation}
The last equality in \eqref{eq:p=1 falpha} holds since $\alpha$ has $\ell$-power order while the order of $\chi$ is prime to $\ell$. (Though $f_\chi(\alpha)$ is only {\it a priori} an element of $Z(F[G])$, it is easily checked to be integral.)  

Note that the isomorphism $\cO[\Delta] \cong \End_{\cO[G]}(P)$ given by \eqref{eq:p=1 falpha} depends on the choice of pinning character $\chi$.  Indeed, one could equally well define
\[
f_{\chi^{-1}}(\alpha) = \sum_{\xi \in \Delta^\vee} \xi(\alpha)e_{R_T^G(\chi^{-1}\xi)} = \sum_{\xi \in \Delta^\vee} \xi^{-1}(\alpha)e_{R_T^G(\chi\xi)}.
\]

\begin{lemma}
\label[lemma]{lem:Ind_T^G in nilpotent p=1}
We have $e\Ind_B^G \cO = 0$, and $e\Ind_T^G \cO$ is projective and indecomposable.
\end{lemma}
\begin{proof}
Since $\Ind_B^G\cO=\cO \oplus \St$ and neither of these representations is in the block, the first statement is obvious.

    For the second statement, let $Q=e\Ind_T^G\cO$, and let $\bar P$ be the projective cover of $R_{T}^G(\bar\chi)$ (as an $e\F[G]$-module), and $P$ a projective $e\cO[G]$-module lifting $\bar P$. Mackey's formula implies that $\Hom_{e\F[G]}(\bar Q,R_T^G(\bar\chi))\cong \F$, so there is a surjection $P \onto Q$. Mackey's formula also shows that $\Hom_{eF[G]}(Q \otimes_\cO F,R_T^G(\chi\xi))\cong F$, so that $Q\otimes_\cO F \cong P\otimes_\cO F$. Hence the surjection $P \onto Q$ must be an isomorphism. 
\end{proof}

\begin{proof}[Proof of \cref{thm:X^T}\eqref{part:p=1 nilp}]
    Since $M^B=\Hom_{e\cO[G]}(e\Ind_B^G\cO,M)$ and $e\Ind_B^G\cO=0$ by \cref{lem:Ind_T^G in nilpotent p=1}, $M^B=0$. Since $M$ is projective and there is only one projective indecomposable $P$, it must be that $M=P^a$ for some integer $a$. Since $P=e\Ind_T^G\cO$ by \cref{lem:Ind_T^G in nilpotent p=1} and $M^T=\Hom_{e\cO[G]}(e\Ind_T^G\cO,M)$, it follows that $M^T=\End_{e\cO[G]}(P)^a$. The result follows from the fact that $\End_{e\cO[G]}(P)=\cO[\Delta]$.
\end{proof}

\subsection{Block algebras for \texorpdfstring{$p \equiv -1 \pmod{\ell}$}{}}
We now assume $p\equiv -1 \pmod{\ell}$ and prove \cref{thm:X^T}(2) by considering the two block types separately.

\subsubsection{The principal block}
For this section, let $\mathcal{B}$ be the principal block of $\cO[G]$, and let $e=e_\mathcal{B}$. Let $\xi: \Delta \to F^\times$ be a nontrivial character, and let $\pi=\overline{R_{T'}^G(\xi)}$, so that $1$ and $\pi$ are the simple $\F[G]$-modules in the principal block. There is a theorem similar to \cref{thm:mortia broue} in this case, but it only asserts an equivalence of the \emph{derived} categories of $e\cO[G]$ and of the principal block of $\cO[N']$. (This is a special case of a general conjecture of Brou\'{e}, see \cite[Theorem 4.6]{BR2006}.) This derived equivalence is not enough for our purposes, so we have to compute directly.

Let $P_1=\Ind_B^G \cO$.

\begin{lemma}
\label{lem:end of P_1}
    There is a presentation 
    \[
    \End_{\cO[G]}(P_1) \cong \frac{\cO[x]}{(x^2-[G:B]x)},
    \]
    and $P_1$ is a projective cover of $\cO$. Moreover, there is a decomposition $e\Ind_T^G \cO = P_1 \oplus P_\pi$, where $P_\pi$ is a projective cover of $\pi$.
\end{lemma}
\begin{proof}
    Write $P_1=\cO[G/B]$ as the module of $B$-invariant functions, and note that $P_1 \otimes F=1\oplus \St$ so $eP_1=P_1$. Consider the endomorphism $x:P_1 \to P_1$ sending a function $f$ to the constant function $\sum_{g \in G/B} f(g)$; in other words, $x=[G:B]e_1$. Then clearly $x^2=[G:B]x$, so this defines a homomorphism 
    \begin{equation}\label{eq:rk2map}
    \frac{\cO[x]}{(x^2-[G:B]x)} \to \End_{\cO[G]}(P_1).
    \end{equation}
    Since $x$ induces a nontrivial map $x: \bar P_1 \to \bar P_1$, it is a split injection of $\cO$-modules. Since $\End_{\cO[G]}(P_1) \otimes F =\End_{\cO[G]}(1 \oplus \St) \cong F \oplus F$, it follows that $\End_{\cO[G]}(P_1)$ is free of rank two as an $\cO$-module, and hence \eqref{eq:rk2map}  is an isomorphism.
    
    To see that $P_1$ is a projective cover of $\cO$, note that $P_1$ is projective since $\ell \nmid |B|$.  Moreover $P_1$ is indecomposable since $\End_{\cO[G]}(P_1)$ is local.  The map $x$ lands in the copy of $\cO$ in $P_1$, which gives a surjection $P_1 \onto \cO$, so $P_1$ must be the projective cover of $\cO$.

    Similarly, $e\Ind_T^G \cO$ is projective. Let $P_\pi$ denote a projective indecomposable $e\cO[G]$-module that lifts a projective cover of $\bar\pi$.  Since $P_1$ and $P_\pi$ are the only indecomposable projective modules, $e\Ind_T^G \cO=P_1^a \oplus {P_\pi}^b$ for some nonnegative integers $a$ and $b$. Then $a=\dim_{F}\Hom_{F[G]}(\Ind_T^G(F),F)=1$ and $b=\dim_{F}\Hom_{F[G]}(\Ind_T^G(F),\pi)$, which a simple computation reveals to also be $1$.
\end{proof}
The module $P_\pi$ is also isomorphic to the following:
\begin{itemize}
    \item $e \Ind_T^G \chi$ for a nontrivial character $\chi$ of $T$;
    \item $e \Ind_U^G\psi$ for a nontrivial character $\psi$ of $U$.
\end{itemize}
Indeed, since $T$ and $U$ are prime-to-$\ell$ subgroups of $G$, these modules are projective, so it is enough to show that they are isomorphic to $P_\pi$ over $F$.  This is easily verified using Frobenius reciprocity and Mackey's formula.
The endomorphism ring of $P_\pi$ has been computed \cite[Theorem 4.11]{Paige}.

\begin{lemma}[Paige]
\label{lem: end of P_pi}
    There is an isomorphism $\cO[\Delta]^+ \isoto \End_{\cO[G]}(P_\pi)$.
\end{lemma}
The isomorphism can be described explicitly as follows. Note that $P_\pi \otimes_\cO F = \St \oplus \bigoplus_{\xi} R_{T'}^G(\xi)$, where $\xi$ ranges over a set of equivalence classes of nontrivial characters of $\Delta$ modulo inversion. Then for $\alpha \in \Delta$, the element $[\alpha]+[\alpha^{-1}] \in \cO[\Delta]^+ $ acts by
\begin{equation}
\label{eq: f alpha}
    2e_1 + 2e_\St + \sum_\xi (\xi(\alpha)+\xi^{-1}(\alpha))e_{R_{T'}^G(\xi)}.
\end{equation}

\begin{lemma}
    \label{lem:hom P1 P_pi}
    There are isomorphisms $\Hom_{\cO[G]}(P_1,P_\pi) \cong \cO$ and  $\Hom_{\cO[G]}(P_\pi,P_1) \cong \cO$ as $\End_{\cO[G]}(P_\pi)$- and $\End_{\cO[G]}(P_1)$-modules, where both rings act on $\cO$ through their augmentations.
\end{lemma}
\begin{proof}
    Since the only irreducible representation that $P_1 \otimes F$ and $P_\pi \otimes F$ have in common is $\St$ with multiplicity one, it follows that $\Hom_{\cO[G]}(P_1,P_\pi) \otimes F \cong F$ and $\Hom_{\cO[G]}(P_\pi,P_1) \otimes F \cong F$ as $F$-vector spaces. 
    As in the proof of \Cref{lem:end of P_1}, a generator of the augmentation ideal of $\End_{\cO[G]}(P_1)$ is $x=[G:B]e_1$, which acts as zero on $\St$, so the above isomorphisms are $\End_{\cO[G]}(P_1) \otimes F$-equivariant if $F$ is given the augmentation action. Similarly, \eqref{eq: f alpha} implies the same for the $\End_{\cO[G]}(P_\pi)\otimes F$-action.
    Since $\Hom_{\cO[G]}(P_1,P_\pi)$ and $\Hom_{\cO[G]}(P_\pi,P_1)$ are torsion-free, it follows that they are isomorphic to $\cO$ with the claimed actions.
\end{proof}
Although we do not require them, explicit generators of $\Hom_{\cO[G]}(P_1,P_\pi)$ and $\Hom_{\cO[G]}(P_\pi,P_1)$ can be described as follows. There is an endomorphism $w: \Ind_T^G \cO \to \Ind_T^G \cO$ defined by a generator $w$ of $W$. The maps
\[
P_1 \hookrightarrow e\Ind_T^G \cO \xrightarrow{w} e\Ind_T^G \cO \onto P_\pi, \ P_\pi \hookrightarrow e\Ind_T^G \cO \xrightarrow{w} e\Ind_T^G \cO \onto P_1
\]
generate $\Hom_{\cO[G]}(P_1,P_\pi)$ and $\Hom_{\cO[G]}(P_\pi,P_1)$, respectively. To see this, it is enough to show that they are nonzero modulo $\varpi$. A computation shows that the element of $\End_{\cO[G]}(P_1)$ obtained by composing them is, in the notation of \cref{lem:end of P_1},
\[
\frac{1-p}{p}(x-[G:B]).
\]
Since this composition is nonzero modulo $\varpi$, both of the maps are also nonzero.
\begin{proof}[Proof of \cref{thm:X^T}\eqref{part:p=-1 principal}]
Since $P_1$ and $P_\pi$ are the only projective indecomposable modules in the block, $M=P_1^a \oplus P_\pi^b$ for some integers $a$ and $b$. 
    Since $e \Ind_T^G \cO = P_1 \oplus P_\pi$ by \cref{lem:end of P_1}, it follows that 
    \[
    \ker(M^T \to M^B)=\Hom_{\cO[G]}(P_\pi,M) \cong \Hom_{\cO[G]}(P_\pi,P_1)^a \oplus \End_{\cO[G]}(P_\pi)^b.
    \]
    The formula for $\ker(M^T \to M^B)$ then follows from \cref{lem:hom P1 P_pi,lem: end of P_pi}.

    To describe the constants $a$ and $b$, we have
    $$
    P_1 \otimes_{\cO} F = F \oplus \St \quad \text{and} \quad
    P_\pi \otimes_\cO F = \St \oplus \bigoplus_{\xi} R_{T'}^G(\xi).
    $$
Thus, $P_1^G \cong \cO$, $P_\pi^G = 0$, and $\rank_{\cO} M^G = a$.
Further, since
$M^B \cong \Hom_{\cO[G]}(\cO[G/B],M)$ and $F[G/B] \cong F \oplus (\St \otimes_\cO F)$, we see that $\rank_{\cO} M^B = 2a+b$.
\end{proof}

We can now describe the algebra $e\cO[G]$ explicitly. Note that 
\[
e\cO[G] = e\Ind_{1}^G \cO=e\Ind_{T}^G(\Ind_{1}^T \cO) = e\Ind_T^G \cO \oplus \bigoplus_\chi e\Ind_T^G \chi \cong P_1 \oplus P_\pi^{p-1},
\]
where $\chi$ ranges over nontrivial characters of $T$, and the last isomorphism follows from \Cref{lem:end of P_1} and the comment following it. Then, since $e\cO[G] \cong \End_{e\cO[G]}(e\cO[G])$, applying \cref{lem:end of P_1,lem: end of P_pi,lem:hom P1 P_pi}, 
there is an isomorphism of rings
\[
e\cO[G] \cong \left(\begin{matrix}
    \cO[x]/(x^2-[G:B]x) & \cO & \cdots &\cO \\
    \cO & \cO[\Delta]^+ & \cdots &\cO[\Delta]^+ \\
     \vdots & \vdots & & \vdots \\
     \cO& \cO[\Delta]^+ & \cdots &\cO[\Delta]^+
\end{matrix}\right),
\]
where:
\begin{itemize}
    \item the matrix ring has $p$ rows and columns;
    \item the action of $\cO[x]/(x^2-[G:B]x)$ and $\cO[\Delta]^+$ on $\cO$ is by the augmentation;
    \item the implicit maps 
    \[
    \cO \times \cO \to \cO[\Delta]^+, \ \cO \times \cO \to \cO[x]/(x^2-[G:B]x) 
    \]
    send $(1,1)$ to an element of the ring whose exact annihilator is the augmentation ideal.
\end{itemize}
\subsubsection{Nilpotent blocks}
\label{sec:nilpotent_blocks2}
Let $\theta: T' \to F^\times$ be a character of prime-to-$\ell$ order such that $\bar \theta^2 \ne 1$.  Let $\mathcal{B}$ be the block containing $R_{T'}^G(\theta)$, and let $e=e_\mathcal{B}$. Just as in \Cref{sec:Nil blocks p=1}, the block $\B$ does not determine the character $\theta$, and we refer to the choice of $\theta$ as a \emph{pinning character}.

An argument just as in \Cref{sec:Nil blocks p=1} shows that $\mathcal{B}$ is nilpotent with defect group $\Delta$ and that $e\cO[G] \cong \mathrm{M}_{p-1}(\cO[\Delta])$. A similar argument to \Cref{lem:Ind_T^G in nilpotent p=1} shows that $P=e\Ind_T^G\cO$ is projective and indecomposable and that $e\Ind_B^G\cO=0$; in particular $\End_{\cO[G]}(P)\cong \cO[\Delta]$.
This isomorphism can be made explicit just as in \eqref{eq:p=1 falpha}: an element $\alpha \in \Delta$ acts on $P$ by the central element 
\begin{equation}
\label{eq:p=-1 falpha}
    f_\theta(\alpha) = \sum_{\xi \in \Delta^\vee} \xi(\alpha) e_{R_{T'}^G(\theta\xi)} = \sum_{\xi \in \Delta^\vee} (\theta\xi)(\alpha)e_{R_{T'}^G(\theta\xi)} \in Z(\cO[G]),
\end{equation}
where the last equality follows since $\alpha$ has $\ell$-power order and $\theta$ has order prime to $\ell$.  Once again, this identification depends on the choice of pinning character $\theta$. Indeed, one could equally well define 
\[
f_{\theta^p}(\alpha) = \sum_{\xi \in \Delta^\vee} \xi(\alpha)e_{R_{T'}^G(\theta^p\xi)} = \sum_{\xi \in \Delta^\vee} \xi^p(\alpha)e_{R_{T'}^G(\theta\xi)}.
\]
Now the proof of \Cref{thm:X^T}\eqref{part:vexing} is exactly analogous to that of \eqref{part:p=1 nilp}.
\begin{proof}[Proof of \Cref{thm:X^T}\eqref{part:vexing}]
Since $e\Ind_B^G\cO=0$ and $M^B=\Hom_{\cO[G]}(e\Ind_B^G\cO,M)$, it follows that $M^B=0$. Since $P=e\Ind_T^G\cO$ is the only projective indecomposable module, $M=P^a$ for some integer $a$. Hence
\[
M^T = \Hom_{\cO[G]}(e\Ind_T^G\cO,M) \cong \End_{\cO[G]}(P)^a \cong \cO[\Delta]^a.\qedhere
\]
\end{proof}

This completes the proof of \Cref{thm:X^T}.

\section{Projective modules of modular forms}
\label{sec:projective}

In this section, we review and slightly generalize a theorem of Serre stating that certain spaces of modular forms
with full level at $p$ are projective $\cO[G]$-modules. We restate the arguments here because, as far as we are aware, these results have only appeared in Serre's course notes \cite{serrecoursenotes} and not in any published articles. (In fact, we independently arrived at the same theorems, with essentially the same proofs, before we were aware of Serre's notes.) Although we focus on coherent cohomology, similar techniques work for other avatars of modular forms, such as modular symbols and \'etale cohomology of modular curves.

\subsection{Galois-module structure of cohomology of tamely ramified covers} 
The projectivity statement we require is derived from general results about the Galois-module structure of coherent cohomology of ramified covers of schemes (see \cite{Nakajima84,Nakajima,ChinburgErez}). 
The following result, which is \cite[Theorem 1]{Nakajima}, suffices for our purposes. 

\begin{theorem}[Nakajima]
\label{thm:nakajima}
    Let $\F$ be a field of characteristic $\ell$, and let $f:X \to Y$ be a finite Galois covering of projective varieties over $\F$ with Galois group $G$. Let $\cF$ be a coherent $G$-sheaf on $X$. Assume that $f$ is \emph{tame} in the sense that, for every geometric point $x \in X$, the order of the stabilizer of $x$ in $G$ is prime to~$\ell$.

    There is a perfect complex $C^\bullet$ of $\F[G]$-modules such that $H^*(X,\cF) \cong H^*(C^\bullet)$ as $\F[G]$-modules. In particular, if $H^*(X,\cF)$ is concentrated in degree $n$, then $H^n(X,\cF)$ is a projective $\F[G]$-module.
\end{theorem}

\subsection{Modular curves and modular forms}
We recall modular curves as (coarse) moduli spaces of elliptic curves with level structure, as in \cite{katzmazur}, but we always consider the case of prime-to-$\ell$ levels and schemes over $\Z_{(\ell)}$ so that the moduli spaces are smooth. We also recall the geometric definition of modular forms with coefficients in a ring.

\subsubsection{Level structures}
Let $a$ and $b$ be integers, and let $S$ be a $\Z[1/ab]$-scheme. For an elliptic curve $E$ over $S$, we consider two types of level structures on $E$:
\begin{description}
    \item[$\Gamma(a,b)$-level structure] an injection $\Z/a\Z \times \Z/b\Z \to E$ of group schemes over $S$.
    \item[$\Gamma_0(a,b)$-level structure] a pair $(A,B)$ of subgroup schemes $A,B \subset E$ such that $A$ and $B$ are \'etale-locally isomorphic to $\Z/a\Z$ and $\Z/b\Z$, respectively, and $A \cap B = 0$.
\end{description}
We use the abbreviations $\Gamma(a)=\Gamma(a,a)$, $\Gamma_1(a)=\Gamma(a,1)$, and $\Gamma_0(a)=\Gamma_0(a,1)$. There is a congruence subgroup of $\SL_2(\Z)$ associated with a level structure in the usual way, and the subgroups associated with $\Gamma(a)$, $\Gamma_1(a)$, and $\Gamma_0(a)$ are the usual ones. The group $\GL_2(\Z/a\Z)$ acts transitively on the set of $\Gamma(a)$-level structures on $E$.

Let $\Gamma$ be one of these level structures; the \emph{level} of $\Gamma$ is the integer $\mathrm{lcm}(a,b)$. If the level of $\Gamma$ is a product $nm$ with $\gcd(n,m)=1$, then a $\Gamma$-level structure is equivalent to a pair of level structures $\Gamma_n$ and $\Gamma_m$ of level $n$ and of level $m$, respectively \cite[Section (3.5), pg.~101]{katzmazur}.  The corresponding congruence subgroups are related by $\Gamma=\Gamma_n \cap \Gamma_m$.

Following \cite[Section (4.4), pg.~109]{katzmazur}, a $\Gamma$-level structure is called \emph{rigid} if, for all $E/S$, the group $\Aut(E/S)$ acts freely on the set of $\Gamma$-level structures on $E$. For example, the level structures $\Gamma(a)$ for $a>2$ and $\Gamma_1(a)$ for $a>3$ are rigid by \cite[Section 2.7, pg.~85]{katzmazur}.

\subsubsection{Modular curves} Let $\Gamma$ be a level structure of level $N$. Define the modular curve $Y(\Gamma)$, which is smooth over $\Z[1/N]$, to be the coarse moduli space of isomorphism classes of pairs $(E,\alpha)$ of an elliptic curve $E/S$ and a $\Gamma$-level structure $\alpha$. If $\Gamma$ is rigid, then $Y(\Gamma)$ is the fine moduli space (\cite[Corollary 4.7.1,pg.~116]{katzmazur}). 
Otherwise choose an integer $a>2$ with $\gcd(N,a)=1$ and define a rigid level structure $\Gamma'=\Gamma\cap \Gamma(a)$. Then $Y(\Gamma)$ can be constructed as the quotient scheme $Y(\Gamma')/\GL_2(\Z/a\Z)$, which is independent of the choice of $a$ \cite[Section (8.1), pg.~224]{katzmazur}.

The scheme $Y(\Gamma)$ has a compactification by adding a finite set of points (called `cusps') $X(\Gamma)=Y(\Gamma) \cup C(\Gamma)$ \cite[Section 8.6]{katzmazur}, which gives a smooth projective curve over $\Z[1/N]$. Note that the $\GL_2(\Z/N\Z)$-action on $\Gamma(N)$-level structures induces a $\GL_2(\Z/N\Z)$-action on $Y(\Gamma(N))$, which extends to an action on $X(\Gamma(N))$.

\subsubsection{Canonical level structures} Let $n$ be an integer, $S$ be a $\Z[1/n]$-scheme, and $E$ be an elliptic curve over $S$. Then a $\Gamma(n)$-level structure $(P,Q)$ on $E$ defines an $n$-th root of unity $e_n(P,Q) \in S$ via the Weil pairing \cite[Section 2.8, pg.~87]{katzmazur}. This defines a morphism $X(\Gamma \cap \Gamma(n)) \to \mu_n$.

Fix a primitive $n$-th root of unity $\zeta_n$, which corresponds to a morphism $\Spec \Z[\zeta_n] \to \mu_n$. Suppose that $S$ is a $\Z[1/n,\zeta_n]$-scheme. A $\Gamma(n)$-level structure $(P,Q)$ on $E$ is \emph{canonical} if $e_n(P,Q)=\zeta_n$.  Define $X(\Gamma \cap \Gamma(n)^\mathrm{can})$ to be the fiber product of $X(\Gamma \cap \Gamma(n)) \otimes_{\Z[1/n]} \Z[1/n, \zeta_n]$ with $\Spec(\Z[1/n, \zeta_n])$ over $\mu_n \otimes_{\Z[1/n]} \Z[1/n, \zeta_n]$; it can be thought of as the compactified moduli space of elliptic curves with $\Gamma$-structure and canonical $\Gamma(n)$-structure. If $K$ is a field extension of $\Q(\zeta_n)$, then $X(\Gamma \cap \Gamma(n)) \otimes_{\Z[1/n]} K$ is isomorphic to a disjoint union of copies of $X(\Gamma \cap \Gamma(n)^\mathrm{can}) \otimes_{\Z[1/n, \zeta_n]} K$ over the embeddings of $\Q(\zeta_n)$ in $K$. In particular, over the complex numbers, $X(\Gamma(n))(\C)$ is a disjoint union of copies of $X(\Gamma(n)^\mathrm{can})(\C)$.  It is $X(\Gamma(n)^\mathrm{can})(\C)$ that can be identified with a classical modular curve given by a quotient of the upper half-plane.

Note that the action of $\SL_2(\Z/n\Z)$ on $X(\Gamma \cap \Gamma(n))$ preserves the subscheme $X(\Gamma \cap \Gamma(n)^\mathrm{can})$, but the action of $\GL_2(\Z/n\Z)$ does not.

\subsubsection{Modular forms} For a commutative $\Z[1/N]$-algebra $A$, let $X(\Gamma)_{/A}$ be the base change of $X(\Gamma)$ to $A$. We define the spaces of modular forms and cusp forms of level $\Gamma$ with coefficients in $A$ in terms of sections of appropriate line bundles on $X(\Gamma)_{/ A}$, as in \cite{katz} (see also \cite[Section 10.13]{katzmazur}). 

First assume that $\Gamma$ is rigid.
Let  $\pi: E \to Y(\Gamma)$ be the universal elliptic curve, and let $\omega$ be the line bundle on $Y(\Gamma)$ defined by $\omega=\mathrm{coLie}(E)=\pi_* \Omega_{\pi}^1$. There is an extension, still denoted $\omega$, of $\omega$ to $X(\Gamma)$.  We have an isomorphism
\[
\omega^{\otimes 2} \isoto \Omega^1_{X(\Gamma)} (\log),
\]
where $\Omega^1_{X(\Gamma)}(\log)$ denotes the sheaf of differential forms on $X(\Gamma)$ with log poles along $C(\Gamma)$ \cite[Theorem 10.13.11, pg.~335]{katzmazur}. Define
\[
M_k(\Gamma,A) \coloneqq H^0(X(\Gamma)_{/A}, \omega^{\otimes k}) = H^0(X(\Gamma)_{/A}, \omega^{{\otimes k-2}} \otimes \Omega^1_{X(\Gamma)}(\log))
\]
and
\[
S_k(\Gamma,A) \coloneqq H^0(X(\Gamma)_{/A}, \omega^{\otimes k}(-\log))=H^0(X(\Gamma)_{/A}, \omega^{{\otimes k-2}} \otimes \Omega^1_{X(\Gamma)}).
\]

If $\Gamma$ is not rigid, choose an integer $a>2$ with $\gcd(a,N)=1$, and let $\Gamma'=\Gamma \cap \Gamma(a)$, which is rigid, and define
\[
M_k(\Gamma,A) = M_k(\Gamma',A)^{\GL_2(\Z/a\Z)},\ S_k(\Gamma,A) = S_k(\Gamma',A)^{\GL_2(\Z/a\Z)}.
\]
\subsubsection{Comparison with canonical modular forms}
If $A$ is a field that contains $\Q(\zeta_n)$, then 
\[
M_k(\Gamma \cap \Gamma(n),A) = \prod_{\Q(\zeta_n) \to A} M_k(\Gamma \cap \Gamma(n)^\mathrm{can},A).
\]
By examining how the action of $\GL_2(\Z/n\Z)$ permutes the factors, one can show that there is an isomorphism of $A[\GL_2(\Z/n\Z)]$-modules
\begin{equation}
\label{eq:modular forms induced from canonical}
    M_k(\Gamma \cap \Gamma(n),A) \cong \Ind_{\SL_2(\Z/n\Z)}^{\GL_2(\Z/n\Z)} M_k(\Gamma \cap \Gamma(n)^\mathrm{can},A).
\end{equation}

Over the complex numbers, just as it is the canonical spaces that are identified with quotients of the upper half-plane, it is the canonical modular forms that can be identified with classical modular forms (namely, holomorphic functions on the upper half-plane satisfying certain conditions).
\subsection{Projectivity of modular forms}\label{subsec:projectivity of mod forms} Fix a level structure $\Gamma$ of level $N$ and an integer $n$ with $\gcd(N,\ell n)=1$ and $\ell \nmid n$. Let $H=H(\Gamma,n)=\Gal(X(\Gamma \cap \Gamma(n))/X(\Gamma))$ be the Galois group of the cover; we have $H=\GL_2(\Z/n\Z)/\{\pm 1\}$ if $-1 \in \Gamma$ and $H=\GL_2(\Z/n\Z)$ otherwise.

\begin{lemma}
\label{lem:tameness}
    Assume that $\ell \ge 3$ and that $\Gamma$ is rigid if $\ell = 3$. Then the $H$-cover ${X(\Gamma \cap \Gamma(n))}_{/ \F} \to {X(\Gamma)}_{/ \F}$ is tame in the sense of \cref{thm:nakajima}.
\end{lemma}
\begin{proof}
     If $\Gamma$ is rigid, then ${Y(\Gamma \cap \Gamma(n))}_{/ \F} \to {Y(\Gamma)}_{/ \F}$ is \'etale and the stabilizers of points in ${Y(\Gamma)}_{/ \F}$ are all trivial. In general, the stabilizers of these points have order dividing $6$, and hence are prime-to-$\ell$ if $\ell>3$ \cite[Corollary 2.7.1, pg.~85]{katzmazur}. Even if $\Gamma$ is rigid, the map $X(\Gamma \cap \Gamma(n)) \to X(\Gamma)$ is ramified at the cusps, but the ramification degree at the cusps divides $n$ , so the map is tame by the assumption that $\ell \nmid n$ \cite[Theorem 10.9.1, pg.~301]{katzmazur}.
\end{proof}

We can now complete the proof of Serre's theorem \cite[Theorem 1, pg.~53]{serrecoursenotes}.
\begin{theorem}[Serre]
\label{prop:projectivity of modular forms}
    Assume that $\ell \ge 3$ and that $\Gamma$ is rigid if $\ell = 3$. Then $M_2(\Gamma \cap \Gamma(n),\cO)$, and, for all $k > 2$,  $M_k(\Gamma \cap \Gamma(n),\cO)$ and $S_k(\Gamma \cap \Gamma(n),\cO)$, are projective $\cO[H]$-modules. 
\end{theorem}
\begin{proof}
    Let $A$ denote either ring $\cO$ or $\F$, and let $\cF$ be either sheaf $\omega^{\otimes k}$ or $\omega^{\otimes k}(-\log)$. Let $d$ denote the number of geometric connected components of $X(\Gamma\cap \Gamma(n))$. Just as in the proof of \cite[Theorem 1.7.1]{katz}, we have, for all $k \ge 2$,
    \[
    H^1(X(\Gamma \cap \Gamma(n))_{/A}, \cF) = \begin{cases}
    A^d & \text{if }k=2 \text{ and }\cF \cong \omega^2(-\log) \cong \Omega^1, \\
    0 & \text{otherwise}.
    \end{cases}
    \]
    In particular, $H^1(X(\Gamma \cap \Gamma(n))_{/\cO}, \cF)$ is torsion-free in every case, so there are base-change isomorphisms
    \[
    H^0(X(\Gamma \cap \Gamma(n))_{/\cO}, \cF) \otimes_\cO \F \cong H^0(X(\Gamma \cap \Gamma(n))_{/\F}, \cF).
    \]
    This implies that $M_k(\Gamma \cap \Gamma(n),\cO) \otimes_\cO \F$ and $S_k(\Gamma \cap \Gamma(n),\cO) \otimes_\cO \F$ are isomorphic to $H^0(X(\Gamma \cap \Gamma(n))_{/\F}, \cF)$ for the appropriate choice of $\cF$. Moreover, except in the case $k=2$ and $\cF=\omega^2(-\log)$, the cohomology $H^*(X(\Gamma \cap \Gamma(n))_{/\F}, \cF)$ is concentrated in degree $0$.
    
    By \Cref{lem:tameness}, we may apply \Cref{thm:nakajima}. This implies that, for $k>2$, $M_2(\Gamma \cap \Gamma(n),\F)$, $M_k(\Gamma \cap \Gamma(n),\F)$ and $S_k(\Gamma \cap \Gamma(n),\F)$ are projective $\F[H]$-modules. Since the corresponding modules with $\cO$-coefficients are $\cO$-free, this implies that they are $\cO[H]$-projective. 
\end{proof}

\begin{remark}
    \label{rem:projectivity for p=3}
    As Serre points out \cite[Contre-exemple, pg.~54]{serrecoursenotes}, the assumption that $\Gamma$ is rigid if $\ell=3$ is necessary:\ for $\ell=3$, $n=13$, and $\Gamma=\SL_2(\Z)$, the space $M_k(\Gamma(13),\cO)$ is not $\cO[H]$-projective. The issue is the fact that the map $Y(\Gamma \cap \Gamma(n))_{/\F} \to Y(\Gamma)_{/\F}$ may be wildly ramified at points of $Y(\Gamma)_{/\F}$ that correspond to elliptic curves with complex multiplication by $\Q(\zeta_3)$. Serre shows that this is the only issue, in the sense that, if $\ell=3$ and $\m$ is a maximal ideal in the Hecke algebra with $T_q \not\in \m$ for some prime $q$ not dividing $n$ or the level of $\Gamma$ with $q \equiv -1 \pmod{3}$, then the localization $M_k(\Gamma \cap \Gamma(n),\cO)_\m$ is $\cO[H]$-projective. In other words, the failure of projectivity comes entirely from maximal ideals that correspond to mod-$3$ Galois representations that are induced from $\Q(\zeta_3)$. 
\end{remark}

\begin{remark}
We focus primarily on the depth-$0$ situation of \cref{prop:projectivity of modular forms} when $n = p$ is prime.  This seems to be the most interesting case from an $\ell$-adic perspective since the kernel of the natural map $\GL_2(\Z/p^m\Z) \to \GL_2(\F_p)$ is pro-$p$.  However, there is no obstruction to applying our methods in higher depth situations since \cref{prop:projectivity of modular forms} is the main input needed.
\end{remark}

\section{Hecke operators and familiar spaces of modular forms}\label{sec:Hecke ops}

Although \cref{prop:projectivity of modular forms} with $n = p$ allows us to apply \cref{thm:X^T} to blocks in spaces of modular forms, that does not help us study congruences between modular forms.  For that, we need to localize our space of modular forms at a maximal ideal of the Hecke algebra.  In this section we introduce the Hecke operators and use them to relate the objects appearing in \cref{thm:X^T} to more familiar spaces of modular forms.  

Let us make a few preliminary observations and establish notation.  Strictly speaking, in order to apply \cref{thm:X^T} to the spaces of modular forms considered in \cref{prop:projectivity of modular forms}, we first need to move from $H$-modules to $G$-modules, where $G=\PGL_2(\F_p)$ is the quotient of $H$ by its center. Let $Z \subset \tilde{T} \subset \tilde{B} \subset H$ be the center, the diagonal torus, and the upper-triangular Borel in $H$, respectively, so that $G=H/Z$, $B=\tilde{B}/Z$, and $T=\tilde{T}/Z$. Let $\tilde{M} = M_k(\Gamma \cap \Gamma(p),\cO)$, and let $M=\tilde{M}^Z$, so that $M$ is an $\cO[G]$-module. The relevant fixed point spaces are
\begin{align*}
    M^T & = \tilde{M}^{\tilde{T}} = M_k(\Gamma \cap \Gamma_0(p,p),\cO), \\
    M^B &= \tilde{M}^{\tilde{B}} = M_k(\Gamma \cap \Gamma_0(p),\cO).
\end{align*}
Since taking fixed parts under a normal subgroup sends projective modules to projective modules,  \cref{prop:projectivity of modular forms} (when its hypotheses are satisfied) implies that $M$ is a projective $\cO[G]$-module. Then \cref{thm:X^T} describes the structure of $\ker(M^T \to M^B)$. In this section we show that $\ker(M^T \to M^B)$ --- and an analogous kernel obtained after localizing at a maximal ideal of the anemic Hecke algebra --- can be described as a more familiar space of modular forms using the $U_p$ Hecke operator.

\subsection{Hecke operators} 
We briefly recall the geometric definition of Hecke operators $T_q$ when $\Gamma = \Gamma_0(N)$ for some $N$ such that $(Nq, \ell) = 1$. The Hecke operators for other choices of $\Gamma$ can be defined similarly, and the Hecke operator $T_\ell$ can be defined in the same way over $\Z[1/N\ell]$ or in a similar way over $\Z[1/N]$ by using Drinfeld level structures to define $X(\Gamma \cap \Gamma_0(\ell))$; see, for instance, \cite[Chapter~2]{mazurwiles}.

Let $q$ be a prime number, and write $N = q^aN'$ for some $a \geq 0$ and $q \nmid N'$.  Set $\Gamma^{(q)} = \Gamma_0(N')$.  There is an isomorphism
\[
\iota: X(\Gamma^{(q)} \cap \Gamma_0(q^{a+1})) \isoto X(\Gamma^{(q)} \cap \Gamma_0(q^a,q))
\]
defined in terms of the moduli problem as follows. For an elliptic curve $E$, a $\Gamma^{(q)}$-level structure $\alpha$ on $E$, and a cyclic subgroup $C \subset E$ of order $q^{a+1}$, let $\pi: E \to E'$ be the isogeny whose kernel is the $q$-torsion subgroup of $C$. Then $\pi(C)$ and $\pi(E[q])$ are cyclic subgroups of $E'$ of order $q^a$ and $q$, respectively, and hence give a $\Gamma_0(q^a,q)$-structure on $E'$. The map $\iota$ sends $(E,\alpha,C)$ to $(E',\pi \circ \alpha,(\pi(C),\pi(E[q])))$.

Using $\iota$, there are two degeneracy maps
\[
\pi_q,\psi_q: X(\Gamma^{(q)} \cap \Gamma_0(q^{a+1})) \to X(\Gamma^{(q)} \cap \Gamma_0(q^{a})) = X(\Gamma_0(N))
\]
given by $\pi_q(E,\alpha,C)=(E,\alpha,qC)$ and $\psi_q(E,\alpha,C)=(E',\pi \circ \alpha,\pi(C))$; that is, $\pi_q$ is the forgetful map, and $\psi_q$ is the composition of $\iota$ with the forgetful map. The Hecke operator $T_q$ is defined by $T_q=\psi_{q,*}\pi_q^*$ as an endomorphism of $M_k(\Gamma_0(N),\cO)$ or $S_k(\Gamma_0(N),\cO)$. If $a>0$, then $T_q$ is also traditionally written as $U_q$.

For a level structure $\Gamma$, let $\T(\Gamma)$ denote the $\cO$-subalgebra of $\End_{\cO}(M_k(\Gamma,\cO))$ generated by the Hecke operators $T_q$ for all primes $q$, and let $\T'(\Gamma) \subseteq \T(\Gamma)$ be the anemic subalgebra generated by $T_q$ for all $q \ne p$.  We also abuse this notation slightly and write $\T(\Gamma \cap \Gamma(p))$ for the $\cO$-subalgebra of $\End_{\cO}(M_k(\Gamma \cap \Gamma(p), \cO)^Z)$ and similarly for $\T'(\Gamma \cap \Gamma(p))$.  We now recall some basic properties of Hecke operators that are useful in what follows.

\begin{lemma}\label{lem:Hecke alg properties}\hfill
\begin{enumerate}[(i)]
\item\label{item: commute}  For two primes $q_1 \neq q_2$, the operators $T_{q_1}$ and $T_{q_2}$ commute.
\item\label{item: cartesian} With notation as above, there is a commutative diagram
\begin{equation}
\label{eq:cart}
\xymatrix{
X(\Gamma^{(q)} \cap \Gamma_0(q^{a+2})) \ar[r]^-{\pi_q} \ar[d]^-{\psi_q} & X(\Gamma^{(q)} \cap \Gamma_0(q^{a+1})) \ar[d]^-{\psi_q} \\
X(\Gamma^{(q)} \cap \Gamma_0(q^{a+1})) \ar[r]^-{\pi_q}  & X(\Gamma^{(q)} \cap \Gamma_0(q^{a})),
}    
\end{equation}
that is Cartesian if and only if $a>0$. 
\item\label{item: G-commute} If $q \neq p$ is prime, then $T_q$ commutes with the $G$-action on $M = M_k(\Gamma \cap \Gamma(p), \cO)^Z.$
\end{enumerate}
\end{lemma}

\begin{proof}
These are all quite standard and can be easily checked.  Note that \eqref{item: cartesian} can be checked by first applying $\iota$ to the top row of \eqref{eq:cart}.
\end{proof}

\subsection{Familiar spaces of modular forms}\label{subsec:familiar mod forms} Let $\Gamma$ be a level structure of level $N$ with $p \nmid N$, and let $M=M_k(\Gamma \cap \Gamma(p),\cO)^Z$. The map $\iota$ induces an isomorphism of anemic Hecke modules
\begin{equation} \label{eq:Gamma0(p,p) isoto Gamma0(p^2)}
  M^T = M_k(\Gamma \cap \Gamma_0(p,p),\cO) \isoto M_k(\Gamma \cap \Gamma_0(p^2),\cO),
\end{equation}
and we can identify the image of $\ker(M^T \to M^B)$ under this map.

\begin{lemma}
\label{lem:ker trace = ker Up}
    Under the isomorphism \eqref{eq:Gamma0(p,p) isoto Gamma0(p^2)}, $\ker(M^T \stackrel{\tr}{\to} M^B)$ is sent to the kernel of $U_p$ on $M_k(\Gamma \cap \Gamma_0(p^2),\cO)$.
\end{lemma}
\begin{proof}
    The trace map $M^T \to M^B$ is identified with the pushforward map $\pi'_*:M_k(\Gamma \cap \Gamma_0(p,p),\cO) \to M_k(\Gamma \cap \Gamma_0(p),\cO)$, where
    \[
    \pi': X(\Gamma \cap \Gamma_0(p,p)) \to X(\Gamma \cap \Gamma_0(p))
    \]
    is the forgetful map. Since $\psi_p=\pi' \circ \iota$ by definition, this shows that $\iota$ sends $\ker(M^T \to M^B)$ to $\ker(\psi_{p,*})$. Since \eqref{eq:cart} is Cartesian for $q=p$ and $a=1$, we have
    \[
    U_p=\psi_{p,*}\pi_p^* = \pi_p^*\psi_{p,*}:M_k(\Gamma \cap \Gamma_0(p^2),\cO) \to M_k(\Gamma \cap \Gamma_0(p^2),\cO).
    \]
    Since $\pi_p^*$ is injective, this implies $\ker(U_p)=\ker(\psi_{p,*})$.
\end{proof}

For the remainder of this section, let $\T=\T(\Gamma \cap \Gamma_0(p^2))$ and $\T'=\T'(\Gamma \cap \Gamma_0(p^2))$.  The next lemma shows that, after localizing at a maximal ideal of $\T'$, the kernel of $U_p$ on $M_k(\Gamma \cap \Gamma_0(p^2), \cO)_{\m'}$ is given by localizing at the maximal ideal $\m$ of $\T$ generated by $\m'$ and $U_p$.  This allows us to describe the output of \cref{thm:X^T} for modular forms in terms of the familiar space $M_k(\Gamma \cap \Gamma_0(p^2), \cO)_\m$ in \cref{sec:main thms}.

\begin{lemma}\label{lem:anemic to full}
Let $\m' \subseteq \T'$ be a maximal ideal, and let $\m \subseteq \T$ be the ideal generated by $\m'$ and $U_p$. Then
\label{lem:indecomposble}\hfill
    \begin{enumerate}[(i)]
        \item \label{part:m max} $\m$ is a maximal ideal of $\T$;
        \item \label{part:Up=0} $U_p$ is zero in $\T_\m$;
        \item \label{part:Up ker=loc} the inclusion $
M_k(\Gamma \cap \Gamma_0(p^2),\cO)_{\m'}^{U_p=0} \isoto M_k(\Gamma \cap \Gamma_0(p^2),\cO)_{\m}$ is an isomorphism;
\item there is an action of $Z(\cO[G]) \times \T_\m$ on $M_k(\Gamma \cap \Gamma_0(p^2),\cO)_{\m}$.
    \end{enumerate}
\end{lemma}
\begin{proof} \hfill
\begin{enumerate}[(i)]
    \item We have to show that $\m$ is a proper ideal. Since $\m' \subset \T'$ is a maximal ideal, there is a $\T'$-eigenform $f' \in M_k(\Gamma \cap \Gamma_0(p^2),\cO)$ such that $T_q-a_q(f') \in \m'$ for all $q \ne p$. To show that $\m$ is proper, we have to show that there is a $\T$-eigenform $f$ with $a_q(f)=a_q(f')$ for all $q \ne p$ and with $a_p(f)=0$. If $f'$ is new at level $p^2$, then it is well known that $a_p(f')=0$ \cite[Theorem 4.6.17]{Miyakebook}. If $f'$ is old at level $p^2$, then there is an eigenform $g \in M_k(\Gamma \cap \Gamma_0(p),\cO)$ with $a_q(f')=a_q(g)$ for all $q \ne p$  \cite[Proposition 5.8.4, pg.~198]{DiamondShurman}. With the notation of \eqref{eq:cart} with $a=1$, abbreviate $\psi_p$ and $\pi_p$ to $\psi$ and $\pi$. Since $T_q$ for $q \ne p$ commutes with $\psi^*$ and $\pi^*$, the forms $\psi^*g$ and $\pi^*g$ have the same $T_q$-eigenvalues as $g$ (and $f'$), so it suffices to find a linear combination of $\psi^*g$ and $\pi^*g$ that is in the kernel of $U_p$. Since \eqref{eq:cart} with $a=1$ is Cartesian, we have $U_p=\pi^*\psi_{*}$.  We compute
    \begin{align*}
            U_p \psi^*g& =\pi^*\psi_{*}\psi^*g = p\pi^*g, \text{ and }\\
            U_p \pi^*g &=\pi^*\psi_{*}\pi^*g =\pi^*U_pg=a_p(g)\pi^*g.
    \end{align*}
    Letting $f=a_p(g)\psi^*g-p\pi^*g$, it follows that $U_pf=0$. 
    \item We have to show that if $f \in M_k(\Gamma \cap \Gamma_0(p^2))_{\m'}$ is an eigenform and $a_p(f) \equiv 0 \pmod{\varpi}$, then $a_p(f)=0$. As in the proof of \eqref{part:m max}, this is clear if $f$ is new, so we may assume that there is an eigenform $g \in M_k(\Gamma \cap \Gamma_0(p), \cO)_{\m'}$ such that $f$ is an $\cO$-linear combination of $\psi^*g$ and $\pi^*g$. The computation of the $U_p$-action in \eqref{part:m max} shows that the eigenvalues of $U_p$ are $0$ and $a_p(g)$, so either $a_p(f)=0$ or $a_p(f)=a_p(g)$. But $U_p$ is invertible on $M_k(\Gamma \cap \Gamma_0(p),\cO)$:\ the eigenvalues on $p$-newforms are $\pm p^{(k-2)/2}$, and the eigenvalues on $p$-oldforms multiply to $p^{k-1}$ and hence are $\ell$-adic units.  So we must have $a_p(f) = 0$.
    
    \item Since $M_k(\Gamma \cap \Gamma_0(p^2),\cO)_{\m}$ is the part of $M_k(\Gamma \cap \Gamma_0(p^2),\cO)_{\m'}$ on which $U_p$ acts topologically nilpotently, this is clear from \eqref{part:Up=0}.
    
    \item Abusing notation, we also write $\m'$ for the pullback of $\m'$ to $\T'(\Gamma \cap \Gamma(p))$ via \eqref{eq:Gamma0(p,p) isoto Gamma0(p^2)} and the natural projection map.   Let $P=\coker(\cO[G/B] \to \cO[G/T])$.  By \cref{lem:ker trace = ker Up}, there is an isomorphism of anemic Hecke modules
    \[
    M_k(\Gamma \cap \Gamma_0(p^2),\cO)_{\m'}^{U_p=0} \cong \ker(M_{\m'}^T \to M_{\m'}^B) = \Hom_{\cO[G]}(P,M_{\m'}).
    \]
   The right-hand side has an action of $Z(\cO[G])$, and this action commutes with the action of $\T'$ by \cref{lem:Hecke alg properties}\eqref{item: G-commute}. Since $\T'$ generates $\T_\m$ by \eqref{part:Up=0}, this implies that there is an action of $\T_\m \times Z(\cO[G])$ on the right-hand side. The left-hand side is identified with $M_k(\Gamma \cap \Gamma_0(p^2),\cO)_{\m}$ by \eqref{part:Up ker=loc}.\qedhere
\end{enumerate}
\end{proof}

The next lemma, which is certainly well known, is the key input needed in \cref{sec:hecke} to establish that there is a unique block associated with a maximal ideal of the anemic Hecke algebra.

\begin{lemma}
\label{lem:indecomp2}
Assume that $\Gamma = \Gamma_1(N)$ or $\Gamma_0(N)$ for some $N$ such that $\gcd(N, p\ell) = 1$.  The $\T_\m$-module $M_k(\Gamma \cap \Gamma_0(p^2),\cO)_\m$ is indecomposable.
\end{lemma}
\begin{proof}
Let $M=M_k(\Gamma \cap \Gamma_0(p^2),\cO)$, and let 
\[
M'=\{f \in M_k(\Gamma \cap \Gamma_0(p^2),F) \ | \ a_n(f) \in \cO \text{ for all }n>0\}.
\]
There is a perfect pairing
\[
M' \times \T \to \cO
\]
given by $(f,t)\mapsto a_1(tf)$ \cite[Theorem (2.2)]{ribetmodp}. In particular, $M'$ is isomorphic to the dual of $\T$ as a $\T$-module, and hence $\End_{\T}(M')=\T$. By definition, there is an exact sequence
\[
0 \to M \to M' \xrightarrow{a_0} F/\cO,
\]
which is equivariant for the action of $\T$ if $F/\cO$ is given the action where each Hecke operator acts by its degree. In particular $U_p$ acts by $p$ on $M'/M$, so $(M'/M)_\m=0$. Then 
\[
\End_{\T_\m}(M_\m)=\End_{\T_\m}(M'_\m) = \T_\m,
\]
which is a local ring, so $M_\m$ is indecomposable.
\end{proof}

\section{The block associated with a maximal anemic-Hecke ideal}
\label{sec:hecke}
In this section we show that for any maximal ideal $\m'$ of $\T'(\Gamma \cap \Gamma(p))$, the $\cO[G]$-module $M \coloneqq M_k(\Gamma \cap \Gamma(p), \cO)^Z_{\m'}$ belongs to a unique block $\B_{\m'}$.  This is the final step needed to apply \cref{thm:X^T} to $M$, which is done in \cref{sec:main thms}.

We briefly describe the strategy.  Since $\T'(\Gamma \cap \Gamma(p))$ commutes with the $G$-action on $M_k(\Gamma \cap \Gamma(p), F)^Z$, we can decompose this space into $G$-representations according to its anemic eigensystems.  That is,
\begin{equation}
M_k(\Gamma \cap \Gamma(p), F)^Z \cong \oplus_\lambda M(\lambda),
\end{equation}
where $\lambda$ runs over $F$-algebra morphisms $\T'(\Gamma \cap \Gamma(p)) \otimes_{\cO} F \to F$ and $M(\lambda)$ is the $\lambda$-isotypic component.  Thus we can associate a $G$-representation, namely $M(\lambda)$, with any anemic Hecke eigensystem $\lambda$ that appears in $M_k(\Gamma \cap \Gamma(p), F)^Z$.  We find a description of $M(\lambda)$ --- using an explicit argument when $\lambda$ is Eisenstein (\cref{subsec:G-rep for Eis series}) and using automorphic representations when $\lambda$ is a cuspidal (\cref{subsec:reduce to automorphic,subsec:classify Kp-fixed reps}) --- which shows that each $M(\lambda)$ belongs to a unique block $\B_\lambda$.  Moreover, our description shows that $\ker(M(\lambda)^T \to M(\lambda)^B)$ is always $1$-dimensional.  Thus every anemic Hecke eigensystem in $M$ appears in $M^T$, which we saw in \cref{lem:ker trace = ker Up,lem:anemic to full} can be identified with $M_k(\Gamma \cap \Gamma_0(p^2), \cO)_\m$.  The fact that all of the $\B_\lambda$'s are equal when $\lambda$ runs over the eigensystems of $M$ follows from the indecomposability given by \cref{lem:indecomp2}.

Throughout this section, fix a weight $k \ge 2$ and a level structure $\Gamma = \Gamma_1(N)$ or $\Gamma_0(N)$ with $\gcd(N,\ell p)=1$.

\subsection{The \texorpdfstring{$G$}{}-representation associated with an Eisenstein series}\label{subsec:G-rep for Eis series} 

For any level $\Gamma'$, let $\mathcal{E}_k(\Gamma',F)$ denote the subspace of $M_k(\Gamma',F)$ spanned by Eisenstein series. 
Recall that the anemic eigensystems of $\mathcal{E}_k(\Gamma_0(p^2), F)$ are well known \cite[\S 4.5, p. 126]{DiamondShurman}.  Via the isomorphism \eqref{eq:Gamma0(p,p) isoto Gamma0(p^2)}, these eigensystems also appear in $\mathcal{E}_k(\Gamma(p), F)^Z$. For a character $\chi = (\psi, \psi^{-1})$ of $T$, write $\lambda_\chi$ for the eigensystem with $T_q$-eigenvalue $\psi(q) + \psi^{-1}(q)q^{k-1}$.

\begin{lemma}\label{lem:block for Eisenstein series}
The anemic Hecke eigensystems on $\mathcal{E}_k(\Gamma \cap \Gamma(p), F)^Z$ are given by $\lambda_\chi$ as $\chi$ runs over characters of $T$.  Let $d = \dim \mathcal{E}_k(\Gamma, F)$.  For each character $\chi$ of $T$,
\[
M(\lambda_\chi) \cong \begin{cases} (\Ind_B^G 1)^{d-1} \oplus \St & \chi = 1 \text{ and } k = 2,\\
(\Ind_B^G \chi)^d &\text{else.}
\end{cases}
\]
In particular, each $M(\lambda_\chi)$ belongs to a single block of $\cO[G]$.
\end{lemma}

\begin{proof}
We give the proof when $k > 2$ and $\Gamma = \SL_2(\Z)$.  The proof in the more general case is entirely analogous.  When $k = 2$ there is one fewer copy of the trivial representation (see \cite[Section 4.6, pg.~130]{DiamondShurman}).

We first describe $\mathcal{E}_k(\Gamma(p), F)^Z$ as a $G$-representation.  The space $\mathcal{E}_k(\Gamma(p)^\mathrm{can},F)$ is isomorphic to $F[(\F_p^2 -\{0\})/\{\pm 1\}]$ as a $F[\SL_2(\F_p)]$-module \cite[Section 4.2, pg.~111]{DiamondShurman}. Since $\SL_2(\F_p)/\{\pm 1\}$ acts transitively on $(\F_p^2 -\{0\})/\{\pm 1\}$ with isotropy group $\tilde{U}=\sm{1}{*}{0}{1} /
\{\pm1\}$, this implies that
\[
\mathcal{E}_k(\Gamma(p)^\mathrm{can},F) \cong \Ind_{\tilde{U}}^{\SL_2(\F_p)/\{ \pm 1\}} F.
\]
Hence by \eqref{eq:modular forms induced from canonical},
\[
\mathcal{E}_k(\Gamma(p),F) \cong \Ind_{\tilde{U}}^{\GL_2(\F_p)/\{\pm 1\}} F.
\]
This shows that $\mathcal{E}_k(\Gamma(p),F)^Z$ is, as a $G$-representation, a direct sum of $\Ind_B^G \chi$ over the characters $\chi$ of $T$. 

For each such $\chi$, the set of $T$-invariant vectors $(\Ind_B^G \chi)^T$ corresponds, via the isomorphism \eqref{eq:Gamma0(p,p) isoto Gamma0(p^2)}, to the $\lambda_\chi$-eigenspace in $\mathcal{E}_k(\Gamma_0(p^2), F)$.  Using \cref{subsec:char0 rep thy}, it follows that $\Ind_B^G \chi \subseteq M(\lambda_\chi)$ for all $\chi \neq \eta$, where $\eta$ is the quadratic character of $T$ (or $G$), and $\St_\eta \subseteq M(\lambda_\eta)$.  Thus it remains to show that $\lambda_\eta$ is the anemic eigensystem of the 1-dimensional subspace $V \subset \mathcal{E}_k(\Gamma(p), F)^Z$ on which $G$ acts by $\eta$.  

We can think of elements of $\mathcal{E}_k(\Gamma(p), F)^Z$ as pairs of elements of $\mathcal{E}_k(\Gamma(p)^\mathrm{can}, F)$ indexed by $\F_p^\times/(\F_p^\times)^2$. Let $a \in \F_p^\times \setminus (\F_p^\times)^2$.  Each $g \in G$ can be uniquely written as $g = \bigl(\begin{smallmatrix} x & 0\\ 
0 & 1
\end{smallmatrix}\bigr)g'$ with $x \in \{1, a\}$ and $g' \in \PSL_2(\F_p)$.  The $G$-action on a pair $(f_1,f_a)$ is given by $g(f_1, f_a) = (g'f_{\det(g)}, g'f_{\det(g)\cdot a})$.  In particular, $V$ is spanned by $(E_k, -E_k)$.  View $\eta$ as a character on $\F_p^\times$ by embedding $\F_p^\times \hookrightarrow T$ by sending $x \to \bigl(\begin{smallmatrix}
    x & 0\\
    0 & 1
\end{smallmatrix} \bigr)$.  Then $T_q(f_1, f_a) = (\eta(q)f_{\eta(q)}, \eta(q)f_{\eta(q)a})$, which shows that $\lambda_\eta$ is indeed the anemic eigensystem for $V$.

The last statement is immediate from the description of blocks in \cref{sec: blocks}.
\end{proof}

\subsection{Reduction to an automorphic calculation}
\label{subsec:reduce to automorphic}
Now let $\lambda$ be an anemic Hecke eigensystem on $S_k(\Gamma \cap \Gamma(p), \cO)^Z$.  We describe the $\lambda$-eigenspace in terms of the automorphic representation $\pi$ associated with any eigenform with eigensystem $\lambda$, or, more specifically, its $p$-component $\pi_p$.

Recall that with a cusp form $f \in S_k(\Gamma \cap \Gamma(p), F)^Z$ one associates an automorphic form $\phi_f$ on $\GL_2$, which generates an automorphic representation $\pi_f = \otimes_q' \pi_{f,q}$ (see \cite[Section 3.6, pg.~341]{Bumpbook}, for instance).  The representation $\pi_f$ carries an action of $\GL_2(\hat{\Z})$, and $\phi_f$ is invariant under an open compact subgroup $K = \prod_q K_q$ determined by $\Gamma$. In particular, the level $\Gamma(p)$ implies that $K_p$ can be taken to be
\[
K_p \coloneqq \ker(\GL_2(\Z_p) \to \PGL_2(\F_p)).
\]

In the next lemma we show that the $G$-subrepresentation generated by any $0 \neq f \in M(\lambda)$ is isomorphic to a $G$-subrepresentation of $\pi_{f,p}^{K_p}$.  Now $\pi_f$ (and hence $\pi_{f,p}^{K_p}$) depends only on the anemic Hecke eigensystem $\lambda$ \cite[Theorem 3.3.6]{Bumpbook}.  Thus to establish that $M(\lambda)$ belongs to a single block of $\cO[G]$, by \cref{lem:reduce to Kp-fixed rep} it suffices to show that $\pi_{f,p}^{K_p}$ belongs to a single block.  This is done in \cref{subsec:classify Kp-fixed reps}.

\begin{lemma}\label{lem:reduce to Kp-fixed rep}
Let $0 \neq f \in S_k(\Gamma \cap \Gamma(p), F)^Z$ be an anemic eigenform with associated automorphic representation $\pi_f$.  The $G$-subrepresentation of $S_k(\Gamma \cap \Gamma(p), F)^Z$ generated by $f$ is isomorphic to a $G$-subrepresentation of $\pi_{f,p}^{K_p}$, which depends only on the anemic Hecke eigensystem of $f$.
\end{lemma}

\begin{proof}
This follows from the standard fact that the geometric $G$-action on $S_k(\Gamma \cap \Gamma(p), F)^Z$ is compatible, under the map $f \mapsto \phi_f$, with the automorphic action of $\GL_2(\Z_p) \subset \GL_2(\hat{\Z})$ on the image (remembering that $K_p$ necessarily acts trivially on this image).  That $\pi_f$, and thus $\pi_{f,p}^{K_p}$, depends only on the anemic eigensystem of $f$ follows from \cite[Theorem 3.3.6]{Bumpbook}.
\end{proof}

\subsection{A classification of \texorpdfstring{$K_p$}{}-fixed vectors}\label{subsec:classify Kp-fixed reps}

Let $\pi$ be an infinite-dimensional irreducible admissible representation of $\GL_2(\Q_p)$ such that $\pi^{K_p}$ is nonzero.  
We now describe the classification of such representations $\pi$ and use it to describe the representation $\pi^{K_p}$ of $G$.  The classification shows that $\pi^{K_p}$ always belongs to a single $G$-block, which finishes the proof that $M(\lambda)$ belongs to a single $G$-block for any cuspidal anemic Hecke eigensystem $\lambda$. 

Let $\mathbf{G}=\GL_2(\Q_p)$ and $\mathbf{T} \subset \mathbf{B} \subset \mathbf{G}$ be the diagonal torus and upper-triangular Borel subgroups.  Let $\mathbf{K}= \GL_2(\Z_p)$ and $\mathbf{T}' \subset \mathbf{G}$ denote a nonsplit torus, which is isomorphic to $ \Q_{p^2}^\times$, such that $\mathbf{T}' \cap \mathbf{K}$ maps to $T' \subset G$ modulo $K_p$. We use the following definition and notational convention for characters of the tori $\mathbf{T}$ and $\mathbf{T}'$.

\begin{definition}
\label{def:level Kp characters}
A character $\tilde{\chi}: \mathbf{T} \to F^\times$ (resp.~$\tilde{\theta}: \mathbf{T}' \to F^\times$) \emph{has level $K_p$} if $\tilde{\chi}|_{\mathbf{T} \cap \mathbf{K}}$ (resp.~$\tilde{\theta}|_{\mathbf{T'} \cap \mathbf{K}}$) factors through a character $\chi:T \to F^\times$ (resp.~$\theta:T'\to F^\times$). We follow this notation convention: the character of level $K_p$ is denoted with a tilde and the corresponding character of the torus of $G$ is denoted by the same letter without a tilde.
\end{definition}

By the classification of infinite-dimensional, irreducible, admissible representations of $\mathbf{G}$ (see \cite[Classification Theorem 9.11, pg.~69, Proposition 12.6, pg.~91, and Proposition 19.1, pg.~126]{BushnellHenniart}), the ones that have nonzero $K_p$-fixed vectors are:
\begin{itemize}
    \item principal series representations $I(\tilde{\chi})$,
    \item  special representations $\sigma(\tilde{\chi})$, and
    \item depth-zero supercuspidal representations $\pi_{\tilde{\theta}}$,
\end{itemize}
where $\tilde{\chi}$ and $\tilde{\theta}$ have level $K_p$, and $\theta^{p+1} \ne 1$. Recall that these representations can be defined as follows. If the smooth induction $\Ind_\mathbf{B}^\mathbf{G} \tilde{\chi}$ is irreducible, then it is denoted $I(\tilde{\chi})$.  Otherwise it has a unique infinite-dimensional irreducible constituent, denoted $\sigma(\tilde{\chi})$.  Inflating $R_{T'}^G(\theta)$ to a representation of $\mathbf{K}$ and extending it to $\Q_p^\times\mathbf{K}$ via $\tilde{\theta}$ gives a $\Q_p^\times\mathbf{K}$-representation whose compact induction to $\mathbf{G}$ is denoted $\pi_{\tilde{\theta}}$.  

If $\pi$ is one of the representations $I(\tilde{\chi})$, $\sigma(\tilde{\chi})$, or $\pi_{\tilde{\theta}}$, then the resulting representation $\pi^{K_p}$ of $G$ can be computed easily using Mackey's formula. (See, for example, 
\cite[Lemme III.3.14, pg.~178]{Vignerasbook}, where this is done more generally for $\GL_n(\Q_p)$.) The result is summarized in the following lemma.

\begin{lemma}\label{lem:invariants of Qp reps} 
For characters $\tilde{\chi}:\mathbf{T}\to F^\times$ and $\tilde{\theta}: \mathbf{T}' \to F^\times$ of level $K_p$, there are isomorphisms of $G$-representations:
\begin{itemize}
    \item $I(\tilde{\chi})^{K_p} \cong R_T^G(\chi)$,
    \item $\sigma(\tilde{\chi})^{K_p} \cong \St_\chi$, 
    \item $\pi_{\tilde{\theta}}^{K_p} \cong R_{T'}^G(\theta)$.
\end{itemize}
In particular, if $\pi$ is an infinite-dimensional irreducible admissible representation of $\GL_2(\Q_p)$ such that $\pi^{K_p} \neq 0$, then $\pi^{K_p}$ belongs to a single $\cO[G]$-block.
\end{lemma}

Combining \cref{lem:reduce to Kp-fixed rep,lem:invariants of Qp reps} and recalling that any $0 \neq f \in S_k(\Gamma \cap \Gamma(p), F)^Z$ gives rise to an infinite-dimensional admissible representation $\pi_{f,p}$ of $\GL_2(\Q_p)$ with $\pi_{f,p}^{K_p} \neq 0$ gives the following corollary.

\begin{corollary}\label{cor:block for cusp form}
For any cuspidal anemic Hecke eigensystem $\lambda$ supported on $S_k(\Gamma \cap \Gamma(p), F)^Z$, the representation $M(\lambda)$ belongs to a single $\cO[G]$-block.
\end{corollary}

\subsection{The block associated with an anemic maximal ideal}\label{subsec:block of m'}
Now fix a maximal ideal $\m' \subset \T'(\Gamma \cap \Gamma(p))$.  For an anemic Hecke eigensystem $\lambda$ supported on $M_k(\Gamma \cap \Gamma(p),\cO)^Z_{\m'}$, let $\B_\lambda$ denote the unique block that contains $M(\lambda)$, which exists by \cref{lem:block for Eisenstein series,cor:block for cusp form}. In this section we show that $\B_\lambda$ is independent of $\lambda$ and depends only on $\m'$.

Recall that $\T'(\Gamma \cap \Gamma_0(p,p)) \cong \T'(\Gamma \cap \Gamma_0(p^2))$ via \eqref{eq:Gamma0(p,p) isoto Gamma0(p^2)}.  We abuse notation slightly and use $\m'$ to denote its image in $\T'(\Gamma \cap \Gamma_0(p,p))$ under the natural projection map or its further image in $\T'(\Gamma \cap \Gamma_0(p^2))$ under this isomorphism.  As in \cref{subsec:familiar mod forms}, let $\m$ denote the maximal ideal of $\T(\Gamma \cap \Gamma_0(p^2))$ generated by $\m'$ and $U_p$.  Recall from \cref{lem:anemic to full} that $M_k(\Gamma \cap \Gamma_0(p^2), \cO)_\m$ carries a natural action of $Z(\cO[G])$.

\begin{proposition}
\label{prop:block of m}
    There is a unique block idempotent $e \in Z(\cO[G])$ that acts as the identity on $M_k(\Gamma \cap \Gamma_0(p^2),\cO)_\m$, and all of the other block idempotents act as zero.
\end{proposition}
\begin{proof}
    This follows immediately from \cref{lem:indecomposble} and \cref{lem:indecomp2}. Indeed, since $M_k(\Gamma \cap \Gamma_0(p^2),\cO)_\m$ is a $\T_\m \times Z(\cO[G])$-module, as a $\T_\m$-module, it is the direct sum over the blocks $Z(\cO[G])$. Because it is an indecomposable $\T_\m$-module, exactly one of those blocks is nonzero.
\end{proof}

Let $e_{\m}$ denote the idempotent associated with $\m$ by \cref{prop:block of m}, and let $\B_\m$ be the corresponding block. 

\begin{corollary}\label{cor:block of m'}
The $G$-representation $M_k(\Gamma \cap \Gamma(p), \cO)_{\m'}^Z$ belongs to the block $\B_\m$.
\end{corollary}

\begin{proof}
By \cref{prop:block of m} and the isomorphism \eqref{eq:Gamma0(p,p) isoto Gamma0(p^2)}, a block idempotent $e$ acts on $f \in M_k(\Gamma \cap \Gamma(p), \cO)_{\m'}^T$ by $0$ unless $e = e_\m$, in which case it acts by $1$.  Since each $M(\lambda)$ belongs to a single block by \cref{lem:block for Eisenstein series,cor:block for cusp form}, it suffices to show that $M(\lambda)^T \neq 0$ for every anemic eigensystem $\lambda$.  When $\lambda$ is Eisenstein, this is immediate from the explicit description in \cref{lem:block for Eisenstein series}.  When $\lambda$ is cuspidal, its associated automorphic representation is one of $I(\tilde{\chi}), \sigma(\tilde{\chi}), \pi_{\tilde{\theta}}$ described in \cref{subsec:classify Kp-fixed reps}, all of which have conductor dividing $p^2$ (where conductor is defined as in \cite[Theorem 1]{Casselman}).  Therefore the same automorphic representation appears at level $\Gamma \cap \Gamma_0(p^2)$.  Once again by \eqref{eq:Gamma0(p,p) isoto Gamma0(p^2)}, it follows that $M(\lambda)^T \neq 0$, as desired.
\end{proof}

We now summarize how to determine the block $\B_\m$ explicitly starting from a Hecke eigenform in $M_k(\Gamma \cap \Gamma(p), \cO)_{\m'}^Z$.  In the statement of the following lemma, recall our notational convention in \cref{def:level Kp characters} for characters of level $K_p$.

\begin{lemma}
\label{lem:determining the block Bm}
Let $\m' \subset \T'(\Gamma \cap \Gamma(p))$ be a maximal ideal, let $0 \neq f \in M_k(\Gamma \cap \Gamma(p),\cO)^Z_{\m'}$ be a $\T'(\Gamma \cap \Gamma(p))$-eigenform.  If $f$ is cuspidal, let $\pi_{f,p}$ be the $p$-component of the automorphic representation of $f$.  If $f$ is Eisenstein, let $\chi_f$ be the associated character of $T$ as in \cref{subsec:G-rep for Eis series}.

The block $\B_\m$ is principal if and only if $\chi_f$ has $\ell$-power order or $\pi_{f,p}$ is one of the following:
\begin{itemize}
    \item unramified principal series,
    \item the special representation $\sigma(\mathbbm{1})$,
    \item ramified principal series $I(\tilde{\chi})$, where $\chi$ has $\ell$-power order,
    \item depth-zero supercuspidal $\pi_{\tilde{\theta}}$, where $\theta$ has $\ell$-power order.
\end{itemize}
The block $\B_\m$ is nilpotent if and only if we are in one of the following two cases:
\begin{itemize}
    \item $p \equiv 1 \pmod{\ell}$ and either $\chi_f^2$ does not have $\ell$-power order or $\pi_{f,p} \cong I(\tilde{\chi})$, where $\chi^2$ is not of $\ell$-power order,
    \item $p \equiv -1 \pmod{\ell}$ and $\pi_{f,p} \cong \pi_{\tilde{\theta}}$, where $\theta^2$ is not of $\ell$-power order.
\end{itemize}
\end{lemma}
\begin{proof}
By \cref{prop:block of m}, $\B_\m=\B_{\lambda_f}$ is the block containing the $G$-representation $M(\lambda_f)$ associated with the anemic eigensystem $\lambda_f$ of $f$. These representations are given by \cref{lem:block for Eisenstein series,lem:invariants of Qp reps}, and the corresponding blocks are described in \cref{sec: blocks}.
\end{proof}

\section{Module structure for level-raising congruences}\label{sec:main thms}
In this section we deduce our main theorem statements about the module structure of $M_k(\Gamma \cap \Gamma_0(p^2), \cO)_\m$ over a subring of $\cO[\Delta]$, where $\m$ is a maximal ideal of the full Hecke algebra that contains $U_p$.  The relevant subring is $\cO[\Delta]^+$ when $\B_\m$ is principal and $\cO[\Delta]$ when $\B_\m$ is nilpotent.  These structure results follow immediately from applying \cref{thm:X^T} to $M_k(\Gamma \cap \Gamma(p), \cO)_{\m'}^Z$, which is possible thanks to \cref{prop:projectivity of modular forms} and \cref{cor:block of m'}.  In \cref{subsec:eigenforms} we establish that $M_k(\Gamma \cap \Gamma_0(p^2), \cO)_\m$ is, in fact, a module over the relevant subring of $\cO[\Delta]$ and describe this action explicitly on eigenforms.  The next two sections give the module structure, first in the generic case and then in the special ``automorphically trivial'' case.

Throughout this section, fix $\Gamma = \Gamma_0(N)$ or $\Gamma_1(N)$ for some $\gcd(N, p\ell) = 1$ and $k \geq 2$.  Let $\m'$ be a maximal ideal of $\T'(\Gamma \cap \Gamma_0(p^2))$, which we also think of as a maximal ideal of $\T'(\Gamma \cap \Gamma(p))$ via the isomorphism \eqref{eq:Gamma0(p,p) isoto Gamma0(p^2)} and the natural projection map.  Let $\m$ be the maximal ideal of $\T(\Gamma \cap \Gamma_0(p^2))$ generated by $\m'$ and $U_p$.

\subsection{Action on eigenforms}
\label{subsec:eigenforms}

The next proposition establishes the existence of the action of the relevant subring of $\cO[\Delta]$ on $M_k(\Gamma \cap \Gamma_0(p^2), \cO)_\m$.  We then make this action explicit on eigenforms.

\begin{proposition}\label{prop:action existence}
The space $M_k(\Gamma \cap \Gamma_0(p^2), \cO)_\m$ carries a natural action of 
\[
\begin{cases}
    \cO[\Delta]^+ &\text{if } \B_\m \text{ is principal,}\\
    \cO[\Delta] &\text{if } \B_\m \text{ is nilpotent.}
\end{cases}
\]
\end{proposition}

\begin{proof}
Just as in the proof of \cref{thm:X^T}, we have an identification
\[
\ker(M^T \to M^B) \cong \Hom_{\cO[G]}(P,M)
\]
for a projective $e_\m\cO[G]$-module $P$ with $\End_{\cO[G]}(P)$ the relevant subring of $\cO[\Delta]$.
Then $M_k(\Gamma \cap \Gamma_0(p^2), \cO)_\m$ inherits this action via the isomorphism \eqref{eq:Gamma0(p,p) isoto Gamma0(p^2)} by \cref{lem:ker trace = ker Up}. Note that if $\ell>3$, then, by \cref{cor:block of m'} and \cref{prop:projectivity of modular forms}, $M$ is a finitely generated projective $e_\m\cO[G]$-module, and the action of $\cO[\Delta]$ or $\cO[\Delta]^+$ is the same as the action described in \cref{thm:X^T}. 
\end{proof}

Note that in the case of a nilpotent block, the action in \cref{prop:action existence} involves a choice of pinning character, as in \cref{sec:Nil blocks p=1,sec:nilpotent_blocks2}. Indeed, as noted in those sections, the identification of $\cO[\Delta]$ as the center of the block involves the choice of pinning character.

We now describe the action of $\cO[\Delta]^+$ or $\cO[\Delta]$ on eigenforms.  To do so, we establish some notation.  Linearizing a character $\chi$ of $\Delta$ defines an $\cO$-algebra homomorphism $\cO[\Delta] \to \cO$ that we continue to denote by $\chi$. Similarly, a pair $\{\chi,\chi^{-1}\}$ defines an $\cO$-algebra homomorphism $\cO[\Delta]^+\to \cO$ that we denote by $\chi^+$, given by 
\[
\chi^{+}([\delta]+[\delta^{-1}])= \chi(\delta)+\chi(\delta^{-1})
\]
for $\delta\in\Delta$. In either case, if $\chi=\mathbbm{1}$ is the trivial character, then this homomorphism is the augmentation. Recall \cref{def:level Kp characters} for the definition and notational convention for characters of level $K_p$.

\begin{proposition}\label{prop:action on eigenforms}
    Let $f \in M_k(\Gamma \cap \Gamma_0(p^2),\cO)_\m$ be an eigenform.  If $f$ is cuspidal, let $\pi_{f,p}$ be the $p$-component of its automorphic representation.  If $f$ is Eisenstein, let $\chi_f$ be the associated character of $T$ as in \cref{subsec:G-rep for Eis series}.  Let $\alpha \in \cO[\Delta]^+$ if $\B_\m$ is principal and $\alpha \in \cO[\Delta]$ if $\B_\m$ is nilpotent.  The following describes the action of $\alpha$ on $f$ given by \cref{prop:action existence}.
    \begin{enumerate}
        \item Assume $p \equiv 1 \pmod{\ell}$.
        \begin{enumerate}
            \item If $\B_\m$ is the principal block, then $\pi_{f,p}$ is either principal series for a level-$K_p$ character $\tilde{\chi}$ or special. Then
            \[
            \alpha f = \begin{cases}
                \chi^+(\alpha)f & \text{if }\pi_{f,p} \cong I(\tilde{\chi}), \\
                \mathbbm{1}^+(\alpha)f &\text{if }\pi_{f,p} \text{ is special,}\\
                \chi_f^+(\alpha)f & \text{if }f \text{ is Eisenstein.}
            \end{cases}
            \]
            \item If $\B_\m$ is nilpotent with pinning character $\chi$, then $\pi_{f,p}\cong I(\tilde{\psi})$ is principal series for a level-$K_p$ character $\tilde{\psi}$ such that $\psi \equiv \chi \pmod \varpi$.  Similarly, when $f$ is Eisenstein there is $\psi \in \{\chi_f, \chi_f^{-1}\}$ such that $\psi \equiv \chi \pmod \varpi$.  Then
            \[
            \alpha f = \psi(\alpha) f.
            \]
        \end{enumerate}
        \item Assume $p \equiv -1 \pmod{\ell}$.
        \begin{enumerate}
            \item If $\B_\m$ is the principal block, then $\pi_{f,p}$ is either unramified principal series, special, or depth-zero supercuspidal for a level-$K_p$ character $\tilde{\theta}$. Then
            \[
            \alpha f = \begin{cases}
                \theta^+(\alpha)f & \text{if }\pi_{f,p} \cong \pi_{\tilde{\theta}}, \\
                \mathbbm{1}^+(\alpha)f &\text{if }\pi_{f,p} \text{ is not supercuspidal,}\\
                \chi_f^+(\alpha)f & \text{if } f \text{ is Eisenstein.}
                \end{cases}
            \]
            \item\label{item:vexing action} If $\B_\m$ is nilpotent with pinning character $\theta$, then $f$ is necessarily cuspidal and $\pi_{f,p}\cong \pi_{\tilde{\psi}}$ is depth-zero supercuspidal for a level-$K_p$ character $\tilde{\psi}$ such that $\psi \equiv \theta \pmod \varpi$. Then
            \[
            \alpha f = \psi(\alpha) f.
            \]
        \end{enumerate}
    \end{enumerate}
\end{proposition}
\begin{proof}
The description of the blocks is given by \cref{lem:determining the block Bm}, and the determination of the representation of $G$ associated with the eigensystem of $f$ is given by \cref{lem:block for Eisenstein series} in the Eisenstein case and \cref{lem:invariants of Qp reps} in the cuspidal case. The result follows from the explicit description of the maps $\cO[\Delta]^+ \to Z(e\cO[G])$ for the principal blocks, given by \eqref{eq:f_alpha principal p=1} and \eqref{eq: f alpha}, and of the maps $\cO[\Delta] \to Z(e\cO[G])$ for the nilpotent blocks, given by \eqref{eq:p=1 falpha} and \eqref{eq:p=-1 falpha}.
\end{proof}

\subsection{Exceptional maximal ideals}\label{sec:exceptional max}
\cref{prop:projectivity of modular forms} shows that $S_k(\Gamma \cap \Gamma(p),\cO)^Z$ is a projective $\cO[G]$-module for all $k>2$. For $k=2$, the obstruction in the proof comes from the fact that $H^*(X(\Gamma \cap \Gamma(p)),\Omega^1)^Z$ is not concentrated in degree zero. This obstruction vanishes locally at most maximal ideals in the Hecke algebra.

\begin{lemma}
\label{lem:generic projectivity}
    Let $\m' \subset \T'(\Gamma \cap \Gamma(p))$ be such that $H^1(X(\Gamma \cap \Gamma(p))/Z,\Omega^1)_{\m'}=0$. Then $S_2(\Gamma \cap \Gamma(p),\cO)^Z_{\m'}$ is a projective $\cO[G]$-module.
\end{lemma}
\begin{proof}
    To simplify notation, let $X=X(\Gamma \cap \Gamma(p))/Z$ and let $C\subset X$ be the cusps. Consider the exact sequence of sheaves on $X$
    \[
    0 \to \Omega_X^1 \to \Omega_X^1(\log) \to \Omega_{C} \to 0,
    \]
    where $\Omega_{C}$ is a skyscraper sheaf on $C$. Taking the long exact sequence in cohomology yields an exact sequence
    \[
    0 \to H^0(X,\Omega^1_X) \to H^0(X,\Omega^1_X(\log)) \to H^0(X,\Omega_C) \to H^1(X,\Omega_X^1) \to H^1(X,\Omega^1_X(\log))
    \]
    Since invariant global sections are the same as global sections on the quotient, we have
    \[
    H^0(X,\Omega^1_X)=H^0(X(\Gamma \cap \Gamma(p)),\Omega^1)^Z=S_2(\Gamma \cap \Gamma(p),\cO)^Z,
    \]
    and similarly for modular forms. Moreover, $H^1(X,\Omega^1_X(\log))=0$ (just as in the proof of \cref{prop:projectivity of modular forms}), so the above sequence simplifies to the residue sequence
    \begin{equation}
        \label{eq:residue sequence}
            0 \to S_2(\Gamma \cap \Gamma(p),\cO)^Z \to M_2(\Gamma \cap \Gamma(p),\cO)^Z \to H^0(X,\Omega_C) \to H^1(X,\Omega^1_X) \to 0.
    \end{equation}
    By assumption, localizing at $\m'$ annihilates the last term. 
    Since $M_2(\Gamma \cap \Gamma(p),\cO)_{\m'}$ is $\cO[G]$-projective by \cref{prop:projectivity of modular forms} and $H^0(X,\Omega_C)^Z_{\m'}$ is $\cO[G]$-projective by \cref{thm:nakajima} (or by inspection), this implies that $S_2(\Gamma \cap \Gamma(p),\cO)^Z_{\m'}$ is $\cO[G]$-projective.
\end{proof}

We call a maximal ideal $\m' \subset \T'(\Gamma \cap \Gamma(p))$ \emph{exceptional} if it is in the support of $H^1(X(\Gamma \cap \Gamma(p))/Z,\Omega^1)$, and \emph{generic} otherwise. The module $H^1(X(\Gamma \cap \Gamma(p))/Z,\Omega^1)$ has rank two, corresponding to the two geometric connected components of $X(\Gamma \cap \Gamma(p))/Z$. Computing the Hecke action, one finds that there are two exceptional maximal ideals:
\begin{itemize}
    \item the \emph{automorphically trivial} maximal ideal $\m'$, with $T_q-(q+1) \in \m'$ for all $q \nmid Np$ and $U_q-q \in \m'$ for all $q \mid N$, and
    \item the \emph{automorphically quadratic} maximal ideal $\m'$, with $T_q-\eta(q)(q+1) \in \m'$ for all $q \nmid Np$ and $U_q-\eta(q)q \in \m'$ for all $q \mid N$, where $\eta:(\Z/p\Z)^\times \to \cO^\times$ is the nontrivial quadratic character.
\end{itemize}
The localization $H^1(X(\Gamma \cap \Gamma(p))/Z,\Omega^1)_{\m'}$ has rank $1$, with $G$ acting trivially or by $\eta$ in the respective cases.
The automorphically trivial case belongs to the principal block, and we discuss it in detail in \cref{sec:exceptional case} below. The automorphically quadratic case belongs to the block containing $\eta$ and can be treated similarly, but we do not discuss it further.

\subsection{Module structure in the generic case} 
We now use \cref{thm:X^T} to describe $M_k(\Gamma \cap \Gamma_0(p^2), \cO)_\m$ as a module over either $\cO[\Delta]^+$ or $\cO[\Delta]$, with the action given by \cref{prop:action existence}. We also describe the structure of the space of cuspforms $S_k(\Gamma \cap \Gamma_0(p^2), \cO)_\m$, except when $k=2$ and $\m'$ is exceptional. The case where $k=2$ and $\m'$ is automorphically trivial is treated in the next section.

\begin{theorem}\label{thm:modular forms count}
    Let $k \ge 2$ and assume that $\ell>3$ or that $\ell=3$ and $\Gamma$ is rigid. Let $\m' \subset \T'(\Gamma \cap \Gamma_0(p^2))$ be a maximal ideal and $\m \subset \T(\Gamma \cap \Gamma_0(p^2))$ be the maximal ideal containing $\m'$ and $U_p$. 
    Define integers $a_0$ and $a_1$ by
    \begin{align*}
            a_0 &=\rank_\cO M_k(\Gamma,\cO)_{\m'}, \\
            a_1 & =\rank_\cO M_k(\Gamma \cap \Gamma_0(p),\cO)_{\m'} - 2 a_0.
    \end{align*}
    \begin{enumerate}
        \item Assume $p \equiv 1 \pmod{\ell}$.
        \begin{enumerate}
            \item \label{part:modforms p=1 principal} If $\B_\m$ is principal, then there is an isomorphism of $\cO[\Delta]^+$-modules
            \[
            M_k(\Gamma \cap \Gamma_0(p^2),\cO)_\m \cong \cO[\Delta]^{a_0} \oplus (\cO[\Delta]^+)^{a_1}.
            \]
            Moreover, if $a_1=0$, then this is an isomorphism of $\cO[\Delta]$-modules.
            \item\label{part:modforms p=1 nilpotent} If $\B_\m$ is nilpotent, then $M_k(\Gamma \cap \Gamma_0(p^2),\cO)_\m$ is a free $\cO[\Delta]$-module.
        \end{enumerate}
        \item Assume $p \equiv -1 \pmod{\ell}$.
        \begin{enumerate}
            \item\label{part:modforms p=-1 principal} If $\B_\m$ is principal, then there is an isomorphism of $\cO[\Delta]^+$-modules
            \[
            M_k(\Gamma \cap \Gamma_0(p^2),\cO)_\m \cong \cO^{a_0} \oplus (\cO[\Delta]^+)^{a_1}.
            \]
            \item\label{item:vexing} If $\B_\m$ is nilpotent, then $M_k(\Gamma \cap \Gamma_0(p^2),\cO)_\m$ is a free $\cO[\Delta]$-module.
        \end{enumerate}
    \end{enumerate}
Moreover, if $k>2$, or if $k=2$ and $\m'$ is generic, then exactly analogous statements hold with modular forms replaced by cuspforms.
\end{theorem}
\begin{proof}
By \cref{prop:projectivity of modular forms}, $M_k(\Gamma \cap \Gamma(p),\cO)^Z$ is a projective $\cO[G]$-module. Since the actions of $\cO[G]$ and $\T'(\Gamma \cap \Gamma(p))$ on it commute by \cref{lem:Hecke alg properties}\eqref{item: G-commute}, the $\T'(\Gamma \cap \Gamma(p))$-direct summand $M_k(\Gamma \cap \Gamma(p),\cO)^Z_{\m'}$ is also $\cO[G]$-projective, and it is a module over $e_\m\cO[G]$ by \cref{cor:block of m'}. The theorem now follows by applying \cref{thm:X^T} to $M=M_k(\Gamma \cap \Gamma(p),\cO)^Z_{\m'}$ and identifying the terms $M^G$, $M^B$ and $\ker(M^T \xrightarrow{\tr} M^B)$ with the appropriate spaces of modular forms. We have $M^G=M_k(\Gamma,\cO)_{\m'}, M^B=M_k(\Gamma \cap \Gamma_0(p),\cO)_{\m'}$, and there is an isomorphism
\[
\ker(M^T \xrightarrow{\tr} M^B) \cong M_k(\Gamma \cap \Gamma_0(p^2),\cO)_{\m}
\]
by \cref{lem:ker trace = ker Up,lem:indecomposble}. Therefore, \cref{thm:X^T} implies \eqref{part:modforms p=1 nilpotent} and \eqref{item:vexing} directly. In the setting of \eqref{part:modforms p=1 principal}, \cref{thm:X^T} implies
\[
M_k(\Gamma \cap \Gamma_0(p^2),\cO)_{\m} \cong (\cO[\Delta]^-)^a \oplus (\cO[\Delta]^+)^b,
\]
where $a= \rank M^G =a_0$ and $a+b=\rank M^B=2a_0+a_1$. Since there are two stabilizations of the $p$-oldforms in $M^B$, we have $a_1 \ge 0$, so this implies
\[
M_k(\Gamma \cap \Gamma_0(p^2),\cO)_{\m} \cong \cO[\Delta]^a \oplus (\cO[\Delta]^+)^{b-a} = \cO[\Delta]^{a_0} \oplus (\cO[\Delta]^+)^{a_1},
\]
proving \eqref{part:modforms p=1 principal}. In the setting of \eqref{part:modforms p=-1 principal}, \cref{thm:X^T} implies
\[
M_k(\Gamma \cap \Gamma_0(p^2),\cO)_{\m} \cong \cO^a \oplus (\cO[\Delta]^+)^b,
\]
where $a=\rank M^G=a_0$ and $2a+b=\rank M^B=2a_0+a_1$, so $b=a_1$, proving \eqref{part:modforms p=-1 principal}.

For the final sentence, the same argument works for cuspforms because $S_k(\Gamma \cap \Gamma(p),\cO)^Z_{\m'}$ is a projective $\cO[G]$-module under the stated hypothesis by \cref{prop:projectivity of modular forms} and \cref{lem:generic projectivity}.
\end{proof}
\begin{remark}
    By \cref{rem:projectivity for p=3}, the theorem also holds for $\ell=3$ if $\Gamma$ is not rigid, so long as there is a prime $q \nmid Np$ with $q \equiv -1 \pmod3$ with $T_q \notin \m'$ (that is, so long as the mod-$3$ Galois representation associated with $\m'$ is not induced from $\Q(\zeta_3)$).
\end{remark}

\subsection{The exceptional automorphically-trivial case}
\label{sec:exceptional case}
We now treat the special case where $k=2$ and $\m'$ is the automorphically trivial maximal ideal with $T_q -q-1 \in \m'$ for all $q \nmid Np$ and $U_q-q \in \m'$ for $q \mid N$. 
\begin{theorem}
\label{thm:cuspforms count eis case}
Let $k=2$ and assume that $\ell>3$ or that $\ell=3$ and $\Gamma$ is rigid. Let $\m' \subset \T'(\Gamma \cap \Gamma_0(p^2))$ be the automorphically-trivial maximal ideal, and let $\m \subset \T(\Gamma \cap \Gamma_0(p^2))$ be the maximal ideal containing $\m'$ and $U_p$. Let  
\begin{align*}
    a_0 &=\rank S_2(\Gamma,\cO)_{\m'}, \\
    a_1 &=\rank S_2(\Gamma \cap \Gamma_0(p),\cO)_{\m'}-2a_0.
\end{align*}

Then $S_2(\Gamma \cap \Gamma_0(p^2), \cO)_\m$ has a natural action of $\cO[\Delta]^+$, and, moreover,
\begin{enumerate}
    \item if $p \equiv 1 \pmod{\ell}$, then there is an exact sequence of $\cO[\Delta]^+$-modules
    \[
    0 \to S_2(\Gamma \cap \Gamma_0(p^2), \cO)_\m \to \cO[\Delta]^{a_0} \oplus (\cO[\Delta]^+)^{a_1} \to \cO[\Delta]^- \to 0.
    \]
    \item if $p \equiv -1 \pmod{\ell}$, then there is an exact sequence of $\cO[\Delta]^+$-modules
    \[
    0 \to S_2(\Gamma \cap \Gamma_0(p^2), \cO)_\m \to \cO^{a_0} \oplus (\cO[\Delta]^+)^{a_1+1} \to \cO \to 0.
    \]
\end{enumerate}
\end{theorem}

The rest of the section is devoted to the proof of this theorem. The proof is similar to the proof of \cref{thm:modular forms count}, in that it relies on \cref{prop:projectivity of modular forms} applied to $S_2(\Gamma \cap \Gamma(p),\cO)^Z$, but the difference is that, since this module is not $\cO[G]$-projective, we cannot apply \cref{thm:X^T}.  Instead we use ideas from \cref{sec:block algebras} similar to those used in the proof of \cref{thm:X^T}. If $e$ is the principal-block idempotent, then the ring $e \cO[G]$ has two isomorphism classes of indecomposable  projective modules that we denote by $P_1$ and $P$.  As in \cref{sec:block algebras}, $P_1$ is (a lift of) the projective envelope of the trivial representation, and $P$ is $P_\St$ if $p \equiv 1 \pmod{\ell}$ or $P_\pi$ if $p \equiv -1 \pmod{\ell}$, where $\pi$ is a cuspidal representation in the principal block.

\begin{lemma}
\label{lem:S_2 injective res}
    There is an exact sequence of $e\cO[G]$-modules
    \[
    0 \to S_2(\Gamma \cap \Gamma(p),\cO)_{\m'}^Z \to P_1^a \oplus P^b \to P_1 \to \cO \to 0
    \]
    for some nonnegative integers $a$ and $b$.
\end{lemma}
\begin{proof}
    It is enough to show that there is an exact sequence
     \[
    0 \to S_2(\Gamma \cap \Gamma(p),\cO)_{\m'}^Z \to \tilde C^0 \to \tilde C^1 \to \cO \to 0
    \]
    of $e\cO[G]$-modules, where $\tilde C^i$ are projective. Indeed, since every projective $e\cO[G]$-module is a direct sum of copies of $P$ and $P_1$, and since $\Hom_{\cO[G]}(P,\cO)=0$, an easy homological-algebra argument reduces this sequence to a sequence of the required form.

    Since $H^1(X(\Gamma \cap \Gamma(p))/Z,\Omega^1)_{\m'} \cong \cO$, the required sequence is obtained by localizing the residue sequence \eqref{eq:residue sequence} at $\m'$.
\end{proof}

\begin{lemma}
\label{lem:S_2 determine a and b}
    Let $a$ and $b$ be integers as in \cref{lem:S_2 injective res}. Then $a=\rank S_2(\Gamma,\cO)_{\m'}$ and
    \[
    b = \begin{cases}
        \rank  S_2(\Gamma \cap \Gamma_0(p),\cO)_{\m'}-a & \text{if }p \equiv 1 \pmod{\ell} \\
        \rank S_2(\Gamma \cap \Gamma_0(p),\cO)_{\m'} -2a+1 & \text{if }p \equiv -1 \pmod{\ell}.
    \end{cases}
    \]
\end{lemma}
\begin{proof}
Let $K=\ker(P_1 \to \cO)$, so that, by \cref{lem:S_2 injective res}, there is an exact sequence
\begin{equation}
\label{eq:S_2 injective res}
  0 \to S_2(\Gamma \cap \Gamma(p),\cO)_{\m'}^Z \to P_1^a \oplus P^b \to K \to 0.  
\end{equation}
Applying $G$-invariants and noting that $(S_2(\Gamma \cap \Gamma(p),\cO)_{\m'}^Z)^G=S_2(\Gamma,\cO)_{\m'}$, $P_1^G=\cO$ and $P^G=0$, there is an exact sequence
\[
0 \to S_2(\Gamma,\cO)_{\m'} \to \cO^a \to K^G.
\]
Applying $G$-invariants to the exact sequence $0 \to K \to P_1 \to \cO \to 0$ yields
\[
0 \to K^G \to \cO \to \cO \to H^1(G,K) \to 0.
\]
Since $H^1(G,K)$ is finite, this implies $K^G=0$, so there is an isomorphism $S_2(\Gamma,\cO)_{\m'} \cong \cO^a$, giving the desired formula for $a$.

Now assume that $p \equiv 1 \pmod{\ell}$. Note that $(S_2(\Gamma \cap \Gamma(p),\cO)_{\m'}^Z)^B=S_2(\Gamma \cap \Gamma_0(p),\cO)_{\m'}$, and that, since $(-)^B=\Hom_{\cO[G]}(1 \oplus \St,-)$, we have $P_1^B=\cO$ and $P^B=\cO$. Therefore, applying $B$-invariants to \eqref{eq:S_2 injective res} yields an exact sequence
\[
0 \to S_2(\Gamma \cap \Gamma_0(p),\cO)_{\m'} \to \cO^{a+b} \to K^B.
\]
A similar argument to the above proof that $K^G=0$ shows that $K^B=0$ in this case, so this gives the desired formula for $b$.

Finally, assume that $p \equiv -1 \pmod{\ell}$. In this case $(-)^B=\Hom_{e\cO[G]}(P_1,-)$ is an exact functor, and applying it to  \eqref{eq:S_2 injective res} yields an exact sequence
\[
0 \to S_2(\Gamma \cap \Gamma_0(p),\cO)_{\m'} \to \End_{\cO[G]}(P_1)^a \oplus \Hom_{\cO[G]}(P_1,P)^b \to \Hom_{\cO[G]}(P_1,K) \to 0.
\]
Since $\rank \End_{\cO[G]}(P_1)=2$, this implies that
\[
2a+b = \rank S_2(\Gamma \cap \Gamma_0(p),\cO)_{\m'} + \rank \Hom_{\cO[G]}(P_1,K).
\]
It remains to show $\rank \Hom_{\cO[G]}(P_1,K) =1$.
Applying the same functor to the exact sequence $0 \to K \to P_1 \to \cO \to 0$ yields the exact sequence
\[
0 \to \Hom_{\cO[G]}(P_1,K) \to \End_{\cO[G]}(P_1) \to \cO \to 0.
\]
Since $\rank \End_{\cO[G]}(P_1)=2$, this implies $\rank \Hom_{\cO[G]}(P_1,K) =1$, as desired.
\end{proof}

\begin{proof}[Proof of \cref{thm:cuspforms count eis case}]
Let $S=S_2(\Gamma \cap \Gamma_0(p^2),\cO)_\m$. Just as in the proofs of \cref{thm:X^T,thm:modular forms count}, the space $S$ can be identified with $\Hom_{\cO[G]}(P,S_2(\Gamma \cap \Gamma(p),\cO)_{\m'}^Z)$. Since $\End_{\cO[G]}(P) \cong \cO[\Delta]^+$, this provides the claimed $\cO[\Delta]^+$-module structure on $S$.
Applying the exact functor $\Hom_{\cO[G]}(P,-)$ to the exact sequence in \cref{lem:S_2 injective res}, and noting that $\Hom_{\cO[G]}(P,\cO) =0$, yields an exact sequence
\[
0 \to S \to \Hom_{\cO[G]}(P,P_1)^a \oplus \End_{\cO[G]}(P)^b \to \Hom_{\cO[G]}(P,P_1) \to 0.
\]
The hom groups are identified with $\cO[\Delta]^\pm$ if $p \equiv 1 \pmod{\ell}$ by \eqref{eq:Homs P_1 and P_st} and with $\cO$ or $\cO[\Delta]^+$ if $p \equiv -1 \pmod{\ell}$ by \cref{lem:end of P_1,lem: end of P_pi,lem:hom P1 P_pi}. Then \cref{lem:S_2 determine a and b} gives the formula for $a$ and $b$, completing the proof.
\end{proof}

\section{Counting newforms}\label{sec:counting newforms}

In this section, we use \cref{thm:modular forms count,thm:cuspforms count eis case} to count the dimension of the space of $p$-newforms of level $\Gamma \cap \Gamma_0(p^2)$. In the case that $\Gamma=\Gamma_0(N)$ for an integer $N$, we use this count inductively, together with an inclusion-exclusion-type argument, to count the dimension of the space of newforms. 

We retain the same notation as in the previous section.  In particular, $k \ge 2$ is an integer, $\Gamma$ is a level structure of the form $\Gamma_0(N)$ or $\Gamma_1(N)$ with $\gcd(N,\ell p)=1$, and we assume that either $\ell > 3$ or $\Gamma$ is rigid.  As usual, $\m' \subset \T'(\Gamma \cap \Gamma_0(p^2))$ is a maximal ideal in the anemic Hecke algebra, $\m \subset \T(\Gamma \cap \Gamma_0(p^2))$ is the maximal ideal containing $\m'$ and $U_p$, and $\B_\m$ is the block associated with $\m'$. 

In this section, we focus on maximal ideals $\m'$ such that $\B_\m$ is the principal block, so that $\m'$ is either generic or automorphically trivial. Define an integer $\delta$ by
\[
\delta =\begin{cases}
    1 & \text{if }k=2 \text{ and }\m'\text{ is automorphically trivial}, \\
    0 & \text{otherwise}.
\end{cases}
\]
In other words, $\delta$ is a Kronecker delta symbol that detects whether or not we are in the exceptional case. The formulae for the dimension of the space of $p$-new cuspforms depend on $\delta$, which reflects the difference between \cref{thm:modular forms count} (for $\delta=0$) and \cref{thm:cuspforms count eis case} (for $\delta=1$). We exploit this dependence on $\delta$ to prove the existence of newforms in some cases when $\delta=1$ and use Hida theory to deduce related results for $k>2$ with $k \equiv 2 \pmod{\ell-1}$.

\subsection{Counting \texorpdfstring{$p$}{}-newforms}
\cref{prop:action on eigenforms} describes the part of $M_k(\Gamma \cap \Gamma_0(p^2),\cO)_\m$ on which $\cO[\Delta]$ or $\cO[\Delta]^+$ acts trivially. For instance, in the nilpotent case, the part of $M_k(\Gamma \cap \Gamma_0(p^2),\cO)_\m$ on which the $\cO[\Delta]$-action is trivial is exactly the part that is minimally ramified at $p$, and the rank of $M_k(\Gamma \cap \Gamma_0(p^2), \cO)_\m$ as an $\cO[\Delta]$-module is the $\cO$-rank of this $p$-minimal part.

In the principal case, the part on which $\cO[\Delta]^+$ acts trivially is the $p$-old part. Let $M_k(\Gamma \cap \Gamma_0(p^2),\cO)^{p\dnew}$ be the saturation of the span of the cuspidal eigenforms that have conductor $p^2$ at $p$ (that is, those that are either supercuspidal or ramified principal series at $p$) together with the Eisenstein series whose prime-to-$p$ eigensystems do not arise in lower level.

\begin{corollary}\label{cor:newform count}
    Suppose that $\B_\m$ is the principal block. Then $M_k(\Gamma \cap \Gamma_0(p^2),\cO)_\m^{p\dnew}$ is the kernel of the augmentation map
    \[
    M_k(\Gamma \cap \Gamma_0(p^2),\cO)_\m \to M_k(\Gamma \cap \Gamma_0(p^2), \cO)_\m \otimes_{\cO[\Delta]^+} \cO,
    \]
    and the analogous statement holds for cuspforms.
    
    Moreover,
    \[
    \frac{2}{|\Delta|-1} \rank M_k(\Gamma \cap \Gamma_0(p^2),\cO)_\m^{p\dnew} = \begin{cases}
        \rank M_k(\Gamma \cap \Gamma_0(p),\cO)_{\m'} & \text{if }p \equiv 1 \pmod{\ell} \\
        \rank M_k(\Gamma \cap \Gamma_0(p),\cO)_{\m'}^{p\dnew} & \text{if }p \equiv -1 \pmod{\ell}.
    \end{cases}
    \]
and
    \[
    \frac{2}{|\Delta|-1} \rank S_k(\Gamma \cap \Gamma_0(p^2),\cO)_\m^{p\dnew} = \begin{cases}
        \rank S_k(\Gamma \cap \Gamma_0(p),\cO)_{\m'}-\delta & \text{if }p \equiv 1 \pmod{\ell} \\
        \rank S_k(\Gamma \cap \Gamma_0(p),\cO)_{\m'}^{p\dnew}+\delta & \text{if }p \equiv -1 \pmod{\ell}.
    \end{cases}
    \]
\end{corollary}
\begin{proof}
For both modular forms and cuspforms, the kernel of the augmentation and the space of $p$-newforms contain the same eigenforms by \cref{prop:action on eigenforms}. Since they are both $\cO$-saturated in the full space, they are both equal to the saturation of the span of those eigenforms.

Now assume that $p \equiv 1 \pmod{\ell}$. By \cref{thm:modular forms count},
\[
M_k(\Gamma \cap \Gamma_0(p^2),\cO)_\m \cong (\cO[\Delta])^{a_0} \oplus (\cO[\Delta]^+)^{a_1},
\]
and the kernel of the augmentation has $\cO$-rank equal to
\[
a_0(|\Delta|-1) + a_1 \frac{|\Delta|-1}{2} = (a_1+2a_0)  \frac{|\Delta|-1}{2} = \rank M_k(\Gamma \cap \Gamma_0(p),\cO)_{\m'} \cdot \frac{|\Delta|-1}{2}.
\]
If $p \equiv -1 \pmod{\ell}$, then
\[
M_k(\Gamma \cap \Gamma_0(p^2),\cO)_\m \cong \cO^{a_0} \oplus (\cO[\Delta]^+)^{a_1},
\]
and the kernel of the augmentation has $\cO$-rank equal to $a_1 \frac{|\Delta|-1}{2}$, where $a_1 = \rank M_k(\Gamma \cap \Gamma_0(p),\cO)_{\m'}^{p\dnew}$.

For cuspforms, if $\delta =0$, then the proof is exactly the same, so assume $\delta=1$. If $p \equiv 1 \pmod{\ell}$, then by \cref{thm:cuspforms count eis case}, 
as a virtual $\cO[\Delta]^+$-module, $S_2(\Gamma \cap \Gamma_0(p^2),\cO)_\m$ is isomorphic to
\[
\cO[\Delta]^{a_0-1} \oplus (\cO[\Delta]^+)^{a_1+1}.
\]
Thus the kernel of the augmentation has $\cO$-rank
\[
(2(a_0-1)+(a_1+1)) \frac{|\Delta|-1}{2} = (2a_0+a_1-1)\frac{|\Delta|-1}{2}.
\]
Since $2a_0+a_1=\rank S_k(\Gamma \cap \Gamma_0(p),\cO)_{\m'}$, this gives the desired result. If $p \equiv -1 \pmod{\ell}$, then by \cref{thm:cuspforms count eis case},
as a virtual $\cO[\Delta]^+$-module, $S_2(\Gamma \cap \Gamma_0(p^2),\cO)_\m$ is isomorphic to
\[
\cO^{a_0-1} \oplus (\cO[\Delta]^+)^{a_1+1}.
\]
Thus the kernel of the augmentation has $\cO$-rank
\[
(a_1+1) \frac{|\Delta|-1}{2} = (\rank S_k(\Gamma \cap \Gamma_0(p),\cO)_{\m'}^{p\dnew}+1) \frac{|\Delta|-1}{2}.\qedhere
\]
\end{proof}

\subsection{Counting newforms}
We now specialize to the case when $\Gamma=\Gamma_0(N)$ for an integer $N$.  Since this is never rigid, we now assume that $\ell > 3$. Since we are only concerned with this type of level, we simplify the notation as follows. For an integer $M$, let $M_k(M) = M_k(\Gamma_0(M), \cO)$ and $S_k(M)=S_k(\Gamma_0(M),\cO)$.   For a divisor $d \mid M$, let $\T(M)^{(d)}$ be the algebra of endomorphisms of $M_k(M)$ generated by $T_n$ for $\gcd(d,n)=1$. Note that for a maximal ideal $\n' \subset \T(Np^2)^{(Np)}$, the localization $\T(Np^2)^{(p)}_{\n'}$ is the product of the local rings $\T(Np^2)^{(p)}_{\m'}$, where $\m'$ runs over maximal ideals $\m'$ of $\T(Np^2)^{(p)}$ that contain $\n'$. Hence we may apply the results of the previous section to each of the ideals $\m'$ separately in order to obtain an analogous result for the $\n'$.  We use this fact repeatedly in this section without further comment.  Moreover, we have the following.

\begin{lemma}\label{lem:block indep of m' over n'}
Let $\m_1', \m_2' \subset \T(Np^2)^{(p)}$ be maximal ideals containing a maximal ideal $\n'$ of $\T(Np^2)^{(Np)}$.  Then $\B_{\m_1} = \B_{\m_2}$.
\end{lemma}

\begin{proof}
Let $I \subset \T(Np^2)^{(Np)}_{\n'}$ be a height-one prime ideal. Then, by the going-up theorem, for $i=1,2$, there are height-one prime ideals $I_i \subset \T(Np^2)^{(p)}_{\m_i'}$ lying over $I$. These ideals correspond to eigenforms $f_i \in M_k(Np^2)_{\m_i'}$ that have the same $\T(Np^2)^{(Np)}$-eigensystem. By \cref{lem:determining the block Bm}, the blocks $\B_{\m_1}$ and $\B_{\m_2}$ are determined by the $\T(Np^2)^{(Np)}$-eigensystems of $f_1$ and $f_2$, respectively, and hence are equal.
\end{proof}

In light of \cref{lem:block indep of m' over n'}, for any maximal ideal $\n' \subset \T(Np^2)^{(Np)}$ we write $\B_{\n'} \coloneqq \B_{\m'}$ for any maximal ideal $\m' \subset \T(Np^2)^{(p)}$ containing $\n'$. We say that $\n'$ is \emph{automorphically trivial} if it is contained in the automorphically-trivial maximal ideal, and we define $\delta$ analogously by
\[
\delta = \begin{cases}
    1 & \text{if }k=2\text{ and }\n'\text{ is automorphically trivial,} \\
    0 & \text{otherwise.}
\end{cases}
\]

Before proceeding, we require some notation and counting results for newforms. Let $S_k(N)^{\new}$ denote the space of newforms. For each pair of integers $n$, $M$ such that $M \mid N$ and $n \mid \frac{N}{M}$, let 
\[
\iota_n: S_k(M) \to S_k(N)
\]
be the map given by $f(z) \mapsto f(nz)$. Recall from the theory of newforms (see \cite[Theorem 5.8.3]{DiamondShurman}) that
\[
S_k(N,F) = \bigoplus_{M \mid N} \bigoplus_{n \mid \frac{N}{M}} \iota_n(S_k(M,F)^{\new}).
\]
If $N=N'd$ with $\gcd(N',d)=1$, let 
\begin{equation}
\label{eq:definition of dnew}
S_k(N,F)^{d\dnew} = \bigoplus_{M \mid N'}\bigoplus_{n \mid \frac{N'}{M}} \iota_n(S_k(Md,F)^{\new}).
\end{equation}

\begin{lemma}\label{lem:Rob's variant}
Let $a, b, c$ be positive integers such that $\gcd(a, bc) = 1 = \gcd(b, c)$, and let $\n' \subset \T(abc)^{(abcp)}$ be a maximal ideal.  Then
\[
\rank S_k(abc)^{c\dnew}_{\n'} = \sum_{d | a} \tau(a/d)\rank S_k(dbc)^{dc\dnew}_{\n'},
\]
where $\tau(x)$ denotes the number of positive divisors of $x$.  
\end{lemma}

\begin{proof}
Since the spaces of modular forms have no $\ell$-torsion, it suffices to prove the formula over $F$.  We work over $F$ but suppress it from the notation.  

Note that $\iota_{y_1y_2} = \iota_{y_1} \circ \iota_{y_2}$.  Using \eqref{eq:definition of dnew} and the assumptions about greatest common divisors, we have
\begin{align*}
S_k(abc)^{c\dnew} &= \bigoplus_{x | ab} \bigoplus_{y | \frac{ab}{x}} \iota_y(S_k(xc)^{\new}) = \bigoplus_{d|a} \bigoplus_{e|b} \bigoplus_{u | \frac{a}{d}} \bigoplus_{v | \frac{b}{e}} \iota_{uv}(S_k(dec)^{\new})\\
&= \bigoplus_{d | a} \bigoplus_{u | \frac{a}{d}} \iota_u(\bigoplus_{e | b} \bigoplus_{v | \frac{b}{e}} \iota_v(S_k(dec)^{\new})) = \bigoplus_{d | a} \bigoplus_{u | \frac{a}{d}} \iota_u(S_k(dbc)^{dc\dnew}).
\end{align*}
The direct sums above are direct sums as $\T(abc)^{(abcp)}$-modules, so we get a similar decomposition after localizing at $\n'$.  Since $\iota_u$ is injective, taking ranks of both sides gives the desired formula.
\end{proof}

Let $\mu(N)$ be the Möbius function, defined by
\[
\mu(N) =\begin{cases}
    (-1)^{\omega(N)} & \text{if }N\text{ is squarefree} \\
    0 & \text{otherwise},
\end{cases}
\]
where $\omega(N)$ is the number of primes dividing $N$.

\begin{corollary}
\label{cor:eis_gen}
Let $\n' \subset \T(Np^2)^{(Np)}$ be a maximal ideal such that $\B_{\n'}$ is principal. Then
\begin{equation}
\label{eq:new eis count}
\frac{2}{|\Delta|-1} \rank S_k(Np^2)^{\new}_{\n'}
= \begin{cases}
\rank S_k(Np)^{N\dnew}_{\n'}  -\mu(N) \delta  & \text{if }   p \equiv 1 \pmod{\ell}\\
\rank S_k(Np)^{\new}_{\n'}  + \mu(N) \delta & \text{if }  p \equiv -1 \pmod{\ell}.
\end{cases}
\end{equation}
\end{corollary}

\begin{proof}
To simplify notation, let $C=\frac{2}{|\Delta|-1}$, and let $s_k(M)^{d\dnew}=\rank S_k(M)_{\n'}^{d\dnew}$. 
We proceed by induction on $\tau(N)$.  The base case follows from \cref{cor:newform count}. 
Using \cref{lem:Rob's variant} with $(a,b,c)=(N,1,p^2)$ gives
\[
s_k(Np^2)^{p\dnew} = \sum_{d \mid N} \tau(N/d) s_k(dp^2)^{\new},
\]
which rearranges to
\begin{equation}
    \label{eq:sk(Np2)}
    s_k(Np^2)^{\new}= s_k(Np^2)^{p\dnew} - \sum_{d \mid N, d\ne N} \tau(N/d) s_k(dp^2)^{\new}.
\end{equation}

Assume that $p \equiv 1 \pmod{\ell}$. Starting with \eqref{eq:sk(Np2)} and substituting for $s_k(Np^2)^{p\dnew}$ using \cref{cor:newform count} and for $s_k(dp^2)^{\new}$ using the induction hypothesis, and multiplying both sides by $C$, yields 
\begin{align*}
  C \cdot {s_k(Np^2)^{\new}} &= s_k(Np)-\delta  - \sum_{d \mid N, d\ne N} \tau(N/d)\left(s_k(dp)^{d\dnew}-\mu(d)\delta\right) \\
  &=\sum_{d \mid N} \tau(N/d)s_k(dp)^{d\dnew} - \delta -\sum_{d \mid N, d\ne N} \tau(N/d)\left(s_k(dp)^{d\dnew}-\mu(d)\delta\right) \\
  &=s_k(Np)^{N\dnew}-\delta+\sum_{d \mid N, d\ne N} \tau(N/d)\mu(d)\delta \\
  &=s_k(Np)^{N\dnew}-\mu(N)\delta.
\end{align*}
The second line uses \cref{lem:Rob's variant} with $(a,b,c)=(N,p,1)$, and the last line uses the M\"obius inversion formula $\sum_{d|N} \tau(N/d)\mu(d)=1$.

The proof for $p \equiv -1 \pmod{\ell}$ is very similar. From \eqref{eq:sk(Np2)}, using \cref{cor:newform count} and the induction hypothesis gives
\begin{align*}
    C \cdot {s_k(Np^2)^{\new}} &= s_k(Np)^{p\dnew}+\delta  - \sum_{d \mid N, d\ne N} \tau(N/d)\left(s_k(dp)^{\new}+\mu(d)\delta\right) \\
    &=\sum_{d \mid N} \tau(N/d) s_k(dp)^{\new} + \delta  - \sum_{d \mid N, d\ne N} \tau(N/d)\left(s_k(dp)^{\new}+\mu(d)\delta\right) \\
    &=s_k(Np)^{\new} + \delta - \sum_{d \mid N, d\ne N} \tau(N/d)\mu(d)\delta \\
    &=s_k(Np)^{\new} + \mu(N)\delta,
\end{align*}
where the second line uses \cref{lem:Rob's variant} with $(a,b,c)=(N,1,p)$.
\end{proof}

\subsection{More on automorphically-trivial Eisenstein congruences}\label{subsec:aut triv Eis congs}
Now assume that $\delta=1$. To illustrate how the dependence on $\delta$ in \cref{cor:newform count} can be used to prove results on the existence of congruences, we first
specialize this discussion to the most classical case $N=1$. If $p \equiv 1 \pmod{\ell}$, then \cref{cor:newform count} implies that
\[
\rank S_2(p^2)^{\new}_{\m'} = \frac{|\Delta|-1}{2}(\rank S_2(p)^{\new}_{\m'}-1),
\]
which is \cref{thm:Eisp=1modell} of the introduction. In particular, it recovers the famous result of Mazur \cite{mazur} that $S_2(p)^{\new}_{\m'} \ne 0$.

If $p \equiv -1 \pmod{\ell}$, then one knows that $S_2(p)^{\new}_{\m'}=0$ (for instance, by examining the constant term of the Eisenstein series). Using this, \cref{cor:newform count} implies
\[
\rank S_2(p^2)^{\new}_{\m'} = \frac{|\Delta|-1}{2}.
\]
In addition to this numerical result, we can prove a structure result. 
\cref{thm:modular forms count} applied in this setting shows 
\[
M_2(p^2)_{\m} \cong \cO[\Delta]^+.
\]
Indeed, in the notation of that theorem,
$a_0=\rank M_2(1)_{\m'}=0$ and $a_1=\rank M_2(p)_{\m'} =1$ (there is only the Eisenstein series).
By duality, this implies that $\T(p^2)_\m \cong \cO[\Delta]^+$, which recovers \cref{thm:LangWake} of the introduction.

More generally, using the fact that ranks are nonnegative and the term $\mu(N)\delta$ may be negative, we obtain the following.

\begin{corollary}
    \label{cor:nonzero eisen spaces}  Assume that $\mu(N)\delta \ne 0$; in other words, assume that $N$ is squarefree, that $k=2$, and that $\n'$ is automorphically trivial. \hfill
\begin{enumerate}[(i)]
    \item If $p \equiv 1 \pmod{\ell}$, then $S_2(Np)_{\n'}^{N\dnew}$ is nonzero if $\mu(N)=1$. Moreover, $S_2(Np^2)_{\n'}^{\new}$ is nonzero if and only if one of the following is true
    \begin{itemize}
        \item  $\mu(N)=-1$,
        \item  $\mu(N)=1$, and $\rank S_2(Np)_{\n'}^{N\dnew} >1$.
    \end{itemize}
    \item If $p \equiv -1 \pmod{\ell}$, then $S_2(Np)_{\n'}^{\new}$ is nonzero if $\mu(N)=-1$. Moreover, $S_2(Np^2)_{\n'}^{\new}$ is nonzero if and only if one of the following is true
    \begin{itemize}
        \item $\mu(N)=1$,
        \item $\mu(N)=-1$, and $\rank S_2(Np)_{\n'}^{\new} >1$.
    \end{itemize}
\end{enumerate}
In particular, in every case at least one of the spaces $S_2(Np^2)_{\n'}^{\new}$ or $S_2(Np)_{\n'}^{N\dnew}$ is nonzero.
\end{corollary}
\begin{proof}
Assume that $p \equiv 1 \pmod{\ell}$. By \cref{cor:eis_gen},
\[
\frac{2}{|\Delta|-1}\rank S_2(Np^2)_{\n'}^{\new} = \rank S_2(Np)_{\n'}^{N\dnew} - \mu(N).
\]
If $\mu(N)=-1$, the right-hand side is positive, so $\rank S_2(Np^2)_{\n'}^{\new}>0$. If $\mu(N)=1$, then
\[
\frac{2}{|\Delta|-1}\rank S_2(Np^2)_{\n'}^{\new} = \rank S_2(Np)_{\n'}^{N\dnew} -1,
\]
so $\rank S_2(Np)_{\n'}^{N\dnew} \ge 1$, with equality if and only if $\rank S_2(Np^2)_{\n'}^{\new} =0$. The proof for $p \equiv -1 \pmod{\ell}$ is similar.
\end{proof}

We use the dependence on $\delta$ in \cref{cor:eis_gen} along with Hida theory to deduce the existence of congruences with Eisenstein series in weights greater than 2.  First a lemma from Hida theory.

\begin{lemma}
\label{lemma:Hida}
For $k \equiv k' \pmod{\ell-1}$ and integers $M$ and $d$ such that $\ell \nmid M$ and $d \mid M$ with $\gcd(d, M/d) = 1$, we have
$$
   S_{k'}(M)^{d \dnew}_{\n'} \ne 0 \implies S_k(M\ell)^{d \dnew}_{\n'} \ne 0.
$$
If $k>2$, we have
$$
   S_{k'}(M)^{d \dnew}_{\n'} \ne 0 \implies S_k(M)^{d \dnew}_{\n'} \ne 0.
$$
\end{lemma}

\begin{proof}
As $S_{k'}(M)^{d \dnew}_{\n'} \ne 0$, there is some $M'$ with $d|M'|M$ such that $S_{k'}(M')^{\new}_{\n'} \ne 0$.
As $T_\ell - \ell -1 \in \n'$, all forms in this space are $\ell$-ordinary, and thus, by Hida theory \cite[Corollary 3.7]{Hida-Congruence_modules}, we have $S_{k}(M'\ell)^{M'\dnew}_{\n'} \ne 0$, and hence $S_k(M\ell)^{d \dnew}_{\n'} \ne 0$. If $k>2$, then there are no $\ell$-new ordinary eigenforms (see \cite[Proposition 5]{gouvea}, for instance), so $S_k(M\ell)^{d \dnew}_{\n'} \ne 0$ implies $S_k(M)^{d \dnew}_{\n'} \ne 0$.
\end{proof}

\begin{corollary}
\label{cor:Hida consequences}
Assume that $k \equiv 2 \pmod{\ell-1}$, that $N$ is squarefree, and that $T_q-q-1 \in \n'$ for all $q \nmid Np$. 
Then the spaces $S_k(Np^2\ell)^{Np^2\dnew}_{\n'}$ and $S_k(Np\ell)^{N\dnew}_{\n'}$ are nonzero. If $k>2$, then $S_k(Np^2)^{\new}_{\n'}$ and $S_k(Np)^{N\dnew}_{\n'}$ are nonzero. 

If $p \equiv -1 \pmod{\ell}$, we additionally have that $S_2(Np\ell)^{Np\dnew}_{\n'}$ and, if $k>2$, $S_k(Np)^{\new}_{\n'}$  are nonzero.
\end{corollary}
\begin{proof}
This follows from \cref{cor:nonzero eisen spaces,cor:eis_gen} together with \cref{lemma:Hida}. We illustrate the argument in the case $p \equiv 1 \pmod{\ell}$, with the $p \equiv -1 \pmod{\ell}$ case being similar.

If $k=2$ let $k'=\ell+1$; otherwise let $k'=k$. Assume $p \equiv 1 \pmod{\ell}$ and first assume that $\mu(N)=1$. By \cref{cor:nonzero eisen spaces}, we have $S_2(Np)_{\n'}^{N\dnew} \ne0$, and \cref{lemma:Hida} implies $S_{k'}(Np)_{\n'}^{N\dnew} \ne 0$ and $S_{k}(Np\ell)_{\n'}^{N\dnew} \ne 0$. By \cref{cor:eis_gen}, we have
 (noting that $\delta=0$ because $k'>2$)
 \[
\frac{2}{|\Delta|-1}\rank S_{k'}(Np^2)_{\n'}^{\new} =\rank S_{k'}(Np)_{\n'}^{N\dnew},
\]
so $S_{k'}(Np^2)_{\n'}^{\new} \ne 0$. Then \cref{lemma:Hida} implies $S_{k}(Np^2\ell)_{\n'}^{Np^2\dnew} \ne 0$. Now assume $\mu(N)=-1$. Then \cref{cor:nonzero eisen spaces} implies that $S_2(Np^2)_{\n'}^{\new} \ne 0$ and \cref{lemma:Hida} implies $S_{k'}(Np^2)_{\n'}^{\new} \ne 0$. Then the same equation from \cref{cor:eis_gen} implies $S_{k'}(Np)_{\n'}^{N\dnew} \ne 0$. 
\end{proof}

\begin{remark}
Let $N$ be squarefree, and let $\epsilon \in \{1,-1\}$ be such that $p \equiv \epsilon \pmod{\ell}$. Then \cref{cor:nonzero eisen spaces} implies that $S_2(Np)_{\n'}^{N\dnew} \ne 0$ if $\epsilon \mu(N)=1$. If $\epsilon \mu(N)=-1$, then we call this an \emph{ambiguous case}, and the question of whether $S_2(Np)_{\n'}^{N\dnew}$ is nonzero involves more subtle arithmetic.

For example, if $N=q$ is a prime number, then $S_2(pq)_{\n'}^{q\dnew}$ is nonzero except possibly in the case $p \equiv 1 \pmod{\ell}$ and $q \not \equiv \pm 1 \pmod{\ell}$. In this case, it is a theorem of Ribet \cite[Theorem 2.3]{yoo} that $S_2(pq)_{\n'}^{q\dnew}$ is nonzero if and only if $\log_p(q) \equiv 0 \pmod \ell$, where $\log_p:(\Z/p\Z)^\times \to \Z/(p-1)\Z$ is a discrete logarithm. Our results imply that $S_2(pqv)_{\n'}^{q\dnew} \ne 0$ for every prime $v \ne q$ (including $v \in \{\ell,p\}$), and it would be interesting to see if one could give a new proof of Ribet's result using this (for instance, by showing that $\log_p(q) \pmod{\ell}$ is an obstruction to lowering the level from $pq\ell$ to $pq$).
\end{remark} 

\section{Galois actions and the local Langlands correspondence}
\label{sec:Galois}

\Cref{thm:modular forms count} and \Cref{prop:action on eigenforms} describe an action of a subring of the group ring of a torus on the space $M_k(\Gamma \cap \Gamma_0(p^2),\cO)_\m$. In this section, we reinterpret that group-ring action in terms of Galois deformation rings. 

We continue with the notation as in \cref{sec:hecke}. In particular, we fix a maximal ideal $\m \subset \T = \T(\Gamma \cap \Gamma_0(p^2))$ in the Hecke algebra containing $U_p$. For an eigenform $f \in M_k(\Gamma \cap \Gamma_0(p^2),\cO)_{\m}$, let $\rho_f: G_\Q\to \GL_2(\cO)$ be the associated Galois representation. The semisimplification of the residual representation $\rho_f \pmod{\varpi}$ depends only on $\m$ and is denoted $\rhobar_{\m}$.

\subsection{The correspondence between blocks and residual inertial representations}\label{subsec:LLC} 
Using the local Langlands correspondence and local-global compatibility, we explain how the block $\B_\m$ is determined by the behavior of $\rhobar_{\m}$ at the inertia group $I_p$ at $p$.

We first recall the local Langlands correspondence for $\GL_2$.  Let $G_p = \Gal(\overline{\Q}_p/\Q_p)$.  We write $\Q_{p^2}$ for the unique unramified quadratic extension of $\Q_p$ and $G_{p^2}$ for its absolute Galois group. The local Langlands correspondence  associates a two-dimensional representation $\rho(\pi): G_p \to \GL_2(\bar\Q_\ell)$ of $G_p$ (or really, of the Weil group) with a smooth irreducible $\ell$-adic representation $\pi$ of $\GL_2(\Q_p)$. We recall $\rho(\pi)$ for the $\pi$ considered in \Cref{subsec:eigenforms}:
\begin{itemize}
    \item If $\pi=I(\tilde{\chi})$ for $\tilde\chi=(\chi_1,\chi_2)$, where $\chi_i:\Q_p^\times \to F^\times$, then $\rho(\pi)=\chi_1 \oplus \chi_2$.
    \item If $\pi=\sigma(\tilde{\chi})$ for $\tilde\chi=(\chi_1,\chi_1)$, then $\rho(\pi)$ is a nontrivial extension of $\chi_1$ by~$\chi_1\epsilon_\ell$, where $\epsilon_\ell$ denotes the $\ell$-adic cyclotomic character.
    \item If $\pi=\pi_{\tilde\theta}$, then $\rho(\pi)=\Ind_{G_{p^2}}^{G_p} \theta'$, where $\theta'$ is an unramified quadratic twist of $\tilde\theta$.
\end{itemize}
Here, we have identified characters of $\Q_p^\times$ or $\Q_{p^2}^\times$ with Galois characters via local class field theory.

There is a local-global compatibility, in the following sense. If $f$ is an eigenform with associated automorphic representation $\pi_f$, and $\rho_f$ is the associated $\ell$-adic Galois representation, then $\rho_f|_p = \rho(\pi_{f,p})$ \cite[Th\'eor\`eme A]{carayol86}.

\begin{lemma}
\label{lem:block and rhobar}
    If $\B_\m$ is principal, then $\rhobar_\m|_{I_p}$ is unipotent. If $\B_\m$ is nilpotent, then $\rhobar_\m|_{I_p}$ is a sum of two nontrivial characters.
\end{lemma}
\begin{proof}
    This is immediate from \cref{prop:block of m}, local-global compatibility, and the above description of $\rho(\pi_{f,p})$.
\end{proof}

\subsection{Inertia-type pseudodeformation rings}
For each of our blocks $\B_\m$, we now define a Galois deformation ring $R_p$ of $\rhobar_\m|_{I_p}$ and show that it is isomorphic to $\cO[\Delta]^+$ if $\B_\m$ is principal and $\cO[\Delta]$ if $\B_\m$ is nilpotent. Since $\rhobar_\m|_{I_p}$ is reducible by \cref{lem:block and rhobar}, it is convenient to work with pseudorepresentations instead of representations. Let $\bar\phi:I_p \to \F^\times$ be a character such that the semisimplification of $\rhobar_\m|_{I_p}$ is $\bar\phi \oplus \bar\phi^{-1}$, and let $\bar D_p=\bar\phi + \bar\phi^{-1}$ be the pseudorepresentation of $I_p$ associated with $\rhobar|_{I_p}$. Informally, $R_p$ is the universal ring of deformations of $\bar D_p$ that extend to pseudorepresentations of $G_p$.

Let $I_p^{(\ell)}$ denote the maximal pro-$\ell$ quotient of $I_p$, which is isomorphic to $\Z_\ell$.  Choose $\tau \in I_p$ such that its image in $I_p^{(\ell)}$ is a topological generator. Let $r=\mathrm{val}_\ell(p^2-1)$, and let $\nabla=\Gal(\Q_p^\mathrm{ur}(\sqrt[\ell^r]{p})/\Q_p^\mathrm{ur})$.  This is a cyclic group of order $\ell^r$, and the choice of $\tau$ determines a generator $\delta_\tau$ of $\nabla$. Under the Artin map, $\nabla$ is identified with the $\ell$-part of $\F_{p^2}^\times$, which is identified with $\Delta$.  Let $\alpha \in \cO$ denote the canonical (multiplicative) lift of $\bar\phi(\tau)$.

Let $\tilde R_p$ be the universal deformation ring of $\bar D_p$ over $\cO$ with trivial determinant, and let $\tilde D_p$ be the universal deformation. Since $\tau$ and $\tau^p$ are conjugate in $G_p$, if a deformation $D$ of $\bar D_p$ extends to a representation of $G_p$, then it must satisfy
\[
\tr D(\tau)=\tr D(\tau^p).
\]
Motivated by this, we define 
\[
R_p = \tilde R_p /(\tr \tilde D_p (\tau)-\tr \tilde D_p (\tau^p)),
\]
and let $D_p$ be the associated deformation.
One can think of $R_p$ as a pseudo-analog of the image of an inertia deformation ring in a local deformation ring.

\begin{lemma}
\label{lem:R_p}
If $\B_\m$ is principal, then there is an isomorphism $R_p \cong \cO[\nabla]^+$, where the deformation $D_p$ satisfies $\tr D_p(\tau) = [\delta_\tau] + [\delta_\tau^{-1}]$. 

If $\B_\m$ is nilpotent, then there is an isomorphism $R_p \cong \cO[\nabla]$, where the deformation $D_p$ satisfies $\tr D_p(\tau) = [\delta_\tau]\alpha+ [\delta_\tau^{-1}]\alpha^{-1}$.
\end{lemma}
\begin{proof}
When $\B_\m$ is principal and $p \equiv -1 \pmod{\ell}$, then this is \cite[Proposition 4.4]{LangWake2}. The proof is very similar in the remaining cases, so we only sketch the arguments; details that are left out can be filled in exactly as in \cite{LangWake2}. 

First suppose that $\B_\m$ is principal. By continuity, $\tilde D_p$ is determined by $\tr \tilde D_p(\tau)$, so there is an isomorphism $\cO \llbracket x \rrbracket \isoto \tilde R_p$ sending $x$ to $\tr\tilde D_p(\tau)-2$. To determine the image of $\tr \tilde D_p(\tau) - \tr \tilde D_p(\tau^p)$, we work over the extension ring 
\[
\frac{\cO \llbracket x \rrbracket[\lambda]}{(\lambda^2-(x+2)\lambda+1)},
\]
where the roots of the characteristic polynomial of $\tilde D_p(\tau)$ are $\lambda$ and $\lambda^{-1}$. In this ring, $\tr \tilde D_p(\tau) - \tr \tilde D_p(\tau^p)$ maps to
\[
\lambda+\lambda^{-1} -\lambda^p -\lambda^{-p}= -\lambda^{-p}(\lambda^{p-1}-1)(\lambda^{p+1}-1)
\]
Up to a unit, this can be shown to be $(\lambda-1)(\lambda^{\ell^r}-1)$. Let
\[
A = \frac{\cO \llbracket x \rrbracket[\lambda]}{(\lambda^2-(x+2)\lambda+1,(\lambda-1)(\lambda^{\ell^r}-1))},
\]
so that $R_p$ is identified with the subalgebra of $A$ generated by $x$. There is a surjective map $A \onto \cO[\nabla]$ given by
\[
x \mapsto [\delta_\tau] + [\delta_\tau^{-1}]-2, \ \lambda \mapsto [\delta_\tau],
\]
and one can show that the kernel of this map has trivial intersection with $R_p$, so this induces the desired isomorphism $R_p \isoto \cO[\nabla]^+$.

Now suppose that $\B_\m$
is a nilpotent block. In this case, by Hensel's lemma, the polynomial $X^2-\tr \tilde{D}_p(\tau)X+1$ in $\tilde{R}_p[X]$ has two roots; let $\lambda \in \tilde R_p$ be the one that is congruent to $\alpha$ modulo the maximal ideal. Then $\tr\tilde D_p(\tau)=\lambda+\lambda^{-1}$ and $\tilde D_p$ is determined by $\lambda$, so there is an isomorphism $\tilde{R}_p \isoto \cO \llbracket x \rrbracket$ sending $\lambda$ to $\lambda(x)=\alpha \cdot (1+x)$. Then, just as in the principal case, $\tr \tilde D_p(\tau) - \tr \tilde D_p(\tau^p)$ maps to
\[
-\lambda(x)^{-p}(\lambda(x)^{p-1}-1)(\lambda(x)^{p+1}-1),
\]
which equals, up to a unit power series, $\lambda(x)^{\ell^r}-1=(1+x)^{\ell^r}-1$. (Note that, in this case, $\lambda(x)-1$ is a unit power series.) Hence, the desired isomorphism is the composition
\[
R_p \cong \frac{\cO \llbracket x \rrbracket}{((1+x)^{\ell^r}-1)} \cong \cO[\nabla],
\]
where $x$ maps to $[\delta_\tau]-1$. 
\end{proof}

One should think of $\cO[\nabla]$ as the deformation ring of $\bar\phi$. Then \cref{lem:R_p} says that $R_p$ is the same as the deformation ring of $\bar\phi$, except that a deformation and its inverse give the same deformation of $\rhobar_\m$ when $\bar\phi$ is trivial.

\subsection{Action of \texorpdfstring{$R_p$}{} on modular forms}
The deformation ring $R_p$ acts on spaces of modular forms through a global deformation ring. Let $\bar D$ be the pseudorepresentation of $\rhobar_\m$.  Let $R$ be the universal deformation ring of $\bar D$ with fixed determinant $\chi$ that is unramified at $p$, and let $D$ be the universal deformation. 
(If $\bar \rho$ is absolutely irreducible, then $R$ is isomorphic to the deformation ring of $\rhobar$, and $D$ is the pseudorepresentation of the universal representation.) Restricting $D$ to $I_p$ gives a deformation of $\bar D_p$ with trivial determinant, so this defines a map $\tilde R_p \to R$. Since $\tau$ and $\tau^p$ are conjugate in $G_\Q$, this map factors through a map $R_p \to R$.

There is a deformation $D_\T$ of $\bar D$ with values in $\T$ characterized by
\[
\tr D_\T(\Frob_q)= T_q
\]
for all but finitely many primes $q$; it can be obtained by gluing the pseudorepresentations associated with eigenforms. This defines a map
$R \to \T$
and hence a map $R_p \to \T$ and an action of $R_p$ on modular forms. Since, by \Cref{lem:R_p}, $R_p$ is generated by $\tr D_p(\tau)$, this action is determined by the condition that, for an eigenform $f \in M_k(\Gamma(p)\cap \Gamma, \cO)_\m^Z$ with Galois representation $\rho_f$,
\begin{equation}
\label{eq:galois action}
    \tr D_p(\tau) \cdot f = \tr(\rho_f(\tau)) f.
\end{equation}
\subsection{Comparison of Galois and automorphic actions} Under class field theory, the ring $\cO[\nabla]$ is identified with $\cO[\Delta]$. In this way, in each case, the ring $R_p$ is identified with the ring appearing in \Cref{thm:modular forms count}. To distinguish the two actions, we refer to the action \eqref{eq:galois action} as the \emph{Galois} action and the action from \Cref{thm:modular forms count} as the \emph{automorphic} action.

\begin{proposition}
\label{prop:compatible actions}
    The Galois and automorphic actions of $\cO[\Delta]$ or $\cO[\Delta]^+$ on $M_k(\Gamma \cap \Gamma_0(p^2),\cO)_\m$ coincide.
\end{proposition}
\begin{proof}
    Since $M_k(\Gamma \cap \Gamma_0(p^2),\cO)_\m$ is $\ell$-torsion free and $M_k(\Gamma \cap \Gamma_0(p^2),\cO)_\m \otimes_\cO F$ has a basis of eigenforms, it suffices to show that the actions coincide on eigenforms. By \Cref{prop:action on eigenforms}, the automorphic action is determined by the local type $\pi_{f,p}$ in the case of cusp forms and the character $\chi_f$ for Eisenstein series.  The Galois action is determined by the local Galois representation $\rho_f|_{I_p}$ by \eqref{eq:galois action}. One readily checks that the two actions correspond under the local Langlands correspondence, so the result follows from local-global compatibility \cite[Th\'eor\`eme (A)]{carayol86}.
\end{proof}

\subsection{The vexing case}\label{subsec:vexing}
By \cref{prop:compatible actions}, we can interpret \cref{thm:modular forms count} and all the results deduced from it in terms of Galois actions rather than automorphic actions. To illustrate this, we explain how to deduce \cref{thm:intro vexing}, which is phrased in terms of the Galois action.

\begin{proof}[Proof of \cref{thm:intro vexing}]
As recalled in \cref{subsec:intro vexing}, the assumption that $\rhobar$ is vexing implies that $\rhobar|_{I_p}$ is a sum of two nontrivial characters, so \cref{lem:block and rhobar} implies that $\B_\m$ is nilpotent. By \cref{thm:modular forms count}\eqref{item:vexing}, $M_k(\Gamma \cap \Gamma_0(p^2),\cO)_\m$ is a free $\cO[\Delta]$-module for the automorphic action. By \cref{prop:compatible actions}, this implies $M_k(\Gamma \cap \Gamma_0(p^2),\cO)_\m$ is a free $\cO[\Delta]$-module for the Galois action, which is the action described in the theorem statement. 
\end{proof}

\begin{remark}
Note that if $f \in S_k(\Gamma_0(p^2))$ with $k \geq 2$ has complex multiplication (necessarily by $\Q(\sqrt{-p})$), then for every prime $\ell \geq 5$ such that $p \equiv -1 \pmod \ell$, the mod-$\ell$ Galois representation of $f$ is vexing at $p$.  For $p \geq 5$, such forms exist if and only if $p \equiv 3 \pmod 4$; in fact, the number of them is equal to the class number of $\Q(\sqrt{-p})$ \cite{Gross}.  Therefore \cref{thm:intro vexing} provides quantitative information about the number of CM-nonCM congruences in $S_k(\Gamma_0(p^2))$.
\end{remark}

\bibliography{modular_bib}
\bibliographystyle{alpha}
\end{document}